%% file: Arxiv.tex
\documentclass[11pt, reqno, twoside, makeidx]{amsart}
\usepackage[margin=2.7cm, marginpar=1cm]{geometry}
\usepackage{wrapfig}
\usepackage{marginnote}
\usepackage[english]{babel}
\usepackage[T1]{fontenc}
\usepackage{amsmath,amsthm,amssymb,amsfonts}
\numberwithin{equation}{section}
\usepackage{enumerate}
\usepackage{enumitem}
\usepackage{faktor}
\usepackage{dsfont}
\usepackage{scalerel,stackengine}
\usepackage{textcomp}
\usepackage{graphicx}
\usepackage{extarrows}
\usepackage{nicematrix}
\usepackage{epigraph}
\usepackage{sidecap}
\usepackage{indentfirst}
\usepackage{tikz}
\usepackage{tikz-cd}
\usepackage{tkz-euclide}
\usetikzlibrary{shapes.misc,arrows,matrix,patterns,decorations.markings,positioning}
\tikzset{cross/.style={cross out, draw=black, minimum size=2*(#1-\pgflinewidth), inner sep=0pt, outer sep=0pt},
cross/.default={4.5pt}}
\usepackage{pgfplots}
\usepackage{caption}
\usepackage{subcaption}
\usepackage{float}
\usepackage{color}
\usepackage{xcolor}
\definecolor{burgundy}{rgb}{0.5,0.0,0.13}
\usepackage{epstopdf}
\usepackage{epsfig}
\usepackage{stmaryrd}
\usepackage{multicol}
\usepackage{verbatim}
\usepackage{hyperref}

\DeclareMathOperator{\Ker}{Ker }
\DeclareMathOperator{\Imm}{Im }
\DeclareMathOperator{\tb}{tb}
\DeclareMathOperator{\rot}{rot}
\DeclareMathOperator{\self}{sl}

\DeclareMathOperator{\Spin}{Spin}

\DeclareMathOperator{\Id}{Id}

\DeclareMathOperator{\height}{height_{\mathcal F}}
\DeclareMathOperator{\xist}{\xi_{std}}
\DeclareMathOperator{\lcm}{lcm}
\renewcommand{\geq}{\geqslant}
\renewcommand{\leq}{\leqslant} 
\renewcommand{\epsilon}{\varepsilon}

\newcommand{\Z}{\mathbb{Z}}
\newcommand{\Q}{\mathbb{Q}}
\newcommand{\F}{\mathbb{F}}

\newcommand{\J}{\mathcal J}
\newcommand{\dd}{\mathrm{d}}

\newcommand{\s}{\mathfrak{s}}

\newtheorem{teo}{Theorem}[section]
\newtheorem*{teo*}{Theorem}
\newtheorem{lemma}[teo]{Lemma}
\newtheorem{prop}[teo]{Proposition}
\newtheorem*{prop*}{Proposition}
\newtheorem{defin}[teo]{Definition}
\newtheorem{cor}[teo]{Corollary}

\newtheorem{remark}[teo]{Remark}

\newcommand{\joindots}{%
  \tikz[baseline=-1.1ex,x=1ex,y=1ex]{
    \draw[line width=0.08ex]
      (0,0) -- (-1,0)   
      (0,0) -- (0,-4)     
      (0,-4) -- (-1,-4) 
      (0,-2) -- (1,-2); 
    \fill (1.35,-2) circle[radius=0.4]; 
    \fill (-1.35,0) circle[radius=0.4];
    \fill (-1.35,-4) circle[radius=0.4];
  }%
}
\newcommand{\joindotss}{%
  \tikz[baseline=-1.1ex,x=1ex,y=1ex]{
    \draw[line width=0.08ex]
      (0,0) -- (-1,0)   
      (0,0) -- (0,-8)     
      (0,-4) -- (-1,-4) 
      (0,-8) -- (-1,-8)
      (0,-2) -- (1,-2) 
      (0,-6) -- (1,-6)
      (2,-2) -- (3,-2)
      (2,-6) -- (3,-6)
      (3,-2) -- (3,-6)
      (3,-4) -- (4,-4);
    \fill (1.35,-2) circle[radius=0.4]; 
    \fill (-1.35,0) circle[radius=0.4]; \fill (1.35,-6) circle[radius=0.4];
    \fill (-1.35,-4) circle[radius=0.4]; \fill (-1.35,-8) circle[radius=0.4];
    \fill (4,-4) circle[radius=0.4];
  }%
}

\usepackage{xpatch}
\makeatletter
\xpatchcmd{\@thm}{\thm@headpunct{.}}{\thm@headpunct{}}{}{}
\makeatother

\pgfplotsset{compat=1.18}
\begin{document}
\title[Heegaard Floer homology and maximal twisting numbers]{Heegaard Floer homology and maximal twisting numbers}
\author{Alberto Cavallo and Irena Matkovi\v c}
\address{HUN-REN Alfr\'ed R\'enyi Institute of Mathematics, Budapest 1053, Hungary}
\email{acavallo@impan.pl}
\address{Uppsala Universitet, Uppsala 751 06, Sweden}
\email{irma6504@student.uu.se}


\begin{abstract}
We adapt the Ozsv\'ath-Szab\'o full path algorithm to every star-shaped graph and establish a correspondence between negative-twisting tight contact structures on any Seifert fibred space over $S^2$, and its Heegaard Floer homology groups equipped with the Alexander filtration induced by the regular fibre. This provides the complete classification of negative-twisting structures on these manifolds; in particular, we distinguish them by their contact invariant $c^+$. We prove that every such structure is symplectically fillable and extend a known obstruction to Stein fillability. In addition, we show that the number of negative-twisting structures can be expressed combinatorially in terms of the Seifert coefficients of the star-shaped graph, while their $d_3$-invariant and homotopy type are determined explicitly through our correspondence. Our results also complete the classification of fillable structures on any small Seifert fibred space. 
\end{abstract}

\maketitle

\thispagestyle{empty}

\section{Introduction}
In \cite{OSz-fullpath} Ozsv\'ath and Szab\'o introduced the full path algorithm as a way to compute the Heegaard Floer group $HF^-(Y_\Gamma)$ combinatorially, when $Y_\Gamma$ is the 3-manifold presented by a suitable plumbing tree $\Gamma$. Namely, under the assumption that $\Gamma$ is negative-definite, and that there is at most one "bad vertex", they showed that every homogeneous non-torsion class $\alpha\in HF^-(Y_\Gamma)$ has even (Maslov) grading and can be written (up to multiplication by $U$) as $F^-_{P_\Gamma,\mathfrak u}(1)$ where $F^-_{P_\Gamma,\mathfrak u}:HF^-(S^3)\rightarrow HF^-(Y_\Gamma,\s)$ is the corresponding cobordism map.

Ozsv\'ath and Szab\'o defined the graded maps $F^\circ$ induced by 4-dimensional $\Spin^c$-cobordisms in \cite{OSz-negative}, and they immediately appeared as a tool of primary importance in low dimensional topology considering the number of their applications. In the full path setting, the 4-manifold is the plumbing $P_\Gamma$ given by the graph, while the $\Spin^c$-structure can be identified with the characteristic vector $V\in\text{Char}(\Gamma)$ corresponding to its first Chern class $c_1(\mathfrak u)\in H^2(P_\Gamma;\Z)$, as the manifold is simply connected, and $\s=\mathfrak u\lvert_{Y_\Gamma}$. The true value of the full path algorithm is that it gives a practical method to not only identify $HF^-(Y_\Gamma)$ up to isomorphism, but also to determine precisely when two characteristic vectors $V_1$ and $V_2$ satisfy $F^-_{P_\Gamma,\mathfrak u_1}(1)=F^-_{P_\Gamma,\mathfrak u_2}(1)$; in practical terms, the full path  defines an equivalence relation on the set of characteristic vectors, whose quotient set is precisely $HF^-(Y_\Gamma)$. Ozsv\'ath and Szab\'o showed in \cite[Theorem 1.2]{OSz-fullpath} that this correspondence holds also when we consider Heegaard Floer homology with integral coefficients, but in this paper we work with $\F$ the field with two elements.

Later on, the full path has been extensively studied by N\'emethi \cite{Nemethi}. He introduced the family of almost-rational graphs, whose precise definition is given in Section \ref{section:two}, generalising the one in \cite{OSz-fullpath}; N\'emethi proved that the full path, called computational sequence in his terminology, still recovers $HF^-(Y_\Gamma)$. The work of N\'emethi found many applications in algebraic geometry; more specifically, in the study of singularities, see \cite{Nemethi}. An example is given by the characterisation of rational singularities as the ones whose resolution is a Heegaard Floer $L$-space, a property which is read from the Heegaard Floer groups and is, in this case, combinatorial thanks to the full path.

N\'emethi \cite{Nemethi} used the full path as the starting point to define lattice cohomology: a purely combinatorial invariant of a 3-manifold, defined from any plumbing tree, which has recently been shown by Zemke \cite{Zemke} to coincide with Heegaard Floer homology. 

It is important to note that the isomorphism given in the proof is not just an extension of the one in \cite[Theorem 1.2]{OSz-fullpath}, but it required a general version of the link surgery formula in Heegaard Floer. The fact that, once the assumption of the graph being almost-rational is dropped, the equivalence classes given by the full path cannot span the whole $HF^-$ can already be seen from the simple example in Figure \ref{235}. The reason is that any graph representing $-\Sigma(2,3,5)$ gives a plumbing with positive $b_2^+$, while the manifold is an $L$-space (it is the only such Brieskorn sphere); Heegaard Floer cobordism maps have a restricted behaviour when the second positive Betti number is non-zero: the image of $F^-$ is necessarily a torsion class \cite{OSz-negative}. Note that by definition $L$-spaces have zero torsion in $HF^-$.

\begin{figure}[t] 
    \begin{tikzpicture}[scale=0.8]
    \tkzDefPoints{0/0/A, 1.5/1/B, 1.5/0/D, 1.5/-1/C}  
    \tkzDrawSegment(A,B)\tkzDrawSegment(A,D)\tkzDrawSegment(B,A)\tkzDrawSegment(A,C)
    \tkzDrawPoints[fill,black,size=5](A,B,C,D)
     \tkzLabelPoint[below left](A){$-1$} 
     \tkzLabelPoint[above right](B){$-2$}\tkzLabelPoint[below right](C){$-5$}
     \tkzLabelPoint[right](D){$-3$}
       \end{tikzpicture}
     \caption{\smaller[1]{The manifold $-\Sigma(2,3,5)$ which is an $L$-space, and thus its Heegaard Floer group $HF^-$ has vanishing $U$-torsion.}}
     \label{235}
\end{figure}
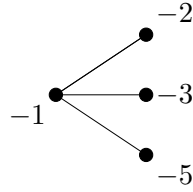

Our first result concerns precisely the homology classes in $HF^-(Y_G)$ which can be identified with full paths, in the case that $G$ is a star-shaped graph. This kind of tree in low dimensional topology corresponds exactly to Seifert fibred spaces (with a sphere as base orbifold). In \cite{CM-negative} we studied the relation between the full path and contact topology when the graph is negative-definite: note that any Seifert fibred space can be oriented in this way. In this paper we focus on the other orientation, whose graph is \emph{indefinite} (with $b_2^+=1$), and the singular case where the manifold is actually a surface bundle over the circle, see Sections \ref{section:two} and \ref{section:four}.

\subsection*{The full path algorithm for a star-shaped graph} Let $G$ be a star-shaped graph, with negative framings $m(1),...,m(|G|)$, and write $Y_G$ and $P_G$ for the corresponding 3-manifold and plumbing. Let $V\in H^2(P_G;\Z)$ be a characteristic vector, we write $[V]$ for its full path, see \cite{OSz-fullpath,CM-negative} and Section \ref{section:two} for the definition. As in \cite[Lemma 2.3]{OSz-fullpath} we define the subspace \[\mathcal B_n=\{Z=(z_1,...,z_{|G|})\in\text{Char}(G)\:|\:|z_i|\leq-m(i)+2n\}\] for any integer $n\geq0$; we say that $[V]$ \emph{ends correctly} when $Z\in\mathcal B_0$ for each $Z\in[V]$. This is equivalent to the definition we use in \cite{CM-negative}. 

There is an involution $\mathcal J:HF^\circ(M,\s)\rightarrow HF^\circ(M,\overline\s)$ on the Heegaard Floer groups of any 3-manifold $M$, which commutes with the maps induced by diffeomorphisms and cobordisms \cite[Theorem 2.4]{OSz}; moreover, if $\xi$ is a contact structure on $M$ then $\J$ acts by conjugation on the contact invariant $c^\circ(\xi)\in HF^\circ(-M,\s_\xi)$, that is $\J c^\circ(\xi)=c^\circ(\overline\xi)\in HF^\circ(-M,\overline\s_\xi)$, see \cite{G-fillability}. In addition, when $L\subset M$ is an oriented link then $\J$ identifies the filtration $\mathcal F^L$ on $HF^\circ(M,\s)$ with $\mathcal F^{-L}$ on $HF^\circ(M,\overline\s)$, given by reversing the orientation of $L$, see \cite[Lemma 3.12]{OSz-multi}. Dai and Manolescu \cite{DM} showed that in the full path setting we have $\J[V]=[-V]$ for every characteristic vector $V\in H^2(P_G;\Z)$ when $G$ is an almost-rational graph. The operation that maps characteristic vectors to their negation can be extended to the lattice cohomology group, see \cite[Section 2.2]{Lattice2}.

The Heegaard Floer group $\widehat{HF}(M,\s)$ can be decomposed as \[\widehat{HF}(M,\s)=\widehat{HF}^{\text{ev}}(M,\s)\oplus\widehat{HF}^\text{od}(M,\s)\] according to the parity of the Maslov grading; namely, when $b_1(M)=0$ given $\s\in\Spin^c(M)$ a class $\alpha$ has even parity when $d(M,\s)-M(\alpha)$ is even, and it has odd parity otherwise. Using the commutativity of the cobordism maps in Heegaard Floer \cite{OSz-negative}, we have that $\widehat F_{P_G,\mathfrak u}:\widehat{HF}(S^3)\rightarrow\widehat{HF}(M,\s)$ coincides with $\psi_*\circ F^-_{P_G,\mathfrak u}$, where $\psi_*:HF^-_*(M)\rightarrow\widehat{HF}_*(M)$ is induced by setting $U=0$ in $CF^-(M)$. The results in \cite{OSz-fullpath,Nemethi} show that, for a negative-definite $G$, when $V=c_1(\mathfrak u)$ then $[V]$ ends correctly if and only if $\widehat F_{P_G,\mathfrak u}(1)$ is non-zero; moreover, the fact that these classes form a basis of $\widehat{HF}^\text{ev}(Y_G)$ is well-known to experts, see \cite[Proposition 2.2]{CM-negative} for a complete proof.
 
\begin{teo}
 \label{teo:main}
 Let $G$ be an indefinite star-shaped graph with negative framings, and let $Y_G=\partial P_G$ be the $3$-manifold presented by the graph. Then the homology classes $[V]=F^-_{P_G,\mathfrak u}(1)\in HF^-(Y_G,\s)$, ranging among each $V\in\emph{Char}(G,\s)$ such that $V=c_1(\mathfrak u)$ and $\s=\mathfrak u\lvert_{Y_G}$, span the subgroup $\emph{Tor}HF^-(Y_G,\s)$; furthermore, they satisfy the following properties:
 \begin{itemize}
     \item the full path $[V]$ corresponds to a non-zero class if and only if there exists an $n\geq0$ such that $Z\in\mathcal B_n$ for each $Z\in[V]$;
     \item the homology classes $[V]=\widehat F_{P_G,\mathfrak u}(1)\in\widehat{HF}(Y_G,\s)$ where $[V]$ ends correctly form a basis of the subgroup $\widehat{HF}^\emph{od}(Y_G,\s)$;
     \item the involution $\J$ coincides with the natural involution in lattice cohomology, and it acts as $\J[V]=[-V]$ on each full path.
 \end{itemize}
\end{teo}

In the case that $G$ is a \emph{singular} star-shaped graph (we mean when $b_2^-(P_G)=|G|-1$ and $b_2^+(P_G)=0$), the corresponding version of Theorem \ref{teo:main} has been proved by Rustamov in \cite[Theorem 1.2]{Rustamov}, we give some details in Section \ref{section:two}.

From standard Heegaard Floer theory the following canonical duality isomorphisms hold: \[\widehat{HF}_*(-M,\s)\simeq\widehat{HF}_{-*}(M,\s)^\bullet\hspace{0.5cm}\text{ and }\hspace{0.5cm}HF^+_*(-M,\s)\simeq HF^-_{-*}(M,\s)^\bullet\:;\] in particular, the second one holds in this form only when $\s\in\Spin^c(M)$ is a torsion $\Spin^c$-structure, see Section \ref{section:two} for more details. We know from Ozsv\'ath and Szab\'o's results, and our Theorem \ref{teo:main} above, that the homology classes given by full paths always have the same parity; hence, we can never obtain a basis of the whole $\widehat{HF}(M,\s)$ only with them (except when $M$ is an $L$-space). Such a basis can be instead obtained using the duality isomorphism: say $[V_1],....,[V_t]$ is the set of all the full paths of $\text{Char}(G^*)$ that end correctly; then we can define the functionals $T_{[V_1]},...,T_{[V_t]}$, which are then homology classes in $\widehat{HF}(Y_G,\s)$, and together with the full paths $[W_1],...,[W_{t+1}]$ of $\text{Char}(G)$ they form a (canonical) basis of $\widehat{HF}(Y_G,\s)$. Note that both $Y_G$ and $-Y_G=Y_{G^*}$ can be presented by a star-shaped graph with negative framings, see Section \ref{section:two}.

In Theorem \ref{teo:main} we identify the full path of a $V\in\text{Char}(G)$ that ends correctly with both a homology class in $HF^-(Y_G)$, and its projection to $\widehat{HF}(Y_G)$. Duality allows us to identify in a similar way the functional $T_{[V]}$ with a class in $\widehat{HF}(-Y_G)$ and its inclusion in $HF^+(-Y_G)$. In our previous works we gave many details about these properties, see \cite{AC,CM-negative} and Section \ref{section:two} below. 

Say $G$ is negative-definite, and write $\gamma\in\widehat{HF}_*^\text{ev}(Y_G,\s)$ and $\delta\in\widehat{HF}_*^\text{od}(Y_G,\s)$ for two non-zero homogeneous classes. The full path gives us a combinatorial formula, in terms of the Alexander filtration $\mathcal F$ on characteristic vectors \cite{Lattice2,Nemethi,BLZ,CM-negative}, for the $\tau$-invariant \cite{OSz-four,Hedden,AC} of a regular fibre $K$ associated to these classes; note that in the case of $\gamma$ the formula has already appeared in \cite[Equation 3.1]{CM-negative}, and it comes from the results of Alfieri in \cite{Alfieri}. By Theorem \ref{teo:main} we have $\gamma=[W_1]+\cdots+[W_k]$ and $\delta=T_{[V_1]}+\cdots +T_{[V_h]}$, where $\mathcal F(W_i)<\mathcal F(W_k),\mathcal F(W_h)$ for $i<k,h$ and $\mathcal F(V_i)>\mathcal F(V_1)$ for $i>1$.

\begin{teo}
 \label{teo:taus}
 Say $\gamma,\delta\in\widehat{HF}_*(Y_G,\s)$ are as above, and $Y_G=M(e_0;r_1,...,r_n)$ is Seifert fibred with negative-definite standard graph. Then one has \[\tau_\gamma(K)=\dfrac{1}{2}\left(\dfrac{1}{-e(Y_G)}+\min_{[Z]=[W_k]}e_1^TQ^{-1}Z\right)\text{ and }\tau_\delta(K)=\dfrac{1}{2}\left(\dfrac{1}{-e(Y_G)}+\max_{[Z]=[-V_1]}e_1^TQ_*^{-1}Z\right)\] where $Q$ and $Q_*$ are the intersection matrices of $G$ and $G^*$, and $e(Y_G):=r_1+...+r_n+e_0$.
\end{teo}

Note that similar formulae can be obtained when $G$ is indefinite, see Theorem \ref{teo:tau}. We show in Section \ref{section:three} that the filtration $\mathcal F^K$, induced by the regular fibre on $\widehat{HF}(Y_G)$, can be determined combinatorially.

In this paper we use the Hirzebruch-Jung convention \cite{NZM,Saveliev} on negative continued fractions; namely, for any rational number $r\in(0,1)$ we set $-\frac{1}{r}=[m_1,...,m_k]$ with $m_i\leq-2$ integer to be the number defined by the following recursive sequence: \[\left\{\begin{aligned}&m_1\hspace{1.25cm}\text{ when }k=1 \\ &m_1+r'\hspace{0.5cm}\text{ when }k>1\text{ and }-\frac{1}{r'}=[m_2,...,m_k]\:.\end{aligned}\right.\]
The reason why this convention is useful in low dimensional topology is the following: it provides a simple way to express the resulting framing of the slam-dunk of a chain of unknots in a smooth surgery presentation of a 3-manifold. In particular, the lens space $L(p,q)$ with $1\leq q< p$ coprime is diffeomorphic to $-\frac{p}{q}$-surgery on the unknot, and then to the surgery on a chain of unknots whose framings are the terms in the negative continued fraction of $-\frac{p}{q}$.

\subsection*{The classification of negative-twisting structures} The \emph{maximal twisting number} of tight contact structures on a Seifert fibred space $M=M(e_0;r_1,...,r_n)$ is an integer which captures the maximal difference between any contact and fibration framing of a regular fibre $K$. When $M$ is a rational homology sphere and its graph has at least three legs, this is \begin{equation}\text{tw}(M,\xi):=\text{TB}_\xi(K)\:-\text{the fibration framing of } K=\text{TB}_\xi(K)+\frac{1}{e(M)}\label{eq:tw}\end{equation} for a tight contact structure $\xi$ on $M$. By convention \cite{G-}, if $\text{tw}(M,\xi)\geq0$ then we say that $\xi$ is \emph{zero-twisting}; otherwise, we say it is \emph{negative-twisting}. In \cite[Section 4]{CM-negative} we specify how to define the maximal twisting number of the regular fibre in a Seifert fibration on a lens space, when the standard graph has two legs.

We recall that in \cite[Section 2]{G-} Ghiggini established a criterion that determines precisely when a certain negative integer is the twisting number of a tight structure on a Seifert fibred space. Ghiggini's result was originally stated only when the standard graphs have exactly three legs and $e_0=-1$ or $-2$. We extended this criterion to any Seifert fibred space in \cite[Proposition 4.1]{CM-negative}. An analogous result was obtained by Massot in \cite{Massot}.

The following is a standard definition in elementary number theory, see \cite{NZM}: we call a rational number $\frac{p}{q}<1$ a \emph{best upper approximation} for $r\in(0,1)\:\cap\Q$ when $\frac{p}{q}>r,$ $\gcd(p,q)=1$ and no rational number $\frac{k}{h}$ with $1\leq h\leq q$ and $1\leq k\leq p$ is in the interval $(r,\frac{p}{q})$. Note that after we fix $q\geq2$ the best upper approximation is unique (provided it exists).

\begin{teo}[Ghiggini-Massot's algorithm for negative maximal twisting numbers]
 \label{teo:Paolo}
 The Seifert fibred space $M=M(e_0;r_1,\dots,r_n)$ with $n\geq3$ and $r_i\in(0,1)$ rational admits a tight contact structure $\xi$ with twisting number equal to  $\emph{tw}(M,\xi)=-q<-1$ if and only if there exist positive integer numbers $p_1,\dots,p_n$ such that 
\begin{itemize}
 \item $\frac{p_i}{q}$ is the best upper approximation for $r_i$ when $i=1,...,n$;
 \item $p_1+...+p_n = -e_0q+n-2$. 
\end{itemize}   
Furthermore, we have that $-1$ is a twisting number if and only if $e_0\leq-2$.
\end{teo}

The starting point of our classification is to determine exactly which negative integers can be twisting number of a tight structure on $Y_G$. In \cite[Theorem 1.2]{CM-negative} we showed that when $G$ is negative-definite the twisting number is unique; this is no longer the case here. For this reason we call $\overline{\text{tw}}(Y_G)$ the highest and $\underline{\text{tw}}(Y_G)$ the lowest negative twisting number of $Y_G$.

In \cite[Theorem 1.1]{CM-negative} the uniqueness of the twisting number when $G$ is negative-definite allowed us to classify the negative-twisting tight structures on $Y_G$. In particular, we showed that all such structures are Stein fillable, and the filling is smoothly the same 4-manifold: this is $X_G$ the manifold obtained by blowing down the graph $G$. We prove in Section \ref{section:three} that this procedure can be done also when $G$ is either indefinite or singular; moreover, it still produces a Stein domain provided that $Y_G$ is not an $L$-space. The difference with respect to the case in \cite{CM-negative} lies in the fact that now not every negative-twisting structure is obtained in this way; namely, Ghiggini \cite{G-} already showed that oppositely oriented Brieskorn spheres may also carry non-Stein fillable structures. Nonetheless, we prove in Sections \ref{section:four}, \ref{section:five} and \ref{section:six} that the Stein fillable structures realised on $X_G$ are precisely the ones whose twisting number is $\overline{\text{tw}}(Y_G)$.
\begin{prop}
 \label{prop:main}
 Suppose that $Y_G=M(e_0;r_1,\dots,r_n)$ is a Seifert fibred space such that $G$ is not negative-definite. The tight contact structures on $M$ whose negative twisting number is $\overline{\emph{tw}}(Y_G)$ are precisely the Legendrian surgeries on all possible Legendrian realisations of the complete blow-down $X_G$ of $G$; moreover, when $G$ is singular and $Y_G$ is not a torus bundle over $S^1$ there are no other negative-twisting structures.

 Furthermore,  every such $\xi$ is Stein fillable and distinguished, up to isotopy, by its contact invariant $c^+(\xi)=T_{[V]}\in HF^+(-Y_G,\s_\xi)$ where $V\in\emph{Char}(G,\s_\xi)$.
\end{prop}

We exclude torus bundles over the circle because for these manifolds the classification has been done by Honda \cite{Honda2}, and equivalently by Giroux \cite{Giroux1}, we give more details in Section \ref{section:four}.

Note that not every characteristic vector $V$ whose full path ends correctly is such that $T_{[V]}$ is the contact invariant of a structure $\xi$ as above. This only happens when the strict transform of $V$, see \cite[Section 5]{BP} for details, is the characteristic vector of a Stein structure on $X_G$; in this case, we say that $V$ is a \emph{realised characteristic vector}. As we discuss in Section \ref{section:five}, for a vector to be realised we require an additional technical assumption in order for $V$ to be uniquely determined in its full path: $V$ should be an "initial" vector, see \cite{CM-negative} and Section \ref{section:two} for details about this terminology.

Let $V_{\text{can}}=(m(1)+2,...,m(|G|)+2)$ be the canonical characteristic vector \cite{Nemethi}, then we denote by $\s_\text{can}$ the $\Spin^c$-structure on $Y_G$ induced by $V_\text{can}$. We know from \cite{BP} that when $G$ is negative-definite $V_\text{can}$ is realised by the Milnor fillable contact structure $\xi_\text{can}$; in fact, one has $c^+(\xi_\text{can})=T_{[V_\text{can}]}\in HF^+(-Y_G,\s_\text{can})$, and then $M(V_\text{can})=d_3(\xi_\text{can})$, and $\s_\text{can}=\s_{\xi_\text{can}}$. We keep this terminology also for a generic $G$, even though $V_\text{can}$ and $\xi_\text{can}$ are not canonical in the same sense.

In \cite[Definition 2.6]{CM-negative} we introduced the $\height$ of a full path $[V]$ as the number of central steps in $[V]$. Here, we generalise this definition to any homogeneous non-zero class in $\widehat{HF}(Y_G,\s)$, see Definition \ref{def}. In Section \ref{section:three} we also prove that the height is equivalent to the Alexander filtration induced in homology by the regular fibre. 

\begin{teo}
 \label{teo:twisting_numbers}
 If $Y_G=M(e_0;r_1,\dots,r_n)$ with $n\geq3$ is an indefinite Seifert fibred space, then a negative twisting number $\emph{tw}(Y_G,\xi)<\overline{\emph{tw}}(Y_G)$ of a tight contact structure $\xi$ on $Y_G$ is given by \[\emph{tw}(Y_G,\xi)=-1-\height([V_\emph{can}]+[V])\] for any realised characteristic vector $V\neq V_\emph{can}$ such that $[V]\in\widehat{HF}_{d_3(\xi_\emph{can})}(Y_G,\s_\emph{can})$; in addition, we have that \[\overline{\emph{tw}}(Y_G)=-1-\height[V_{\emph{can}}]\:.\] All the negative twisting numbers are obtained in this way; therefore, they are determined by the Heegaard Floer homology of $Y_G$. Furthermore, when $\s_\emph{can}$ is spin we have that \[\underline{\emph{tw}}(Y_G)=-1+e_1^TQ^{-1}V_\emph{can}\:,\] where $Q$ is the intersection matrix of the graph $G$.
\end{teo}

For the other twisting numbers the classification is more complex, and in order to classify these "additional" structures we need to distinguish two different behaviours depending on the graph $G$. We say that an indefinite Seifert fibred space $Y_G$ is of \emph{type B} when $G$ contains one of the seven graphs in Figure \ref{Torus} as a subgraph, and of \emph{type A} otherwise. These seven special graphs are exactly the only Seifert fibred spaces, whose base orbifold is a sphere, which are torus bundles over a circle, see \cite[Subsection 2.1]{Hatcher}.

For a manifold of type A we show that there is always at most one additional structure for each $\Spin^c$-structure, and we construct them using the results of Massot \cite{Massot}. For one of type B we instead show that additional structures form \emph{pyramids}: for each homotopy class of contact structures, which consists of $\xi_1,...,\xi_k$ with twisting number $\overline{\text{tw}}(Y_G)$, then there are $\frac{k(k-1)}{2}$ structures, that we denote by $\xi_{ij}$ for $1\leq i<j\leq k$, with twisting number lower than $\overline{\text{tw}}(Y_G)$. Note that the number $k$ depends on the choice of the homotopy class which in turn is fixed by the underlying $\Spin^c$-structure, see Section \ref{section:six}. 

The only two infinite families of oppositely oriented Brieskorn spheres whose classification of tight contact structures appears in literature are of type B; namely, these are $-\Sigma(2,3,6k-1)$ for $k\geq2$, which was done by Ghiggini and Van Horn-Morris in \cite{GvHM}, and $-\Sigma(2,3,6k+1)$ for $k\geq1$, which was done by Tosun in \cite{Tosun} by an analogous argument. More generally, the count of tight structures is known also for particular surgeries on a singular fibre of $-\Sigma(2,3,6k+1)$, due to Wan \cite{Wan}. 
In Figure \ref{2323} we show the manifold $-\Sigma(2,3,23)$ which belongs to the first family; in this case there are 6 structures, they are all homotopic and belong to the same pyramid.

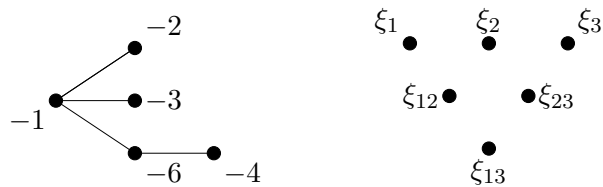
\begin{figure}[t] 
    \begin{tikzpicture}[scale=0.7]
    \tkzDefPoints{0/0/A, 1.5/1/B, 1.5/0/D, 1.5/-1/C, 3/-1/E}  
    \tkzDrawSegment(A,B)\tkzDrawSegment(A,D)\tkzDrawSegment(B,A)\tkzDrawSegment(A,C)\tkzDrawSegment(E,C)
    \tkzDrawPoints[fill,black,size=5](A,B,C,D,E)
     \tkzLabelPoint[below left](A){$-1$} 
     \tkzLabelPoint[above right](B){$-2$}\tkzLabelPoint[below right](C){$-6$}\tkzLabelPoint[below right](E){$-4$}
     \tkzLabelPoint[right](D){$-3$}
       \end{tikzpicture}
       \hspace{1cm}
       \begin{tikzpicture}[scale=0.7]
    \tkzDefPoints{0/0/A, 1.5/0/B, 3/0/C, 0.75/-1/D, 2.25/-1/E, 1.5/-2/F}  
    \tkzDrawPoints[fill,black,size=5](A,B,C,D,E,F)
     \tkzLabelPoint[above left](A){$\xi_1$} \tkzLabelPoint[above](B){$\xi_2$}\tkzLabelPoint[above right](C){$\xi_3$}
     \tkzLabelPoint[left](D){$\xi_{12}$}\tkzLabelPoint[right](E){$\xi_{23}$}\tkzLabelPoint[below](F){$\xi_{13}$}
       \end{tikzpicture}
     \caption{\smaller[1]{The oppositely oriented Brieskorn sphere $-\Sigma(2,3,23)$ is a manifold of type B. The Stein fillable structures $\xi_i$ have contact invariant $T_{[V_i]}$, where $V_i=(1,0,-1,-4,-4+2i)$ for $i=1,2,3$. The structures in each row of the pyramid have twisting number $-5,$ $-11$ and $-17$ respectively. }}
     \label{2323}
\end{figure}

One of the additional advantages of our identification between negative-twisting contact structures and the full path algorithm is the following: we can describe the classification without needing to distinguish between manifolds of type A and B. In fact, our main result can be stated in the form below.

\begin{teo}
 \label{teo:classification}
 Suppose that $Y_G=M(e_0;r_1,\dots,r_n)$ is a Seifert fibred space, and $G$ is indefinite. The tight contact structures $\xi$ on $M$ whose negative twisting number is $-q=\emph{tw}(Y_G,\xi)<\overline{\emph{tw}}(Y_G)$ are in one-to-one correspondence with the unordered pairs $(V_i,V_j)$ for $i\neq j$ of realised characteristic vectors in $\emph{Char}(G,\s_\xi)$ such that \[[V_i],[V_j]\in\widehat{HF}_{d_3(\xi)}(Y_G,\s_\xi)\hspace{0.5cm}\text{ and }\hspace{0.5cm}q=1+\height([V_i]+[V_j])\:.\] Furthermore, these structures are all symplectically fillable. 
\end{teo}    
 
The exact number of negative-twisting tight contact structures on $Y_G$, up to isotopy, is equal to the upper bound from convex surface theory. We write this number explicitly in Section \ref{section:six}.
 
\begin{prop}
 \label{prop:classification}
 Suppose that $Y_G=M(e_0;r_1,\dots,r_n)$ is a Seifert fibred space, and $G$ is indefinite. Then $\xi$ is a negative-twisting tight structure if and only if $c^+(\xi)\in HF_\emph{red}(-Y_G,\s_\xi)$ is non-vanishing. Furthermore, if $\emph{tw}(Y_G,\xi)<\overline{\emph{tw}}(Y_G)$ then $c^+(\xi)=T_{[V_1]}+...+T_{[V_k]}$ for some $k\geq2$, where $V_1,...,V_k$ are realised characteristic vectors in $\emph{Char}(G,\s_\xi)$ which satisfy \[\height([V_1]+[V_k])=\height([V_1]+...+[V_k])=-1-\emph{tw}(Y_G,\xi)\] and $(V_1,V_k)$ is the unordered pair corresponding to $\xi$ in Theorem \ref{teo:classification}. 
\end{prop}

In Sections \ref{section:six} and \ref{section:seven} we describe exactly what are the coordinates, with respect to the canonical basis $\{T_{[V_1]},...,T_{[V_t]}\}$, of the contact invariant $c^+(\xi)\in HF_\text{red}(-Y_G,\s_\xi)$ for each structure appearing in Theorem \ref{teo:classification} and Proposition \ref{prop:classification}.

\subsection*{Applications}
We prove the following results in Section \ref{section:eight}. Our first example concerns Brieskorn spheres; in other words, any Seifert fibred space which is an integral homology sphere, see \cite{Saveliev}. Equipped with the non-canonical orientation, a Brieskorn sphere is presented by an indefinite plumbing graph; therefore, the classification of its negative-twisting structures is a special case of Proposition \ref{prop:main} and Theorem \ref{teo:classification}.
\begin{cor}
 \label{cor:Brieskorn}
 Suppose that $Y=\Sigma(a_1,...,a_n)$ is a canonically oriented Brieskorn sphere different from $\Sigma(2,3,6k\pm1)$ with $k\geq1$. Then $-Y$ is always of type A; furthermore, any negative-twisting structure $\xi$ on $-Y$ is either as in Proposition \ref{prop:main}, when $\emph{tw}(-Y,\xi)=\overline{\emph{tw}}(-Y)$, or is isotopic to the unique structure such that $\emph{tw}(-Y,\xi)=\underline{\emph{tw}}(-Y)$.   
\end{cor}
In the two families remaining we know from \cite{GvHM,Tosun} that there is a unique pyramid of size $k-1$ and $k$ respectively; in particular, it is a famous result of Etnyre and Honda \cite{EH} that $-\Sigma(2,3,5)$ admits no tight structure. Combining our classifications in this paper and in \cite{CM-negative}, we determine exactly which Brieskorn spheres have a unique structure with a given orientation.
\begin{teo}
 \label{teo:Brieskorn}
 The only Brieskorn spheres carrying a unique tight contact structure, up to isotopy, are $\Sigma(2,3,5),$ $-\Sigma(2,3,11)$ and $-\Sigma(2,3,7)$. Such a structure is negative-twisting and Stein fillable.   
\end{teo}

Combining a previous result of Matkovi\v c \cite{Irena(f)} with the ones in this paper, we complete the classification of fillable structures on any small Seifert fibred space. In particular, every negative-twisting structure on these manifolds is fillable, confirming a conjecture in \cite{Irena(f)}.

\begin{teo}
 Every symplectically fillable structure on $M(e_0;r_1,r_2,r_3)$ is either included in our classifications (either above or in \cite[Theorem 1.1]{CM-negative} or \cite[Theorem 1.1]{Irena(f)}) when $e_0\leq-1$, or is included in the classification in \cite[Theorem 1.1]{GLS} when $e_0\geq0$. 
\end{teo}
\begin{proof}
 Any fillable structure is tight and then either negative- or zero-twisting. Then the claim follows from Proposition \ref{prop:main}, Theorem \ref{teo:classification} and \cite{CM-negative} in the first case, and from \cite{GLS,Irena(f)} in the second one. Note that when $n=3$ the structures considered in the latter papers are the only zero-twisting structures on $M$, as half convex Giroux torsion makes the structure overtwisted; moreover, when $e_0\leq-2$ there are no zero-twisting structures by convex surface theory arguments \cite{Wu}. 
\end{proof}

This includes rational surgeries on the torus knot $T_{d_2,\pm d_1}$, except for the slope $\pm d_1d_2$ when fillable structures are easily determined anyway. In general, for surgeries on torus links we have the following result.

\begin{cor}
 \label{cor:surgeries}
 Every negative-twisting structure on $S^3_{r}(T_{kd_2,\pm kd_1})$ with $r\in\Q,$ $k\geq1$ and $1\leq d_2\leq d_1$ coprime is either included in our classifications (either above or in \cite[Corollary 5.3]{CM-negative}), or it is the connected sum of some tight structures on lens spaces.    
\end{cor}

For surgeries on torus knots we can explicitly determine the surgery interval which admit any given twisting number. Note that $S^3_{r}(T_{d_2,d_1})$ is of type B only when $(d_2,d_1)=(2,\pm3)$.

\begin{prop}\label{prop:Tsurgeries}
 The manifold $S^3_{r}(T_{d_2,d_1})$ with $r\in\Q$ carries a negative-twisting tight structure if and only if $r<d_1d_2-d_2-d_1$ or $r>d_1d_2$. When $(d_2,d_1)\neq(2,3)$ there are at most two possible negative twisting numbers: one is always equal to either $-d_2-d_1$ or $-1$, while the other is
 \[-d_2-d_1-sd_1d_2\hspace{0.5cm}\text{ for an }s>0\hspace{0.5cm}\text{ when }\hspace{0.5cm}-d_1d_2+r\in\big[[n_1,...,n_{h-1}],[n_1,...,n_h]\big)\:,\] where $[n_1,...,n_h]=-\frac{d_1+d_2+sd_1d_2}{s+1}$ is reduced. Furthermore, when $(d_2,d_1)=(2,3)$ one twisting number is always equal to $-5$, while the other ones are
 \[-5-6s\hspace{0.5cm}\text{ for }\hspace{0.5cm}\left\{\begin{aligned}
    &s>0\hspace{2.05cm}\text{ when }\hspace{0.5cm}r=0 \\ &s=1,...,n-2\hspace{0.5cm}\text{ when }\hspace{0.5cm}r\in\left[\frac{1}{n},\frac{1}{n-1}\right)\:.
 \end{aligned}\right.\]
 The manifold $S^3_{r}(T_{d_2,-d_1})$ carries a tight structure with twisting number $-1$ if and only if $r>-d_2d_1$; moreover, when $(d_2,d_1)\neq(2,3)$ another negative twisting number appears if and only if it exists for $S^3_{r}(T_{d_2,d_1})$, and its value is lowered by $2(d_1d_2-d_1-d_2)$. When $(d_2,d_1)=(2,3)$ the other twisting numbers are the ones on $S^3_{r}(T_{2,3})$ lowered by $2$.
\end{prop}

Combining Proposition \ref{prop:Tsurgeries} with the elementary number theory in Section \ref{section:six}, we can determine the exact number of negative-twisting structures on either $S^3_r(T_{d_2,d_1})$ or $S^3_r(T_{d_2,-d_1})$. Take $[m_1,...,m_k]$ to be $-d_1d_2+r$ and $-\frac{d_1d_2+r}{d_1d_2+r-1}$ respectively, and let $L$ be the number of tight structures on the lens space obtained by removing the corresponding leg; then we have:
\begin{itemize}[leftmargin=0.7cm]
    \item either $|m_1+d_1+d_2|\cdot|m_2+1|\cdots|m_k+1|$ or $L\cdot|m_1+1|\cdot|m_2+1|\cdots|m_k+1|$ structures with twisting number $\overline{\text{tw}}$ equal to either $-d_1-d_2$ or $-1$ (if it existed);
    \item either $|m_3+1|\cdots|m_k+1|$ pyramids of size $|m_2+1|$ or $|m_7+1|\cdots|m_k+1|$ pyramids of size $|m_6+1|$ when $(d_2,d_1)=(2,\pm3)$ and $0<r<1$, at most one in each $\Spin^c$-structure;
    \item $\max\{|m_h-n_h|\cdot|m_{h+1}+1|\cdots|m_k+1|,\:1\}$ structures with twisting number smaller than $\overline{\text{tw}}$ (if it existed) when $(d_2,d_1)\neq(2,\pm3)$, at most one in each $\Spin^c$-structure, where $[n_1,...,n_h]$ is either the reduced fraction $-\frac{d_1+d_2+sd_1d_2}{s+1}$ or $-\frac{d_1+d_2-(s+2)d_1d_2}{s+1+d_1+d_2-(s+2)d_1d_2}$ for an $s>0$.
\end{itemize}

\begin{figure}[h] 
       \begin{tikzpicture}[scale=0.6]
    \tkzDefPoints{0/0/A, 1.5/0/B, 3/0/C, 4.5/0/G, 0.75/-1/D, 2.25/-1/E, 3.75/-1/H, 1.5/-2/F, 3/-2/I, 2.25/-3/L}  
    \tkzDrawPoints[fill,black,size=5](A,B,C,E,G)\tkzDrawPoints[fill,gray,size=5](D,F,H,I)\tkzDrawPoints[fill,burgundy,size=5](L)
     \tkzLabelPoint[above left](A){$\xi_1$} \tkzLabelPoint[above left](B){$\xi_2$}\tkzLabelPoint[above right](C){$\xi_3$}\tkzLabelPoint[above right](G){$\xi_4$}
     \tkzLabelPoint[left](D){$\xi_{12}$}\tkzLabelPoint[right](H){$\xi_{34}$}\tkzLabelPoint[above](E){$\xi_{23}$}
     \tkzLabelPoint[left](F){$\xi_{13}$}\tkzLabelPoint[right](I){$\xi_{24}$}
     \tkzLabelPoint[below](L){$\xi_{14}$}
       \end{tikzpicture}
     \caption{\smaller[1]{We call the grey structures casing stones, and the red one pyramidion. If the pyramid appeared in a spin structure then $c^+(\xi_{23})$ and $c^+(\xi_{14})$ would both be self-conjugate.}}
     \label{Pyramid}
\end{figure}
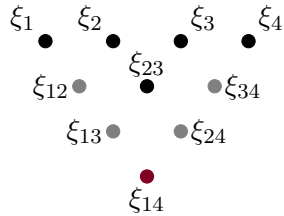

In Theorem \ref{teo:classification} we show that every negative-twisting structure on a Seifert fibred space is symplectically fillable; moreover, Proposition \ref{prop:main} tells us that when the twisting number is equal to $\overline{\text{tw}}(Y_G)$ we can actually produce a Stein filling. In some cases it is possible to show that Stein fillings do not exist.

We introduce some terminology for negative-twisting structures on Seifert fibred spaces $Y_G$ of type B, based on Figures \ref{2323} and \ref{Pyramid}. We say that a contact structure $\xi$ with $\text{tw}(Y_G,\xi)<\overline{\text{tw}}(Y_G)$ is a \emph{pyramidion} if it appeared as the apex of its pyramid; in other words, when it has the lowest twisting number in its homotopy class. In the same way, we say that $\xi$ is a \emph{casing stone} if it appeared in a pyramid with base $\xi_1,...,\xi_k$ as either $\xi_{1j}$ for $1<j<k$ or $\xi_{lk}$ for $1<l<k$.

\begin{prop}
 \label{prop:obstruction}
 Let $Y_G$ be a Seifert fibred space with $G$ indefinite, and $\xi$ be a negative-twisting structure on $Y_G$. If $\emph{tw}(Y_G,\xi)<\overline{\emph{tw}}(Y_G)$ and $c^+(\xi)$ is self-conjugate under $\J$, that is $\J c^+(\xi)=c^+(\xi)$, then $\xi$ is not Stein fillable.   
\end{prop}
\begin{proof}
 We describe the argument for the obstruction, which comes from Ghiggini's work \cite{G-notStein}. By the proof of Proposition \ref{prop:classification} we have that $c^+(\xi)=T_{[V_{-h}]}+\cdots +T_{[V_h]}\in HF^+(-Y_G,\s_\xi)$ for some $V_{\pm1},...,V_{\pm h}\in\text{Char}(G,\s_\xi)$ realised; since $c^+(\xi)$ is self-conjugate, we can assume that $[V_{-i}]=[-V_i]=\J[V_i]$ for every $i$. Suppose that $Y_G$ admits a Stein filling $(X,J)$; from \cite{OSz-contact,G-fillability} we know that $F^+_{\overline X,J}(c^+(\xi))=1$ and $F^+_{\overline X,-J}(c^+(\xi))=F^+_{\overline X,-J}(c^+(\overline\xi))=1$, while from Plamenevskaya \cite{OlgaP} we obtain that $J$ coincides with $-J$ as a $\Spin^c$-structure. Hence, for each $i$ we have $F^+_{\overline X,J}(T_{[V_i]})=F^+_{\overline X,J}(T_{[V_{-i}]})$ which implies $F^+_{\overline X,J}(c^+(\xi))=0$, and this is a contradiction. 
\end{proof}

Note that $c^+(\xi)$ is self-conjugate for any additional structure on a type A manifold provided that $\s_\xi$ is spin. When the structure is neither a casing stone nor a pyramidion a different obstruction, coming from a result of Christian and Menke in \cite{CM}, has been given in \cite{Min,ES} for a certain choice of the coefficients. 

For Brieskorn spheres Corollary \ref{cor:Brieskorn} and Proposition \ref{prop:obstruction} obstruct Stein fillability for every negative-twisting structure $\xi$ with $\text{tw}(Y_G,\xi)<\overline{\text{tw}}(Y_G)$ except when $Y_G=-\Sigma(2,3,6k\pm1)$. In the latter case, from the results mentioned above, only the casing stones can possibly be Stein fillable.

\subsection*{Sample computations} We now illustrate our method by writing the classification explicitly in two examples: the first one is $-\Sigma(3,4,47)=S^3_{1/4}(T_{3,4})$, while the second one is $M(-2;\frac{1}{2},\frac{1}{2},\frac{4}{7},\frac{6}{11})$, see Figure \ref{Examples}. We want to highlight that our method, which involves computing the realised characteristic vectors through the full path algorithm, behaves in the same way independently of the fact that the first manifold is a surgery on a torus knot. We also observe that the number of structures and pyramids (when the manifold is of type B), and the maximal twisting number, can be determined immediately using the formulae in Subsection \ref{subsection:background} and Ghiggini-Massot's algorithm in Theorem \ref{teo:Paolo}, while by using the full path we can see the gradings of each contact structure.

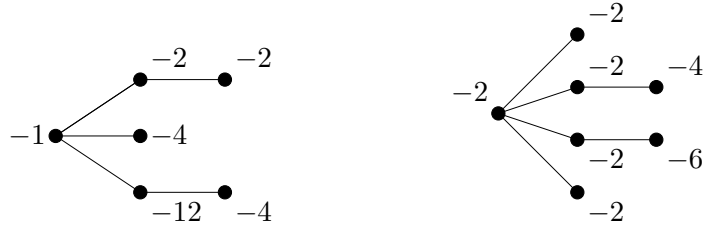
\begin{figure}[ht]
 \begin{tikzpicture}[scale=0.75]
    \tkzDefPoints{0/0/A, 1.5/1/B, 3/1/E, 1.5/-1/D, 3/-1/F, 1.5/0/C}  
    \tkzDrawSegment(A,B)\tkzDrawSegment(A,D)\tkzDrawSegment(B,A)\tkzDrawSegment(A,C)\tkzDrawSegment(B,E)\tkzDrawSegment(F,D)
    \tkzDrawPoints[fill,black,size=5](A,B,C,D,E,F)
     \tkzLabelPoint[left](A){$-1$} 
     \tkzLabelPoint[above right](B){$-2$}\tkzLabelPoint[right](C){$-4$}\tkzLabelPoint[above right](E){$-2$}
     \tkzLabelPoint[below right](D){$-12$}\tkzLabelPoint[below right](F){$-4$}
       \end{tikzpicture}\hspace{2cm}
      \begin{tikzpicture}[scale=0.7] 
      \tkzDefPoints{0/0/A, 1.5/1.5/B, 1.5/-1.5/D, 1.5/0.5/C, 3/0.5/E, 1.5/-0.5/G, 3/-0.5/F}      
       \tkzDrawSegment(A,B)\tkzDrawSegment(A,D)\tkzDrawSegment(A,C)\tkzDrawSegment(A,G)\tkzDrawSegment(C,E)\tkzDrawSegment(G,F)
    \tkzDrawPoints[fill,black,size=5](A,B,C,D,E,F,G)
     \tkzLabelPoint[above left](A){$-2$} \tkzLabelPoint[above right](B){$-2$}\tkzLabelPoint[above right](C){$-2$}
     \tkzLabelPoint[below right](D){$-2$}\tkzLabelPoint[below right](G){$-2$}\tkzLabelPoint[below right](F){$-6$}\tkzLabelPoint[above right](E){$-4$}
\end{tikzpicture}
       \caption{\smaller[1]{A plumbing graph representing $-\Sigma(3,4,47)=M(-1;\frac{2}{3},\frac{1}{4},\frac{4}{47})$ (left), and one representing $M(-2;\frac{1}{2},\frac{1}{2},\frac{4}{7},\frac{6}{11})$ (right).}}
     \label{Examples}\end{figure}

\begin{table}[t]
\centering
\begin{NiceTabular}{|c|c|c c|c|c|}[
  cell-space-top-limit=4pt,
  cell-space-bottom-limit=4pt
]
\hline
\Block{2-1}{$\s_\xi$}
& \Block{2-1}{$d_3(\xi)$}
& $V\in\mathrm{Char}(M)$ &
& Combined
& \Block{2-1}{$\mathrm{tw}(M,\xi)$}
\\
& & realised & & height & \\
\hline\hline\hline
\Block{15-1}{$\s_\text{can}=\overline \s_\text{can}$} & \Block{2-1}{$19$} & $(1,0,0,-2,-10,-2)$ & \Block{2-1}{\joindots} & \Block{2-1}{222} & \Block{2-1}{$-7$ and $-223$} \\
 &  & $(1,0,0,-2,-2,2)$ & &  &  \\
\cline{2-6}
 & \Block{2-1}{15} & $(1,0,0,-2,-10,0)$ & $\bullet$ & \Block{2-1}{198} & \Block{13-1}{$-7$}  \\
 &  & $(1,0,0,-2,-2,0)$ & $\bullet$ &  &  \\
\cline{2-5}
 & \Block{2-1}{11} & $(1,0,0,-2,-10,2)$ & $\bullet$ & \Block{2-1}{174} &  \\
 &  & $(1,0,0,-2,-2,-2)$ & $\bullet$ &  &  \\
\cline{2-5}
 & \Block{2-1}{5} & $(1,0,0,-2,-8,-2)$  & $\bullet$ & \Block{2-1}{126} &  \\
 &  & $(1,0,0,-2,-4,2)$  & $\bullet$ &  &  \\
\cline{2-5}
 & \Block{2-1}{3} & $(1,0,0,-2,-8,0)$  & $\bullet$ & \Block{2-1}{102} & \\
 &                &  $(1,0,0,-2,-4,0)$   & $\bullet$ &                     &         \\
\cline{2-5}
 & \Block{2-1}{1} & $(1,0,0,-2,-8,2)$  & $\bullet$ & \Block{2-1}{78} & \\
 &                & $(1,0,0,-2,-4,-2)$    & $\bullet$ &                     &         \\ 
\cline{2-5}
 & \Block{3-1}{$-1$} & $(1,0,0,-2,-6,-2)$  & $\bullet$ & \Block{3-1}{30} & \\
 &                &  $(1,0,0,-2,-6,0)$  & $\bullet$ &                     &         \\ 
 &                &  $(1,0,0,-2,-6,2)$   & $\bullet$ &                     &         \\
\hline
\end{NiceTabular}
\caption{\smaller[1]{There are $16$ negative-twisting contact structures on $M=-\Sigma(3,4,47)$ which is of type A. One of them has twisting number $-223$, while for the other fifteen this is $-7$.}}
\label{Table1}
\end{table}      

\begin{table}[t]
\centering
\begin{NiceTabular}{|c|c|c c|c|c|}[
  cell-space-top-limit=4pt,
  cell-space-bottom-limit=4pt
]
\hline
\Block{2-1}{$\s_\xi$}
& \Block{2-1}{$d_3(\xi)$}
& $V\in\mathrm{Char}(M)$ & 
& Combined
& \Block{2-1}{$\mathrm{tw}(M,\xi)$}
\\
& & realised  & & height & \\
\hline\hline\hline
\Block{3-1}{$\s_\text{can}$} & \Block{3-1}{$\frac{5}{36}$} & $(0,0,0,-2,0,-4,0)$ & \Block{3-1}{\joindotss} & \Block{3-1}{4} & \Block{3-1}{$-1,$ $-3$ and $-5$} \\
 &  & $(0,0,0,0,0,-2,0)$ &  &  &  \\
 &  & $(0,0,0,2,0,0,0)$  & &  &  \\
\hline\hline
\Block{3-1}{$\overline\s_\text{can}$} & \Block{3-1}{$\frac{5}{36}$} & $(0,0,0,-2,0,0,0)$ &  \Block{3-1}{\joindotss} & \Block{3-1}{4} & \Block{3-1}{$-1,$ $-3$ and $-5$}  \\
 &  & $(0,0,0,0,0,2,0)$ & &  &  \\
 &  & $(0,0,0,2,0,4,0)$ & &  &  \\
\hline\hline
\Block{3-1}{$\s_1=\overline\s_1$} & \Block{3-1}{$\frac{1}{4}$} & $(0,0,0,-2,0,-2,0)$ & \Block{3-1}{\joindotss} & \Block{3-1}{4} & \Block{3-1}{$-1,$ $-3$ and $-5$} \\
 &  & $(0,0,0,0,0,0,0)$ &  &  &  \\
 &  & $(0,0,0,2,0,2,0)$ & &    &  \\
\hline\hline
\Block{2-1}{$\s_2$} & \Block{2-1}{$-\frac{7}{36}$} & $(0,0,0,-2,0,2,0)$  & \Block{2-1}{\joindots} & \Block{2-1}{2} & \Block{2-1}{$-1$ and $-3$} \\
 &  & $(0,0,0,0,0,4,0)$  &  &  &  \\
\hline\hline
\Block{2-1}{$\overline\s_2$} & \Block{2-1}{$-\frac{7}{36}$} & $(0,0,0,0,0,-4,0)$  & \Block{2-1}{\joindots}  & \Block{2-1}{2} & \Block{2-1}{$-1$ and $-3$} \\
 &                &  $(0,0,0,2,0,-2,0)$ & &                     &         \\
\hline\hline
$\s_3$ & $-\frac{3}{4}$ & $(0,0,0,-2,0,4,0)$ & $\bullet$ & 0 & $-1$ \\
\hline\hline
$\overline\s_3$ & $-\frac{3}{4}$ & $(0,0,0,2,0,-4,0)$ & $\bullet$ & 0 & $-1$ \\
\hline
\end{NiceTabular}
\caption{\smaller[1]{There are $26$ negative-twisting contact structures on $M=M(-2;\frac{1}{2},\frac{1}{2},\frac{4}{7},\frac{6}{11})$ which is of type B. Three of them have twisting number $-5$, eight of them have twisting number $-3$, and for the other fifteen this is $-1$. The structures are distributed into seven pyramids: three of size $3$, two of size $2$, and two of size $1$. Note that $|H_1(M;\Z)|=36$, but only seven $\Spin^c$-structures support realised characteristic vectors.}}
\label{Table2}
\end{table}      

Let us assume that $M=-\Sigma(3,4,47)=M(-1;\frac{2}{3},\frac{1}{4},\frac{4}{47})$. Then the possible negative twisting numbers are $-7$ and $-223$; in fact, we have that \[(p_1,p_2,p_3)=(5,2,1)\hspace{0.5cm}\text{ and }\hspace{0.5cm}(P_1,P_2,P_3)=(149,56,19)\:;\] hence, comparing each $r_i$ with $\frac{p_i}{7}$ and $\frac{P_i}{223}$ for $i=1,2,3$, from Equation \eqref{eq:upper1} the number of contact structures is \[1\cdot1\cdot15=15\hspace{0.5cm}\text{ and }\hspace{0.5cm}1\cdot1\cdot1=1\] matching the computation in Table \ref{Table1} and in Proposition \ref{prop:Tsurgeries}.

The realised characteristic vectors are obtained as described in Theorem \ref{teo:realised}: \[(1,0,0,-2,x,y)\hspace{0.5cm}\text{ where }\hspace{0.5cm}x=-10,-8,...,-2\hspace{0.5cm}\text{ and }\hspace{0.5cm}y=-2,0,2\:.\] The combined height in Tables \ref{Table1} and \ref{Table2} is the maximum height of any homology class $[V_i]+[V_j]$ among realised vectors with the given gradings (provided that there are at least two of them). Note that the unique structure on $M$ with twisting number $-223$ appears in grading 19 because, according to Theorems \ref{teo:twisting_numbers} and \ref{teo:classification}, we need $\height([V_i]+[V_j])=-1-\underline{\text{tw}}(M)=222$.

We now assume that $M=M(-2;\frac{1}{2},\frac{1}{2},\frac{4}{7},\frac{6}{11})$; we use the same procedure as above. The possible negative twisting numbers are $-1,$ $-3$ and $-5$; in fact, we have that $e_0\leq-2$ and \[(p_1,p_2,p_3,p_4)=(2,2,2,2)\hspace{0.5cm}\text{ and }\hspace{0.5cm}(P_1,P_2,P_3,P_4)=(3,3,3,3)\:;\] hence, from Equation \eqref{eq:upper2} and comparing each $r_i$ with $\frac{p_i}{3}$ and $\frac{P_i}{5}$ for $i=1,...,4$, the number of contact structures is \[1\cdot1\cdot3\cdot5\cdot1=15\:,\hspace{0.5cm}1\cdot1\cdot2\cdot4\cdot1=8\hspace{0.5cm}\text{ and }\hspace{0.5cm}1\cdot1\cdot1\cdot3\cdot1=3\] matching the computation in Table \ref{Table2}.

The realised characteristic vectors are again obtained from Theorem \ref{teo:realised}: \[(0,0,0,x,0,y,0)\hspace{0.5cm}\text{ where }\hspace{0.5cm}x=-2,0,2\hspace{0.5cm}\text{ and }\hspace{0.5cm}y=-4,-2,...,4\:.\] 

\begin{remark}
  We prove in Theorem \ref{teo:twisting_numbers} that for any $V\in\emph{Char}(G)$ realised one has \[\height[V]=-1-\overline{\emph{tw}}(Y_G)=\height[V_\emph{can}]\:.\] This means that for every $[V']\in\widehat{HF}_{M(V)}(Y_G,\mathfrak u\lvert_{Y_G})$, where $c_1(\mathfrak u)=V$, such that $[V']\neq[V]$ we have $\height([V]+[V'])>\height[V_\emph{can}]$. Hence, the combined height in Tables \ref{Table1} and \ref{Table2} can only equal $\height[V_\emph{can}]$ if the contact structure with invariant $T_{[V]}$ is alone in its homotopy class.
\end{remark}

\subsection*{Acknowledgements} {\smaller[1] We are indebted to Paolo Ghiggini for his mathematical insight and invaluable help. We thank Hyunki Min for noticing a minor oversight in Proposition \ref{prop:Tsurgeries}. We thank the Matematiska institutionen at Uppsala universitet for their friendly hospitality. A.C. has been partially supported by the HORIZON-ERC-2023-ADG 101141468 KnotSurf4d project. }

\section{On the full path algorithm}
\label{section:two}
We assume that the reader is familiar with \cite[Section 2]{CM-negative}. Every Seifert fibred space whose base orbifold is a sphere is presented by a star-shaped graph that we call the \emph{standard graph}. Unless the manifold is either $S^3,$ $S^1\times S^2$ or a lens space, the standard graph has at least three legs and it is unique once we fix our convention \cite{Saveliev}. 

We denote the standard graph of a Seifert fibred space $M$ by $G$, or by $\Gamma$ when we assume $M$ to be oriented in the way its graph is negative-definite. We take $M=M(e_0;r_1,...,r_n)$ where $e_0\in\Z$ and $r_i\in(0,1)\:\cap\Q$; in addition, for the reason we explained above, in this paper we always assume that $n\geq3$. As usual the framings on the vertices in the legs of $G$ are determined by the negative continued fraction expansion of the $r_i$'s; hence, except for the central vertex whose framing is $e_0$, all the others are smaller or equal than $-2$. 

The dual graph $G^*$ is then the standard graph of \[-M=M(-e_0-n;1-r_1,...,1-r_n)\:;\] in particular, we have that $e(-M)=-e(M)$ where $e(M)=e_0+r_1+...+r_n$. We know from the literature \cite{Saveliev} that \[e(M)\left\{\begin{aligned}&<0\hspace{0.5cm}\text{ when }G\text{ is negative-definite} \\ &=0\hspace{0.5cm}\text{ when }b_1(M)=1 \\ &>0\hspace{0.5cm}\text{ when }G\text{ is indefinite}\:.\end{aligned}\right.\] Since removing the central vertex always produces a negative-definite graph, the plumbing $P_{G}$ is negative-definite if and only if $b_2^+(P_{G^*})=1$, or vice-versa $b_2^+(P_G)=1$ if and only if $P_{G^*}$ is negative-definite.

We call $G'$ any \emph{maximal $S^3$-subgraph} of $G$; in other words, a maximal subgraph that can be cancelled by a sequence of blow-downs, starting from $G$. The graph $G'$ always corresponds to the Seifert fibration of a torus knot $T_{d_2,\pm d_1}$ in $S^3$, including the degenerate cases $T_{1,d_1}$ and $T_{1,1}$; the differences appearing in these two cases are highlighted in \cite[Sections 4 and 5]{CM-negative}.

Let us denote by $Y$ any 3-manifold presented by an \emph{almost-rational graph}, that is the union of some negative-definite plumbing trees such that lowering the framing of one vertex makes $Y$ an $L$-space. We always denote by $-Y$ the manifold obtained by reversing the orientation. Every Seifert fibred space $M$ as above with $e(M)\neq0$ has standard graph which is almost-rational with exactly one orientation. 

We recall the notation about Heegaard Floer maps relating the three main flavours of the homology groups: \[HF^-(-Y,\s)\overset{\psi^*}{\longrightarrow}\widehat{HF}(-Y,\s)\overset{\rho^*}{\longrightarrow}HF^+(-Y,\s)\] and \[HF^-(Y,\s)\overset{\psi_*}{\longrightarrow}\widehat{HF}(Y,\s)\overset{\rho_*}{\longrightarrow}HF^+(Y,\s)\:,\] induced by the projection $CF^-\rightarrow\widehat{CF}$ and inclusion $\widehat{CF}\rightarrow CF^+$ respectively. For rational homology 3-spheres we have the splitting \[\widehat{HF}(Y,\s)\simeq\widehat{HF}^\text{ev}(Y,\s)\oplus\widehat{HF}^\text{od}(Y,\s)\:,\] where the first is the subgroup of the classes with Maslov grading $i$ with the same parity of the correction term $d(Y,\s)$, that is $i-d(Y,\s)\equiv0$ $\text{mod }2$, while the second one is generated by the classes with Maslov grading $j$ such that $j-d(Y,\s)\equiv1$ $\text{mod }2$.

It follows from \cite{OSz-fullpath,Nemethi}, see also \cite[Theorem 2.1]{CM-negative}, that if $Y$ is almost-rational then \begin{equation}\widehat{HF}^\text{ev}(-Y,\s)\simeq_{\rho^*}\Ker U\subset HF^+(-Y,\s)\label{eq:ev}\end{equation} and \begin{equation}\widehat{HF}^\text{od}(-Y,\s)=\Ker\rho^*\simeq_{\psi^*}\text{Tor}HF^-(-Y,\s)\:/\: U\cdot HF^-(-Y,\s)\:.\label{eq:od}\end{equation}
We know that the full path algorithm from \cite{OSz-fullpath} provides a basis as an $\F$-vector space of \begin{equation}HF^-(Y,\s)/ U\cdot HF^-(Y,\s)\simeq\widehat{HF}^\text{ev}(Y,\s)=\left(\widehat{HF}^\text{od}(-Y,\s)\right)^\perp\label{eq:even}\end{equation} through the canonical duality $\widehat{HF}_*(-Y,\s)\simeq\widehat{HF}_{-*}(Y,\s)^\bullet$. We want to show that a relation similar to Equation \eqref{eq:even} holds for $-M$ when $M$ is a negative-definite Seifert fibred space. 

We recall that a 4-manifold $X$, together with a $\Spin^c$-structure $\mathfrak u$, induces maps in Heegaard Floer homology between $S^3$ and $(Y=\partial X,\s=\mathfrak u\lvert_Y)$; moreover, when $c_1(\mathfrak u)$ is torsion the image has Maslov grading equal to \begin{equation}M(F^\circ_{X,\mathfrak u}(1))=\dfrac{c_1^2(\mathfrak u)[X]-3\sigma(X)-2\chi(X)+2}{4}\:,\label{eq:Maslov}\end{equation} and when $b_1(X)=0$ the homology class $F^-_{X,\mathfrak u}(1)$ is non-torsion if and only if $b_2^+(X)=0$, see \cite{OSz-negative}.
\begin{teo}
 \label{teo:fullpath}
 Let $Y$ be a negative-definite Seifert fibred space, and denote by $\Gamma$ its standard graph and by $P_{\Gamma^*}$ the plumbing associated to its dual $\Gamma^*$ so that $\partial P_{\Gamma^*}=-Y$. Then assuming that the central vertex of $\Gamma^*$ has negative framing, the full paths of $\Gamma^*$ ending correctly provide an $\F$-basis of $\widehat{HF}^\emph{od}(-Y,\s)$. More specifically, for every characteristic vector $V\in\Z^{|\Gamma^*|}\simeq H^2(P_{\Gamma^*};\Z)$ one has $F^-_{P_{\Gamma^*},\mathfrak u}(1)=[V]\in\emph{Tor}HF^-(-Y,\s)$ where $c_1(\mathfrak u)=V$.     
\end{teo}
Note that this implies that we have a basis $\{[V_1],...,[V_{t}],T_{[W_1]},...,T_{[W_{t+1}]}\}$ of $\widehat{HF}(-Y,\s)$, obtained by identifying the full path of each characteristic vector $V_i=c_1(\mathfrak u_i)$ with $\widehat F_{P_{\Gamma^*},\mathfrak u_i}(1)$ and of each $W_j=c_1(\mathfrak v_j)$ with $\widehat F_{P_{\Gamma},\mathfrak v_j}(1)$ where $\mathfrak v_j\in\Spin^c(P_\Gamma)$, assuming they end correctly. Then the functional $T_{[W_j]}:HF^-(Y,\s)\rightarrow\F$ defined by \[T_{[W_j]}[W_j]=1\hspace{0.5cm}\text{ and }\hspace{0.5cm}T_{[W_j]}[W_k]=0\hspace{0.5cm}\text{ when }\hspace{0.5cm}j\neq k\] is an element of $\Ker U\subset HF^+(-Y,\s)$, and then it is identified with one of $\widehat{HF}^\text{ev}(-Y,\s)$ by Equation \eqref{eq:ev}. In the same way, we have that $\{T_{[V_1]},...,T_{[V_{t}]},[W_1],...,[W_{t+1}]\}$ is a basis of $\widehat{HF}(Y,\s)$, see \cite[Subsection 2.4]{CM-negative} for more details.

We postpone the proof of Theorem \ref{teo:fullpath} to Subsection \ref{subsection:proof}. We do not give all the possible details about characteristic vectors and the full path, as we refer to the original paper of Ozsv\'ath and Szab\'o \cite{OSz-fullpath} for these; nonetheless, we will highlight the differences between the indefinite and negative-definite case. Note that even though the original proof works for every almost-rational graph, here we rely heavily on the fact that graph is star-shaped. We do not think that our results can be generalised easily to a larger family of 3-manifolds.

\subsection{The surgery exact triangle}
\label{subsection:triangle}
We start by describing the strategy of the proof of Theorem \ref{teo:fullpath}. Say $Y$ is represented by the negative-definite graph $\Gamma$, and let $P=P_{\Gamma^*}$ be the plumbing of $\Gamma^*$; we consider the following exact surgery triangle 
\begin{equation}
    \begin{tikzcd}
 HF^\circ(-Y) \arrow[dr,"h" swap] & & HF^\circ(-Y_1) \arrow[ll,"g" swap] \\
 & HF^\circ(-Y_0) \arrow[ur,"f" swap] &  
\end{tikzcd}\label{eq:triangle}
\end{equation}
where $-Y_0$ is the connected sum of lens spaces obtained by removing the centre of $\Gamma^*$, and $-Y_1$ is the manifold whose graph has framing on the centre lowered by one with respect to $\Gamma^*$. Clearly, we have that $-Y_0$ is presented by a negative-definite tree, but $-Y_1$ may not be. Since lowering the framing on the central vertex enough times yields a tree with no bad vertices, and then a negative-definite $L$-space, we can use an inductive argument to prove that the set $\{F^-_{P,\mathfrak u}(1)\}$ where $c_1(\mathfrak u)=V\in H^2(P;\Z)$ is a system of generators of $\text{Tor}HF^-(-Y)$.

A main problem immediately arises: since $-Y$ has indefinite graph, while the manifold we use in the initial step has negative-definite graph, this means that at some point in the argument we may be forced to consider an intermediate case where the manifold is not a rational homology sphere. The full path algorithm works also in this case, with the due modifications, as shown by Rustamov \cite[Theorem 1.2]{Rustamov}. 

We devote this subsection to describe the maps involved in the exact triangle in Equation \eqref{eq:triangle}. We recall that each map is induced by a 4-dimensional cobordism $W$ between two vertices of the triangle, and it is given by summing the maps induced in Heegaard Floer homology by each pair $(W,\mathfrak v)$ with $\mathfrak v\in\Spin^c(W)$; Ozsv\'ath and Szab\'o showed that such maps are all vanishing except for a finite number of $\mathfrak v$'s.

Let $\Gamma^*_0$ be the graph of $-Y_0$, and $\Gamma^*_1$ be the one of $-Y_1$; denote by $P_0$ and $P_1$ the corresponding plumbings $P_{\Gamma^*_0}$ and $P_{\Gamma^*_1}$. By construction the plumbing $P_0$ is negative-definite and $P_1=P_0\cup_{-Y_0}W$, where $W$ is the cobordism from $-Y_0$ to $-Y_1$ given by the 2-handle attachment on the central vertex. A $\Spin^c$-structure on $W$ can then be determined by the first coordinate of a characteristic vector $V\in H^2(P_1;\Z)$, and the components of $f$ are precisely the maps $F^\circ_{W,\mathfrak v_1}:HF^\circ(-Y_0,\s_0)\rightarrow HF^\circ(-Y_1,\s_1)$ where $\mathfrak v_1\lvert_{-Y_0}=\s_0$ and $\mathfrak v_1\lvert_{-Y_1}=\s_1$. In order to give a complete notation, for any $\mathfrak u_0\in\Spin^c(P_0)$ such that $\mathfrak u_0\lvert_{-Y_0}=\s_0$ the vector $V=(v_1,...,v_{|\Gamma^*_1|})$ is identified with $c_1(\mathfrak u_1)$ where $\mathfrak u_1=\mathfrak u_0\cup_{-Y_0,\s_0}\mathfrak v_1$, while $V'=(v_2,...,v_{|\Gamma^*_1|})$ is identified with $c_1(\mathfrak u_0)\in H^2(P_0;\Z)$. Since $Y_0$ is an $L$-space, for every $\s_0\in\Spin^c(Y_0)$ there is an $\mathfrak u_0$ as above such that $F^\circ_{P_0,\mathfrak u_0}(1)=\theta_{\s_0}$, the generator of $HF^\circ(-Y_0,\s_0)$; hence, we have that \[f_{\mathfrak v_1}(\theta_{\s_0})=F^\circ_{W,\mathfrak v_1}(F^\circ_{P_0,\mathfrak u_0}(1))=F^\circ_{P_1,\mathfrak u_1}(1)\:.\]
The cobordism between $-Y_1$ and $-Y$ is instead given by a blow-up on the central vertex. Therefore, the map $g$ has two components that we denote by $g_\epsilon$ with $\epsilon=\pm 1$. 
Since we are using the inductive argument mentioned above, we can assume that Theorem \ref{teo:fullpath} holds for $P_1$, thus if $b_2^+(P_1)=1$ (resp. $b_2^+(P_1)=0$) then each non-zero class $\gamma\in \widehat{HF}(-Y_1,\s_1)$ with odd (resp. even) grading can be written as $\widehat F_{P_1,\mathfrak u_1}(1)$ for a certain $\mathfrak u_1\in\Spin^c(P_1)$. We conclude that \[g_\epsilon(\gamma)=g_\epsilon(\widehat F_{P_1,\mathfrak u_1}(1))=g_\epsilon(f_{\mathfrak v_1}(\theta_{\s_0}))=\widehat F_{P,\mathfrak u_1\cup_{-Y_1}\epsilon}(1)\] where $\mathfrak v_1=\mathfrak u_1\lvert_{W}$ and $\s_0=\mathfrak u_1\lvert_{-Y_0}$. Note that exactness implies that $g_\epsilon(\gamma)=0$ if $f_{\mathfrak v_i}(\theta_{\s_0})\neq0$ only for an odd number of $\mathfrak v_i\in\Spin^c(W)$ for which $f_{\mathfrak v_i}(\theta_{\s_0})=f_{\mathfrak v_1}(\theta_{\s_0})$.

Finally, the map $h$ is gotten by adding a 0-framed 2-handle along the regular fibre of $-Y$ (a meridian of central vertex). We observe that then its dual $h^\bullet:HF^\star(Y_0)\rightarrow HF^\star(Y)$ corresponds to attaching the 2-handle on the central vertex of $Y_0$ (the symbol $\star$ denotes the flavour dual to $\circ$); in other words, the map $h^\bullet$ would play the same role as $f$ if we took the exact triangle with the oppositely oriented manifolds instead. For this reason we can just write \[h^\bullet_\mathfrak v(\eta_{\s_0})=F^\star_{W',\mathfrak v}(F^\star_{P_{\Gamma_0},\mathfrak u_0}(1))=F^\star_{P_{\Gamma},\mathfrak u}(1)\] where $\eta_{\s_0}$ is the generator of $HF^\star(Y_0,\s_0)$, the $\Spin^c$-structure $\mathfrak v\in\Spin^c(W')$ is such that $\s_0=\mathfrak v\lvert_{Y_0}$, and $\s=\mathfrak v\lvert_{Y}$, while $\mathfrak u=\mathfrak u_0\cup_{Y_0,\s_0}\mathfrak v$. Note that one always has that $b_2^+(P_\Gamma)=0$.

\subsection{Heegaard Floer homology and plumbings of star-shaped graphs}
Let us use the notation $(Z,\s)$ for a Seifert fibred space with $b_1(Z)=1$. We recall that since $H^2(Z;\Z)\simeq\Z\oplus\mathfrak G$ where $\mathfrak G$ is a finite abelian group, we have infinitely many $\Spin^c$-structures: we denote by $\mathfrak t$ the $|\mathfrak G|$ ones with torsion first Chern class, and with $\mathfrak j$ the other ones which are indexed by $\{i\}\times\mathfrak G$ with $i\in\Z\setminus\{0\}$. 

Taking a look at its Heegaard Floer homology \cite{OSz}, since $b_1(Z)=1$ we have that $HF^-(Z,\mathfrak t)$ has rank two as an $\F[U]$-module; therefore, we have two correction terms that we denote by $d_\text{od}(\mathfrak t)$ and $d_\text{ev}(\mathfrak t)$, where $d_\text{ev}(\mathfrak t)$ is the one with same parity as the twisted coefficient correction term. This is consistent with the notation in \cite{Rustamov} because of the results in \cite{LR}; moreover, the difference $d_\text{od}(\mathfrak t)-d_\text{ev}(\mathfrak t)$ is odd. In addition, the group $HF^+(Z,\mathfrak t)$ possesses two distinguished $\F[U]$-submodules isomorphic to $\mathcal T:=\F[U,U^{-1}]/ U\cdot\F[U]$ with degree given by the correction terms.

On the other hand, for non-torsion $\Spin^c$-structures we need to be careful as the version of $HF^-$ considered here and in \cite{Rustamov} is not the original one defined by Ozsv\'ath and Szab\'o: we set 
\vspace{-0.1cm}
\[HF^-(Z,\mathfrak j):=HF^+(-Z,\mathfrak j)^\bullet\] identifying it with its completion as an $\F[[U]]$-module. In this way, we have that \[\dim_\F HF^-(Z,\mathfrak j)=\dim_\F HF^+(Z,\mathfrak j)\] 
and $HF^-(Z,\mathfrak j)$ has only $U$-torsion elements. For these $\Spin^c$-structures we cannot define the Maslov grading as an absolute $\Z$-grading, but just as an $\F$-grading.  

If $\mathfrak t$ is torsion then we say that a homology class is in $\widehat{HF}^\text{ev}(Z,\mathfrak t)$ when its grading has the same parity of $d_\text{ev}(\mathfrak t)$. This means that the image under $\psi_*$ of generators of distinct $\F[U]$-towers always live in subgroups that have opposite parity. \vspace{-0.1cm}

\begin{prop}
 \label{prop:sign}
 Suppose that $Y$ and $Z$ are closed oriented $3$-manifolds such that $b_1(Y)=0$ and $b_1(Z)=1$. Then \[\widehat{HF}^\emph{ev}(Y,\s)\simeq\widehat{HF}^\emph{ev}(-Y,\s)\hspace{0.5cm}\text{ and }\hspace{0.5cm}\widehat{HF}^\emph{ev}(Z,\mathfrak t)\simeq\widehat{HF}^\emph{od}(-Z,\mathfrak t)\] for every $\s\in\Spin^c(Y)$ and every $\mathfrak t\in\Spin^c(Z)$ torsion.   
\end{prop}
\begin{proof}
    For $Y$ it follows by the canonical duality $\widehat{HF}_*(-Y,\s)\simeq\widehat{HF}_{-*}(Y,\s)^\bullet$ which holds for every 3-manifold and preserves the parity of the Maslov grading. For $Z$ we need to observe that the correction terms of $-Z$ are $-d_\text{od}(\mathfrak t)$ and $-d_\text{ev}(\mathfrak t)$; since as in \cite{Rustamov} we set $d_\text{od}(-Z,\mathfrak t):=-d_\text{ev}(Z,\mathfrak t)$ and vice-versa, the parity of the homology classes is swapped under the duality identification.
\end{proof}
From now on we assume that $Z$ is a Seifert fibred space with $b_1(Z)=1$; we exclude $S^1\times S^2$ as this case is not relevant in the paper. From \cite[Theorem 1.2]{Rustamov} we know that the full path algorithm provides bases of $\widehat{HF}^\text{ev}(Z,\mathfrak \s)$ and the even subgroup of $HF^-(Z,\mathfrak \s)$. 
\vspace{-0.1cm}

\begin{lemma}[Rustamov]
 \label{lemma:1}
 Let $Z$ be as above, and denote by $G$ its standard graph. Then for every $\s\in\Spin^c(Z)$ we have that there are $\mathfrak u_1,...,\mathfrak u_t\in\Spin^c(P_G)$ extending $\s$ such that the maps $F^-_{P_G,\mathfrak u_j}:HF^-(S^3)\rightarrow HF^-(Z,\s)$ altogether give an $\F[U]$-basis $\{F^-_{P_G,\mathfrak u_1}(1),...,F^-_{P_G,\mathfrak u_t}(1)\}$ of $\emph{Tor}HF^-(Z,\s)$ when $\s$ is non-torsion, and an $\F[U]$-basis of $\F[U]_{d_\emph{ev}}\oplus\emph{Tor}HF^-(Z,\s)$ when $\s$ is torsion. 
\end{lemma}

We now need to prove an equivalent version of the lemma above in the case that $\Gamma^*$ is indefinite. 

\begin{lemma}
 \label{lemma:2}   
  Let $Y$ be a negative-definite Seifert fibred space, and denote by $\Gamma$ its standard graph. Then for every $\s\in\Spin^c(-Y)$ we have that there are $\mathfrak u_1,...,\mathfrak u_t\in\Spin^c(P_{\Gamma^*})$ extending $\s$ such that the maps $F^-_{P_{\Gamma^*},\mathfrak u_j}:HF^-(S^3)\rightarrow HF^-(-Y,\s)$ altogether give $\{F^-_{P_{\Gamma^*},\mathfrak u_1}(1),...,F^-_{P_{\Gamma^*},\mathfrak u_t}(1)\}$, an $\F[U]$-basis of $\emph{Tor}HF^-(-Y,\s)$. 
\end{lemma}
\begin{proof}
 We assume that the framing on the central vertex of $\Gamma^*$ is negative (that is $e_0>-n$), otherwise $Y$ would be an $L$-space and there is nothing to prove as $\text{Tor}HF^-(-Y,\s)\simeq\{0\}$. We keep the notation of Subsection \ref{subsection:triangle}; hence, we set $P=P_{\Gamma^*}$ (with $b_2^+(P)=1$) and $\partial P=-Y$ while $\partial P_{\Gamma}=Y$ (with $b_2^+(P_{\Gamma})=0$). We can use the minus flavour of the triangle.

 We take the exact triangle in Equation \eqref{eq:triangle}. We take the map $h^\bullet:HF^+(Y_0)\rightarrow HF^+(Y)$ defined by fixing $\Spin^c$-structures $\mathfrak u_0$ on $P_{\Gamma_0}$ such that $\s_0=\mathfrak u_0\lvert_{Y_0}$ and by \vspace{-0.1cm}
 \[h^\bullet(\eta_{\s_0})=\bigoplus_{\substack{\mathfrak u\in\Spin^c(P_{\Gamma}) \\ \mathfrak u\text{ extends }\mathfrak u_0}} F^+_{P_{\Gamma},\mathfrak u}(1)\:.\] 
 Since $b_2^+(P_{\Gamma})=0$ the map $F_{P_{\Gamma},\mathfrak u}$ is non-zero in $HF^\infty$ \cite{OSz-negative}, then the commutativity of the cobordism maps implies $F^+_{P_{\Gamma},\mathfrak u}(1)=\pi_*(F^\infty_{P_{\Gamma},\mathfrak u}(1))$ where $\pi_*:HF^\infty\rightarrow HF^+$. This means that $\Imm F^+_{P_{\Gamma},\mathfrak u}$ is the subgroup $\mathcal T^+$ of $HF^+(Y)$; therefore, by duality we obtain that 
 \[\Imm g=\Ker h=(\Imm h^\bullet)^\perp=(\mathcal T^+)^\perp=\text{Tor}HF^-(-Y)\:.\] 
 To find a system of generators of $\text{Tor}HF^-(-Y)$ we proceed as follows.
 Since the components of $g$ lie in different $\Spin^c$-structures, it follows easily that we can find $\mathfrak u_1',...,\mathfrak u_t'\in\Spin^c(P)$ extending $\mathfrak s$ such that the set $\{F^-_{P,\mathfrak u_i'}(1)=g_\epsilon(F^-_{P_1,\mathfrak u_i'\lvert_{P_1}}(1))\}_{1\leq i\leq t}$, where $\epsilon$ denotes the restriction of $\mathfrak u_i'$ to the cobordism induced by the blow-up, is a system of generators of $\text{Tor}HF^-(-Y,\mathfrak s)$ as an $\F[U]$-module. We can then extract a basis from it. 
\end{proof}

As we mentioned in the previous subsections, if $Y$ is negative-definite then one of the results coming from Ozsv\'ath-Szab\'o's full path algorithm is that there is a basis of $\widehat{HF}^\text{ev}(Y,\s)$ which consists of homology classes of the form $\widehat F_{P_\Gamma,\mathfrak u}(1)$ with $\mathfrak u\in\Spin^c(P_\Gamma)$. We can now show an equivalent version for $-Y$; for $Z$ this follows in the same way using the results in \cite{Rustamov}.

\begin{cor}
 \label{cor:even}
 Suppose that $M$ is a Seifert fibred space. If $b_1(M)=1$ then there is a basis of $\widehat{HF}^\emph{ev}(M,\s)$ which consists of homology classes of the form $\widehat F_{P_{G},\mathfrak u}(1)$ with $\mathfrak u\in\Spin^c(P_{G})$.

 Furthermore, if $M$ is indefinite then there is a basis of $\widehat{HF}^\emph{od}(M,\s)$ which consists of homology classes of the form $\widehat F_{P_{G},\mathfrak u}(1)$ with $\mathfrak u\in\Spin^c(P_{G})$.
\end{cor}
\begin{proof}
 From the commutativity of the cobordism maps we have that $\widehat F=\psi_*\circ F^-$, thus implying that if $\widehat F_{P_G,\mathfrak u}(1)$ is non-zero then $F^-_{P_G,\mathfrak u}(1)$ is also non-zero.
 When $M$ is indefinite the claim follows from \cite{OSz-fullpath}, Proposition \ref{prop:sign}, and Lemma \ref{lemma:2} and its proof. 
\end{proof}
In general, when the manifold is not a rational homology sphere we cannot determine a priori whether $d_\text{ev}(Z,\mathfrak t)$ is the maximal or the minimal between the two correction terms. However, if we apply the same results to $(-Z,\mathfrak t)$, then by duality $d_\text{od}(Z,\mathfrak t)=-d_\text{ev}(-Z,\mathfrak t)$; hence, the full path allows us to find both correction terms combinatorially.    

\subsection{The full path of any standard graph}
\label{subsection:proof}
In this section we consider a star-shaped graph $G$ whose central vertex $S_1$ has negative framing $m(1)$, and each other vertex $S_i$ has framing $m(i)\leq-2$ for $1\leq i\leq|G|$. We denote by $Q$ the intersection matrix of the graph, so that $m(1),...,m(|G|)$ are the elements on the main diagonal of $Q$. 

Following the methods in \cite[Section 2]{OSz-fullpath} and the notation in \cite[Section 2]{CM-negative}, we fix a $\Spin^c$-structure $\s$ on the 3-manifold presented by $G$, and we introduce the set $\Z^{\geq0}\times\text{Char}(G,\s)$ where we identify a characteristic vector $V$ with $c_1(\mathfrak u)\in H^2(P_G;\Z)$ for every $\Spin^c$-structure $\mathfrak u$ extending $\s$ on the plumbing $P_G$. We recall that $V=(v_1,...,v_{|G|})$ is characteristic when $v_i\equiv m(i)\text{ mod }2$ for $1\leq i\leq|G|$.
We denote each pair $(k,V)$ by $U^k\cdot V$ to highlight the relation with Heegaard Floer homology.

Ozsv\'ath and Szab\'o define the following equivalence relation on the set $\Z^{\geq0}\times\text{Char}(G,\s)$. Whenever the equality $v_i=-m(i)+2n$ holds for some integers $i$ and $n$ we have that \[U^{n+k}\cdot(V+2Qe_i)\sim U^k\cdot V \hspace{1cm}\text{if }n\geq0\] while \[U^k\cdot(V+2Qe_i)\sim U^{k-n}\cdot V\hspace{1cm} \text{ if }n\leq0\] for every $k\geq0$.

\begin{defin}
 The equivalence classes of $\sim$ are the full paths, and we denote them by $[V]$ where $V=(v_1,...,v_{|G|})$ is a characteristic vector. Let $\mathcal B_n$ be the subset of characteristic vectors such that $|v_i|\leq-m(i)+2n$ for every $1\leq i\leq|G|$; if $V\in\mathcal B_0$ then we say that $[V]$ \emph{ends correctly} if every other $Z\in[V]$ is also in $\mathcal B_0$. Furthermore, the vector $V$ is called \emph{initial} when \[m(i)+2\leq v_i\leq-m(i)\:,\] while it is called \emph{terminal} when \[m(i)\leq v_i\leq-m(i)-2\] for $i=1,...,|G|$, see \cite{OSz-fullpath,CM-negative}.
\end{defin}

It is proved in \cite[Proposition 3.2]{OSz-fullpath} that if $Q$ is negative-definite then the full paths ending correctly have unique initial and terminal characteristic vectors; moreover, as we mentioned in the previous subsections their number equals the dimension of $\widehat{HF}^\text{ev}$. Regarding the uniqueness, the proof in \cite{OSz-fullpath} holds in the same way when $Q$ is indefinite; this is no longer true if $\det(Q)=0$: in this case a full path that ends correctly may have loops in it, see \cite[Section 3]{Rustamov} for some examples. 
Note that, when this happens, every vector $Z$ in $[V]$ possesses coordinates $w_i$ and $w_j$ that are equal to $m(i)$ and $-m(j)$. 

We now define the space $\mathbb H^+(G,\s)$. Given a function $\phi:\text{Char}(G,\s)\rightarrow\mathcal T^+_0$ with finite support, there is an induced map \[\widetilde\phi:\Z^{\geq0}\times\text{Char}(G,\s)\longrightarrow\mathcal T^+_0\hspace{0.5cm}\text{ defined by }\hspace{0.5cm}
\widetilde\phi(U^n\cdot V) = U^n\cdot\phi(V)\:;\] then we set $\mathbb H^+(G,\s)$ to be the set of functions $\phi$ whose induced map $\widetilde\phi$ is constant among each full path. Clearly, the space $\mathbb H^+(G,\s)$ inherits the structure of an infinitely generated $\F[U]$-module, or is zero in the case every full path is unbounded and then no non-zero $\phi$ has finite support. 

\begin{lemma}
 \label{lemma:phi}
 There is an identification between the subspace $\Ker U^{n+1}\subset\mathbb H^+(G,\s)$ for every $n\geq0$ and the space $\mathcal K_n$ of the maps from $\Z^{\geq0}\times\emph{Char}(G,\s)/\sim$ to $\F$ which vanish on the full paths containing a representative of the form $U^k\cdot V$ with $k>n$.
\end{lemma}
\begin{proof}
 The proof follows the one of the second part of \cite[Lemma 2.3]{OSz-fullpath}. The homomorphism $\Ker U^{n+1}\rightarrow\mathcal K_n$ is induced by the duality map \[\begin{aligned}\mathbb H^+(G,\:&\s)\times\left(\Z^{\geq0}\times\text{Char}(G,\s)\right)\longrightarrow\F \\ &(\phi,\:U^k\cdot V)\mapsto\left(U^k\cdot\phi(V)\right)_0\end{aligned}\] where $(\cdot)_0$ denotes the component of $\mathcal T^+_0$ in degree zero. This map is injective, because if $U^k\cdot\phi(V)\notin\Ker U\subset\mathcal T^+_0$ and $U^{n+1}\cdot\phi(V)=0$ for any $k\leq n$ and $V\in\text{Char}(G,\s)$, then $\phi(V)$ has at the same time degree at least $2(n+1)$ and at most $2n$; thus $\phi$ is zero. The map is also surjective, because for each $\tau\in\mathcal K_n$ we can define \[\phi(V)=\tau(V)+\sum_{h=1}^n U^{-h}\cdot\tau(U^h\cdot V)\:.\] Clearly, we have that $\phi\in\mathbb H^+(G,\s)$ and $U^{n+1}\cdot\phi=0$.
\end{proof}

We define a map $T^+$ inducing an $\F[U]$-equivariant $\Spin^c$-preserving map from an appropriate subspace of $HF^+$ to \[\mathbb H^+(G):=\bigoplus_{\s\in\Spin^c(M)}\mathbb H^+(G,\s)\:.\] Let $M$ be any Seifert fibred space whose standard graph $G$ has $e_0\leq-1$: we define $\mathcal Z\subset HF^-(M)$ to be equal to $\Imm g$ as in the exact triangle in Subsection \ref{subsection:triangle}. 

When $G$ is negative-definite the map $g$ is surjective \cite{OSz-fullpath}. When $G$ is indefinite then $\mathcal Z$ consists of the classes with odd degree; in other words, the ones that generate $\text{Tor}HF^-(M)$ because of Lemma \ref{lemma:2}. Finally, when $b_1(M)=1$ the subspace $\mathcal Z$ is spanned by the torsion and the $\F[U]$-towers whose degrees have even parity, see Lemma \ref{lemma:1}. 

We set $T^+:\mathcal Z^\bullet\rightarrow\mathbb H^+(G)$ as follows: the plumbing $P_{G}$, turned upside-down, can be viewed as giving a cobordism from $-M$ to $S^3$; now fix $\alpha\in\mathcal Z^\bullet$ the subgroup such that $\mathcal Z^\bullet\oplus\mathcal Z^\perp= HF^+(-M)$ determined by the parity of the Maslov grading, and let \[T^+(\alpha):\text{Char}(G)\longrightarrow\mathcal T^+_0\] be the map given by \[T^+(\alpha)(V)=F^+_{\overline{P_G},\mathfrak u}(\alpha)\in\mathcal T^+_0=HF^+(S^3)\:,\] where $\mathfrak u\in\Spin^c(P_G)$ is the $\Spin^c$-structure whose first Chern class is $V$. The fact that $T^+$ is well-defined, equivariant and $\Spin^c$-preserving follows from Lemmate \ref{lemma:1} and \ref{lemma:2} and the properties of the maps induced by cobordisms, see \cite[Theorem 7.1]{OSz-four}, while that it is constant among full paths from \cite[Theorem 3.1]{OSz-symplectic}. In addition, it is also easy to check that if $U^{n+1}\cdot\alpha=0$ then \[\tau_\alpha:U^k\cdot V\mapsto \left(U^k\cdot T^+(\alpha)(V)\right)_0\] is in $\mathcal K_n$ for $n\geq0$.

Note that when $G$ is negative-definite then $\mathcal Z^\bullet$ is the whole $HF^+(-M)$. This is the case studied by Ozsv\'ath and Szab\'o in \cite{OSz-fullpath} and by N\'emethi in \cite{Nemethi}. We prove that $T^+$ is an isomorphism; note that when $b_1(M)=1$ we have the same result from \cite[Theorem 1.2]{Rustamov}.

\begin{prop}[Rustamov]
 \label{prop:zero}
 Suppose that $Z$ is a Seifert fibred space such that $b_1(Z)=1$, and let $G$ be its standard graph. Then $T^+:\mathcal Z^\bullet\rightarrow\mathbb H^+(G)$ is a $\Spin^c$-preserving isomorphism of $\F[U]$-modules, where here $\mathcal Z^\bullet$ is the odd subgroup of $HF^+(-Z)$.
\end{prop}

We are now ready to prove the relation between the full path and the Heegaard Floer homology groups of Seifert fibred spaces stated in Theorem \ref{teo:fullpath}.

\begin{proof}[Proof of Theorem \ref{teo:fullpath}]
 We are going to prove that $HF_\text{red}(Y)\simeq\mathbb H^+(\Gamma^*)$ where $Y$ is the negative-definite Seifert fibred space represented by $\Gamma$. Note that because of Equation \eqref{eq:even} we know that $\mathcal T_{d(Y,\s)}^+$ contains precisely all the classes with even parity of $HF^+(Y,\s)$; hence, for these manifolds we can identify the quotient $HF_\text{red}$ with the subgroup of the classes in $HF^+$ with odd parity. Theorem \ref{teo:fullpath} will then follow by duality, Lemma \ref{lemma:2} and Corollary \ref{cor:even}.  

 We prove that $T^+:\mathcal Z^\bullet=HF_\text{red}(Y)\rightarrow\mathbb H^+(\Gamma^*)$ is an isomorphism. We use the same notation as in Subsection \ref{subsection:triangle}; hence, we call $Y_1$ the manifold obtained by lowering the central vertex of $G$ by one, and $Y_0$ the one by removing it. The manifold $Y_0$ has negative-definite standard graph. We have the following commutative diagram with exact lines,

 \begin{center}
\begin{tikzcd}
0 \arrow[r] & HF_\text{red}(Y) \arrow[r, "g^\bullet"] \arrow[d, "T^+"] & \mathcal Z_1^\bullet \arrow[d, "T^+_1" ] \arrow[r, "f^\bullet"] & HF^+(Y_0)  \arrow[d, "T^+_0"]  \\
0 \arrow[r] &  \mathbb H^+(\Gamma^*)  \arrow[r, "\mathbb G^+"] & \mathbb H^+(\Gamma^*_1) \arrow[r, "\mathbb F^+"]  &  \mathbb H^+(\Gamma^*_0)
\end{tikzcd} 
\end{center} 
 where the first line comes from Lemma \ref{lemma:2}, while the second one is purely combinatorial and follows from \cite[Lemma 2.10]{OSz-fullpath}. We recall that $T^+_1$ and $T_0^+$ are isomorphisms by induction. Note that since $e(-Y)>0$, lowering the framing on the central vertex may yield a manifold $-Y_1$ with $b_1(Y_1)=1$.

 We prove that $T^+$ is injective: suppose that $\alpha\in HF_\text{red}(Y)$ is such that $T^+(\alpha)=0$, thus $T^+_1(g^\bullet(\alpha))=0$; this implies $g^\bullet(\alpha)=0$ and then $\alpha=0$, because $T_1^+$ is an isomorphism and $g^\bullet$ is injective. We prove that $T^+$ is surjective: take $\phi\in\mathbb H^+(G)$, thus $\mathbb F^+(\mathbb G^+(\phi))=0$; hence, there exists a class $\beta\in HF^+(-Z_1)$ such that $T_1^+(\beta)=\mathbb G^+(\phi)$ and $f^\bullet(\beta)=0$. By exactness we find a class $\alpha\in\mathcal Z^\bullet$ such that $g^\bullet(\alpha)=\beta$, and then by commutativity we have \[\mathbb G^+(T^+(\alpha))=T_1^+(g^\bullet(\alpha))=T_1^+(\beta)=\mathbb G^+(\phi)\:;\] we conclude that $T^+(\alpha)=\phi$ using the injectivity of $\mathbb G^+$.
  
 From Lemma \ref{lemma:phi} the subgroup $\Ker U\cap\:HF_\text{red}(Y,\s)$ coincides with the space $\mathcal K_0$ of maps from $\Z^{\geq0}\times\text{Char}(\Gamma^*,\s)/\sim$ to $\F$ which vanish on the full paths containing a representative of the form $U^k\cdot V$ with $k\geq1$. In addition, by duality we have an identification between $\widehat{HF}^\text{od}(-Y)$ and the full paths of $\Gamma^*$ which end correctly.
\end{proof}

We conclude this section by completing the proof of Theorem \ref{teo:main}, our first main result in the introduction. All the claims are already proven in Theorem \ref{teo:fullpath}, we just need to show that the conjugation involution $\J$ commutes with the isomorphism $T^+$.

\begin{proof}[Proof of Theorem \ref{teo:main}]
 As we mentioned above the claim and the first two items follow from Theorem \ref{teo:fullpath}. The fact that $\mathcal J$ coincides with the reflection $[V]\mapsto[-V]$, on the set of the full paths for every standard graph $G$, can be observed for any manifold $M$, presented by a star-shaped graph with negative framings, from the commutativity of the cobordism maps under $\J$ \cite{G-fillability}; in fact, set $[V]=c_1(\mathfrak u)$ then we have that $\J[V]=\J\cdot F^-_{P_G,\mathfrak u}(1)=F^-_{P_G,-\mathfrak u}(1)=[-V]$. In addition, if $\alpha\in HF^-(M,\s)$ is not represented by a full path then by our results in this section it belongs to the canonical $\F[U]$-tower whose parity is either even (when $b_1(M)=0$) or odd (when $b_1(M)=1$); hence, both the involutions are the identity on this subgroup because they preserve the Maslov grading.
\end{proof}

\section{Invariants from the standard graph}
\label{section:three}
\subsection{Gradings}
\label{subsection:gradings}
The set of the characteristic vectors of a plumbing tree $G$, that here we assume to be either negative-definite or indefinite, corresponds to $\Spin^c(P_G)$. Therefore, each vector induce a $\Spin^c$-structure on the manifold $Y_G=\partial P_G$. Using linear algebra, we describe an easy criterion to determine when the same structure is induced by different vectors. 
\begin{prop}
 \label{prop:spinc}
 Suppose that $G$ is a plumbing tree such that $b_1(Y_G)=0$, and let $Q_G$ be its intersection matrix. Two characteristic vectors $V,W\in H^2(P_G;\Z)$ extends the same $\Spin^c$-structure on $Y_G$ if and only if $Q_G^{-1}(V-W)\in2\Z^{|G|}$. Furthermore, the vector $V$ extends a spin structure if and only if $Q_G^{-1}V\in\Z^{|G|}$.  
\end{prop}
\begin{proof}
 As explained in \cite{OSz-fullpath}, the vectors $V$ and $W$ correspond to $c_1(\mathfrak u_1)$ and $c_1(\mathfrak u_2)$ where $\mathfrak u_i\in\Spin^c(P_G)$. The matrix $Q_G$ is the intersection form of $P_G$, which is simply connected; therefore, the $\Spin^c$-structure on the boundary is the same if and only if the Chern classes are in the same equivalence class of $H^2(P_G;\Z)/\:2Q_G\cdot H^2(P_G;\Z)$, which means precisely that the vector $V-W$ is in the image of $2Q_G$ over the integers.

 We have that $\s\in\Spin^c(Y_G)$ is spin if and only if $\s=\overline\s$; hence, in terms of vectors if and only if $V$ and $-V$ restrict to the same structure on $Y_G$. Using the previous part of the statement, we then have that $V$ extends a spin structure if and only if $Q_G^{-1}(V-(-V))\in2\Z^{|G|}$ which implies the claim.
\end{proof}
We now consider the standard graph $\Gamma$ of a negative-definite Seifert fibred space $Y$. We know from \cite{OSz-fullpath,Nemethi} that each characteristic vector $W$ whose full path ends correctly has a well-defined grading $M(W)=\frac{W^TQ^{-1}W+|\Gamma|}{4}$, which coincides with the Maslov grading. From Theorem \ref{teo:fullpath} we know that if $V=c_1(\mathfrak u)\in H^2(P_{\Gamma^*};\Z)$ ends correctly then $\widehat F_{P_{\Gamma^*},\mathfrak u}(1)\in\widehat{HF}(-Y)$ can be identified with its full path $[V]$ through the isomorphism \[\text{Tor}HF^-(-Y)\:/\:U\cdot HF^-(-Y)\simeq_{\psi^*}\:\widehat{HF}^\text{od}(-Y);\] since $b_2^+(P_{\Gamma^*})=1$ we can define the Maslov grading of $V$ as \begin{equation}M(V)=\dfrac{c_1^2(\mathfrak u)[P_{\Gamma^*}]+|\Gamma^*|-6}{4}=\dfrac{V^TQ_*^{-1}V+|\Gamma^*|-6}{4}\label{eq:Maslov2}\end{equation} using Equation \eqref{eq:Maslov}. It is easy to check that, with this definition in place, the isomorphism $T^+$ used in the proof of Theorem \ref{teo:fullpath} is degree-preserving; in particular, if $V_1,V_2\in\text{Char}(\Gamma^*,\s)$ end correctly then $M(V_1)-M(V_2)$ is even.

Let us consider the knot $K\subset Y$ isotopic to a regular fibre of the Seifert fibration. As in \cite[Sections 2 and 3]{CM-negative} we have that $K$ is presented as an unmarked leaf of $\Gamma$, attached to the central vertex. The knot $K$ allows us to define the Alexander filtration on $\widehat{HF}^\text{od}(-Y)$, in the same way as we did in \cite{CM-negative} for $\widehat{HF}^\text{ev}(Y)$:
\begin{equation}
 \mathcal F(V)=-\dfrac{[D]\cdot[D]}{2}+\dfrac{c_1(\mathfrak u)[D]}{2}=\dfrac{1}{2e(Y)}+\dfrac{e_1^TQ_*^{-1}V}{2}   \label{eq:Alexander}
\end{equation}
for every characteristic vector $V\in\text{Char}(\Gamma^*)$, where $D\subset P_{\Gamma^*}$ is the disk bounded by $K$. We recall that we are indexing the central vertex as $S_1$; moreover, if $V_1,V_2\in\text{Char}(\Gamma^*,\s)$ end correctly then $\mathcal F(V_1)-\mathcal F(V_2)$ is an integer: \[\mathcal F(V_1)-\mathcal F(V_2)=\dfrac{e_1^TQ_*^{-1}(V_1-V_2)}{2}=b_1\in\Z\] where $B=(b_1,...,b_{\Gamma^*})=Q_*^{-1}\left(\frac{V_1-V_2}{2}\right)$ from Proposition \ref{prop:spinc}. 

\begin{remark}
 \label{remark:Alexander}
 We proved in \cite[Lemma 2.4]{CM-negative} that if $W_1,W_2\in\emph{Char}(\Gamma,\s)$ end correctly and $\mathcal F(W_1)=\mathcal F(W_2)$ then they are in the same full path. The same proof works also for $-Y$: if $V_1,V_2\in\emph{Char}(\Gamma^*,\s)$ end correctly and $\mathcal F(V_1)=\mathcal F(V_2)$ then they are in the same full path. Say $\{[V_1],...,[V_t]\}$ is the basis of $\widehat{HF}^\emph{od}(-Y,\s)$ given by the full paths ending correctly, this means that there are intervals $[a_{V_i},b_{V_i}]\subset\Q$ such that $b_{V_i}<a_{V_{i+1}}$ for $i=1,...,t$, and the $i$-th interval contains only the values of $\mathcal F$ taken by the vectors in the full path $[V_i]$. 
\end{remark}

From what we said in Section \ref{section:two} we have a basis $\{[W_1],...,[W_{t+1}],T_{[V_1]},...,T_{[V_{t}]}\}$ of $\widehat{HF}(Y,\s)$; the first part is given by the full paths $\{[W_1],...,[W_{t+1}]\}$ that end correctly, which is a basis of $\widehat{HF}^\text{ev}(Y,\s)$ from \cite{OSz-fullpath,Nemethi}, while the second part is given by $\widehat{HF}^\text{od}(Y,\s)\simeq\widehat{HF}^\text{ev}(-Y,\s)^\perp$, because $\{[V_1],...,[V_{t}]\}$ is a basis of $\widehat{HF}^\text{od}(-Y,\s)$ by Theorem \ref{teo:fullpath} and Corollary \ref{cor:even}.

In \cite[Section 2]{CM-negative} we defined the filtration $\mathcal F$ on $\widehat{HF}^\text{ev}(Y,\s)$, while Equation \eqref{eq:Alexander} does the same for $\widehat{HF}^\text{od}(-Y,\s)$. We can extend $\mathcal F$ on $\widehat{HF}^\text{od}(Y,\s)$, and then on the whole $\widehat{HF}(Y,\s)$, in the same way as the Alexander filtration in knot Floer homology, see \cite{OSz-fourgenus}; this means the subspace of the functionals $T_{[V]}$ such that $\mathcal F(T_{[V]})\leq m$ is the orthogonal of the one of the vectors $V$ such that $\mathcal F(V)\leq-m-1$. Explicitly, we write \begin{equation}\mathcal F(T_{[V]})=-\mathcal F(V)=\dfrac{1}{-2e(Y)}-\dfrac{e_1^TQ_*^{-1}V}{2}\:.\label{eq:minus}\end{equation}

\subsection{Tau-invariants}
Let $\mathcal F_m^K$ be the level $m$ of the filtration on $\widehat{CF}(M)$ induced by the regular fibre $K\subset M$, that is the leaf attached to the first vertex of its standard graph. Then for $\gamma\in\widehat{HF}(M)$ non-zero the invariant $\tau_\gamma(K)$ is defined as the minimal $m$ such that $\mathcal F_m^K$ contains a cycle representing $\gamma$, see \cite{OSz-four,AC}.

In \cite{Alfieri}, and more explicitly in \cite[Equation (3.1)]{CM-negative} is shown how to compute the $\tau$-invariant of $K$ in an almost-rational graph $Y$, with respect to a homology class represented by a full path $[W]$ that ends correctly: \begin{equation}\tau_{[W]}(K)=\dfrac{1}{2}\left(-e_1^TQ^{-1}e_1+\min_{[Z]=[W]}e_1^TQ^{-1}Z\right)\:.\label{eq:tau}\end{equation} Here, we prove that the same formula holds also for any star-shaped graph whose corresponding 3-manifold is a rational homology sphere. 
In the proof we use some terminology from the lattice cohomology setting, but in this paper we do not introduce this object explicitly. We refer to the existing literature \cite{Nemethi,Lattice2,BLZ,Zemke,Alfieri} for details.

\begin{teo}
 \label{teo:tau}
 Suppose that $\Gamma$ is the negative-definite standard graph of a Seifert fibred space $Y$, and let $\gamma\in\widehat{HF}_*(-Y,\s)$ be a non-zero homogeneous class. Then \begin{equation*}
    \tau_\gamma(K)=\dfrac{1}{2}\left(\dfrac{1}{e(Y)}+e_1^TQ_*^{-1} V_k\right)=\dfrac{1}{2}\left(\dfrac{1}{e(Y)}+\min_{[Z]=[V_k]}e_1^TQ_*^{-1}Z\right)\:,
\end{equation*} 
where $\gamma=[V_1]+\cdots+[V_k]$ for some $V_i\in\emph{Char}(\Gamma^*)$ initial such that $[V_i]$ ends correctly, and $\mathcal F(V_1)<\cdots<\mathcal F(V_k)$.
\end{teo}

\begin{proof}
 In \cite{BLZ} Borodzik, Liu and Zemke prove that there is an isomorphism between the full completed versions of the link Floer and the lattice cohomology chain complex of a link given by attaching leaves to a plumbing tree, with the assumption that the latter one represents a rational homology sphere. 
 
 In our setting, where we are interested in the regular fibre $K$ of $\Gamma^*$, we just consider the knot case; hence, we refer to \cite[Theorem 5.1]{BLZ} which gives an isomorphism $\mathcal{CFK}(-Y,K)\simeq\mathbb{CFL}(\Gamma^*,S_0)$ of $A_\infty$-modules over the ring $\F[U,V]$, where $S_0$ denotes the leaf attached to the centre and $\mathbb{CFL}(\Gamma^*,S_0)$ is the generalised knot lattice complex. As explained in \cite[Lemma 5.9]{BLZ} this isomorphism respects both the Maslov and the Alexander grading; in particular, the gradings $\text{gr}_w$ and $A$ in \cite{BLZ} corresponds precisely to our definition of $M$ and $\mathcal F$ for any $V\in\text{Char}(\Gamma^*)$. For more details about the algebraic setting we refer to \cite{BLZ}.

 It is important to observe that the isomorphism in \cite[Theorem 5.1]{BLZ} requires the full version of the surgery formula in Heegaard Floer, see \cite{Zemke}; nonetheless, it restricts to the characteristic vectors of $\Gamma^*$, which are the generators in cube grading zero of the complex \cite{Lattice2}. It is possible to check directly, see \cite[Lemma 3.1]{Lattice2}, that if $V_1$ and $V_2$ are related by a step in the full path then their sum is in fact in the image of the differential. Since we now know from Theorem \ref{teo:main} that any $\gamma$ as in the statement can be identified with a linear combination of full paths (ending correctly), we can compute the invariant $\tau_\gamma(K)$ using the isomorphism in \cite{BLZ}: \[\begin{aligned}\tau_\gamma(K)&=\min\big\{m\in\Z\:|\:\gamma\in\Imm i_*:H_*(\mathcal F^K_m\widehat{CF}(-Y))\rightarrow\widehat{HF}_*(-Y,\s)\big\}=\\ &=\min_{[Z]=[V_k]}\mathcal F(Z)=\dfrac{1}{2}\left(-e_1^TQ_*^{-1}e_1+e_1^TQ_*^{-1} V_k\right)\end{aligned}\] because $V_k$ is initial, and $\mathcal F$ increases along the full paths; note that the cycles whose homology class is $\gamma$ are precisely linear combinations of characteristic vectors in the $[V_i]$'s, see \cite{Lattice2,Zemke}. We conclude by observing that $\frac{1}{e(Y)}=-\frac{1}{e(-Y)}=-e_1^TQ_*^{-1}e_1$.
\end{proof}

We now prove Theorem \ref{teo:taus} by showing that there is a combinatorial formula for any non-zero homogeneous homology class in $\widehat{HF}_*(Y,\s)$ when $Y$ has negative-definite standard graph. The formula depends on the parity of the Maslov grading of the class; hence, we need to distinguish the case of $\gamma\in\widehat{HF}_*^\text{ev}(Y,\s)$ and $\delta\in\widehat{HF}_*^\text{od}(Y,\s)$.

\begin{proof}[Proof of Theorem \ref{teo:taus}]
 The first case follows by Equation \eqref{eq:tau}. For the second case, using the formula in \cite[Lemma 2.2]{AC} and Theorem \ref{teo:tau}, we can compute the invariant $\tau_\delta(K)$ where $\delta\in\widehat{HF}_*(Y,\s)$ is the class identified with $T_{[V_1]}+\cdots+T_{[V_k]}$ and $\mathcal F(V_1)<\cdots<\mathcal F(V_K)$:
 \begin{equation}\begin{aligned}
 \tau_\delta(K)&=-\min\{\tau_\gamma(K)\:|\:\gamma\in\widehat{HF}_*(-Y,\s)\text{ and }\langle\gamma,\delta\rangle=1\}=\\ &=-\dfrac{1}{2}\left(\dfrac{1}{e(Y)}+e_1^TQ_*^{-1} V_1\right)=\dfrac{1}{2}\left(\dfrac{1}{-e(Y)}+\max_{[Z]=[-V_1]}e_1^TQ_*^{-1}Z\right)\:. \end{aligned}
    \label{eq:taumirror}
 \end{equation}
\end{proof}

We can show that the Alexander filtration $\mathcal F$ corresponds precisely to the filtration $\mathcal F^K$ in knot Floer homology. 
\begin{cor}
 \label{cor:filtration}
  For every integer $m$ the subgroup of $\widehat{HF}(Y)$ generated by the vectors $[W_i]$ for $i=1,...,t+1$ and the functionals $T_{[V_j]}$ for $j=1,...,t$ such that $\mathcal F(\cdot)\leq m$ coincides with the image of $i_*:H_*(\mathcal F^K_m\widehat{CF}(Y))\rightarrow\widehat{HF}(Y)$, the $m$-level of the filtration induced by $K$.  
\end{cor}
\begin{proof}
 It follows by Theorem \ref{teo:tau} and Equation \eqref{eq:taumirror} because now we can compute the $\tau$-invariant of $K$ with respect to every homogeneous class $\alpha$ in $\widehat{HF}(Y)$, and by definition $\tau_\alpha(K)$ is the minimal $m$ such that $\mathcal F^K_m\widehat{HF}(Y)$ contains $\alpha$. 
\end{proof}

\subsection{The height function}
\label{subsection:height}
In \cite[Definition 2.6]{CM-negative} we introduced an invariant called the \emph{height function} of a full path $[W]$ which ends correctly as $-e_1^TQ_G^{-1}\frac{W+W'}{2}$, where $W$ and $W'$ are the initial vectors of $[W]$ and $[-W]$ respectively; moreover, from \cite[Subsection 2.2]{CM-negative} the height of $[W]$ also coincides with the number of central steps, that are steps in a full path taken when $w_1=-m(1)$. We now proceed to define a general version of this invariant.
\begin{defin}
 \label{def}
  Assume $M$ is a Seifert fibred space with $b_1(M)=0$. Let $\gamma\in\widehat{HF}(M,\s)$ be a homology class which can be written as a linear combination of full paths $[Z_1],...,[Z_k]$ ending correctly, and let the $Z_i$'s be the initial vectors. Then the height of $\gamma$ is \begin{equation*}
 \height(\gamma)=\max_{A\in[Z_1]\cup\dots\cup[Z_k]}\mathcal F(A)\:-\min_{B\in [Z_1]\cup\dots\cup[Z_k]}\mathcal F(B) \:,
 \label{eq:height}
\end{equation*}
 that is the difference between the maximal and the minimal value of the Alexander filtration, induced by the leaf $K$, on all the full paths which are components of $\gamma$. 
 
 In other words, if $Z_1',...,Z_k'$ are the initial vectors of $[-Z_1],...,[-Z_k]$ while the $Z_i$'s are ordered in the way that $\mathcal F(Z_1)<\cdots<\mathcal F(Z_k)$, then \[\height(\gamma)=\mathcal F(-Z_k')-\mathcal F(Z_1)=-e_1^TQ_G^{-1}\left(\dfrac{Z_1+Z_k'}{2}\right)\:.\]
\end{defin}
Note that the initial vector $V'$ of the full path $[-V]$ is equal to $-\overline V$, where $\overline V$ is the terminal vector of $[V]$.
Let us consider $Z_1$ and $-Z_k'=\overline Z_k$; since they restrict to the same $\Spin^c$-structure on $M$, we know from Proposition \ref{prop:spinc} that there exists an integral vector $B$ such that $2Q_GB=Z_1+Z_k'$; therefore, we have the following corollary.
\begin{cor}
 The height of $\gamma$ can be expressed without involving the Alexander filtration as $\height(\gamma)=|b_1|$, where $B=(b_1,...,b_{|G|})$ is the vector defined above.  
\end{cor}
\begin{proof}
 We have that \[\height(\gamma)=-e_1^TQ_G^{-1}\left(\dfrac{Z_1+Z_k'}{2}\right)=-e_1^TB=-b_1\:,\] and this number is non-negative by definition.  
\end{proof}

\subsection{Blowing down the graph}
When $\Gamma$ is negative-definite there is a unique way to blow-down the graph. This is no longer true for $\Gamma^*$, but in this paper we will often consider a special 4-manifold $X$ obtained in such a way. We recall that if $e_0\geq0$ then the corresponding Seifert fibred manifold is an $L$-space. 
\begin{lemma}
 \label{lemma:blowdown}
 Suppose that $\Gamma$ is the negative-definite standard graph of a Seifert fibred space, and every maximal $S^3$-subgraph $(\Gamma^*)'$ represents the fibration of $T_{d_2,d_1}$ for $1\leq d_2\leq d_1$ fixed coprime integers. Then there is a canonical way to blow-down $\Gamma^*$ completely so that $(\Gamma^*)'$ disappears; furthermore, the resulting $4$-manifold $X_{\Gamma^*}$ satisfies $b_2^+(X_{\Gamma^*})=1$. 
\end{lemma}

\begin{proof}
 The framings on each leg of $\Gamma^*$ are given by the continued fraction expansion of $r_i$, which means that they are all not bigger than $-2$. If $e_0<-1$ then the maximal $S^3$-subgraph is empty, and there is nothing to prove, if $e_0=-1$ then we start precisely by blowing down the central vertex and we continue by always blowing down the $-1$-circle in the leg of $(\Gamma^*)'$ with the highest coefficient (at that step). This is possible because $(\Gamma^*)'$ represents a torus knot.

 As described in \cite[Section 5]{CM-negative} the manifold $X_{\Gamma^*}$ is obtained by attaching 2-handles on an appropriate cable of a torus link. Since we are always blowing down a circle with framing $-1$, we have that $b_2^+$ is preserved. The subgraph $(\Gamma^*)'$ is negative-definite thus $b_2^+(X_{\Gamma^*})$ is the same as the one of $\Gamma^*$ without the legs containing $(\Gamma^*)'$. 
\end{proof}

When a star-shaped graph $G$ is negative-definite the manifold $X_{G}$ is a Stein domain, see \cite[Section 5]{CM-negative}. When $G$ is indefinite and it either contains the standard graph of $T_{d_2,-d_1}$ or $G'$ is not unique, we are going to show in Corollary \ref{cor:negative} that $Y_G$ is an $L$-space.

\begin{remark}
 \label{remark:blowdown}
 It follows by a result of Lisca and Mati\'c \cite{LM} that if $-Y$, the Seifert fibred space whose standard graph is $\Gamma^*$, is not an $L$-space then the manifold $X_{\Gamma^*}$ described by the procedure in Lemma \ref{lemma:blowdown} is obtained by attaching Stein $2$-handles to $D^4$. 
\end{remark}

Note that the converse of this remark does not hold: take the standard graph of $M(-1;\frac{2}{3},\frac{1}{2},\frac{1}{3})\cong S^3_{-9}(T_{2,-3})$, we have that $X_{\Gamma^*}$ is a Stein domain even though this 3-manifold is an $L$-space. However, all of its structures are zero-twisting and thus they do not concern the topic of this paper.

\begin{remark}
 \label{remark:L}
 By \cite{LS} a Seifert fibred $L$-space whose standard graph has central vertex with framing $e_0\geq-1$ admits no negative-twisting structures.
\end{remark}

Using the same notation as in \cite{CM-negative}, we recall that given a plumbing tree $G$ the characteristic vector $V_\text{can}=(m(1)+2,...,m(|G|)+2)\in\text{Char}(G)$ is called \emph{canonical vector}.

\begin{teo}[$L$-space criterion]
 \label{teo:criterion}
 A Seifert fibred space with $b_1=0$ is an $L$-space if and only if, when it is oriented in the way that its standard graph is indefinite, the canonical vector $V_\emph{can}$ does not end correctly. Furthermore, in this case one has $e_0\geq-1$.
\end{teo}
\begin{proof}
 From Theorem \ref{teo:fullpath} we know that a Seifert fibred space with indefinite standard graph $\Gamma^*$ is an $L$-space if and only if $\Gamma^*$ has no characteristic vector whose full path ends correctly. Since $V_\text{can}$ is the vector with minimal coordinates among the ones that initiate their full path, if there exists a $V\in\text{Char}(\Gamma^*)$ whose full path ends correctly then the same necessarily holds for $V_\text{can}$ too.      
\end{proof}

The following result follows from \cite[Theorem 1.3]{LM-L}, but it can also be seen as a consequence of the $L$-space criterion.

\begin{cor}
 \label{cor:negative}
 If the maximal $S^3$-subgraph of $G$ either corresponds to the fibration of a negative-torus knot or is not unique, then the underlying Seifert fibred space is an $L$-space.   
\end{cor}
\begin{proof}
 It is easy to check that under the given assumptions $V_\text{can}$ cannot end correctly; thus we just conclude by applying Theorem \ref{teo:criterion}.
\end{proof}

We prove the following property of $V_\text{can}$ which we are going to apply in Sections \ref{section:four} and \ref{section:seven}.

\begin{lemma}
 \label{lemma:can}
 Suppose that $Y$ is a Seifert fibred space with standard graph $\Gamma$ and $b_1(Y)=0$. If $\Gamma$ is negative-definite then $\mathcal F(W_\emph{can})>\mathcal F(W)$ for any initial $W\in\emph{Char}(\Gamma)$ different from $W_\emph{can}$. Similarly, for its dual $\Gamma^*$ one has $\mathcal F(V_\emph{can})<\mathcal F(V)$ for any $V\in\emph{Char}(\Gamma^*)$ with the same properties.    
\end{lemma}
\begin{proof}
 According to the definition of $\mathcal F$, we just have to prove that $e_1^TQ^{-1}W_\text{can}>e_1^TQ^{-1}W$ and $e_1^TQ_*^{-1}V_\text{can}<e_1^TQ_*^{-1}V$. Consider the generic intersection matrix \[
Q_G=
\begin{pmatrix}
e_0 & \epsilon^T\\
\epsilon & A
\end{pmatrix}\]
for any standard graph $G$, where $A$ is the matrix made by the tridiagonal blocks $A_1,...,A_n$ corresponding to the legs of $G$, and $\epsilon$ is the incidence vector of the centre whose entries are either zero or one. Since each $A_i$ is the intersection matrix of a chain of unknots, it is an irreducible negative-definite $M$-matrix; hence, each $A_i^{-1}$ has negative entries.

Now take the Schur complement of $A$ expressed by $\Delta:=e_0-\epsilon^T A^{-1}\epsilon$: since $\det Q_G\neq0$, one has $\Delta\neq 0$ and the block inverse formula gives
\[e_1^TQ_G^{-1}=\Delta^{-1}(1,\:-\epsilon^T A^{-1})=\Delta^{-1}(1,-e_1^TA_1^{-1},...,-e_1^TA_n^{-1})\:.\] Then the row vector $-\epsilon^T A^{-1}$ has positive entries, thus every entry of $e_1^TQ_G^{-1}$ has the same sign as $\Delta$. By standard linear algebra, we have the congruence
\[Q_G\sim
\begin{pmatrix}
\Delta & 0\\
0 & A
\end{pmatrix}\] and then \[(b_2^+(P_G),b_2^-(P_G),0)=(0,|G|-1,0)+\Big(\frac{1+\text{sgn}(\Delta)}{2},\frac{1-\text{sgn}(\Delta)}{2},0\Big)\:.\]
This implies that
\begin{itemize}
    \item if $G$ is negative-definite then $\Delta<0$, and then every entry of $e_1^TQ_G^{-1}$ is negative;
    \item if $G$ is indefinite then $\Delta>0$, and then every entry of $e_1^TQ_G^{-1}$ is positive.
\end{itemize}
Now let $Z=(z_1,...,z_{|G|})$ be any initial characteristic vector in $G$; since $z_i\geq m(i)+2$ for each $i$, we have that $Z-Z_\text{can}\neq0$ has non-negative coordinates. Therefore, we conclude that in the case of $\Gamma$ one has $e_1^TQ^{-1}(W-W_\text{can})<0$, while in the case of $\Gamma^*$ one has $e_1^TQ_*^{-1}(V-V_{\text{can}})>0$. This proves the claim. 
\end{proof}

\section{The classification when \texorpdfstring{$b_1=1$}{e(Y)=0}}
\label{section:four}
Let $Z$ be a Seifert fibred space whose standard graph $G$ has at least three legs, and such that $b_1(Z)=1$; hence, we have $e(Z)=e_0+r_1+\cdots +r_n=0$. These manifolds are surface bundles, obtained by capping off a regular fibre in a Seifert fibred space by performing $0$-surgery on it. We are going to classify all the negative-twisting structures on $Z$.

We recall that for torus bundles the classification follows by the results of Honda \cite{Honda2} and Giroux \cite{Giroux1}; namely, there is a unique Stein fillable structure, together with infinitely many structures obtained by adding (non-vertical) Giroux torsion along the torus. These structures are weakly but not strongly fillable, see \cite[Theorem 1]{DG}.

\begin{teo}
 \label{teo:classification0}
 Let $Z=M(e_0;r_1,\dots,r_n)$ be a genus $g>1$ surface bundle obtained by $0$-surgery on a Seifert fibred knot in a rational homology sphere. The negative-twisting tight contact structures on $Z$ are precisely the Legendrian surgeries on all the Legendrian realisations of the blow-down of the standard graph; moreover, they are all Stein fillable and distinguished, up to isotopy, by their contact invariant $c^+$ in $HF^+(-Z)$.  
\end{teo}

We briefly describe why the 4-manifold $X_G$ is a Stein domain; in particular, these manifolds always have at least one Stein fillable structure. We use Lemma \ref{lemma:blowdown} after we modify $G$ by lowering the framing on a vertex $S_i$ outside of an $S^3$-subgraph $G'$ (we have $n\geq3$); in this way, we obtain a negative-definite graph $\widetilde G$ such that $X_{\widetilde G}$ is a Stein domain, which differs from $X_G$ by having framing (say $k$) lowered by one on a single Legendrian knot $\mathcal K_i\subset(S^3,\xist)$. If either $S_i$ can be chosen to be not connected to the centre of $G$, or $G'$ represents $T_{d_2,d_1}$ with $d_2>1$, then $k<\text{TB}_{\xist}(K_i)$ because either the knot type $K_i$ is an unknot, thus $k\leq-3$, or $K_i$ is a positive torus knot, thus $\text{TB}_{\xist}(K_i)>0$. In the remaining case when $G'$ represents $T_{1,d_1}$, the framing on $S_i$ should be low enough for this vertex not to be blown down together with $G'$, see Figure \ref{Graph} for an example.

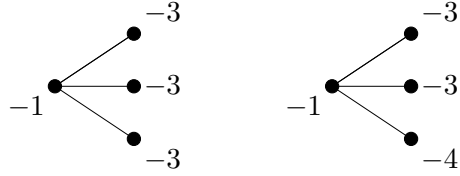
\begin{figure}[ht] 
     \begin{tikzpicture}[scale=0.7]
    \tkzDefPoints{0/0/A, 1.5/1/B, 1.5/-1/D, 1.5/0/C}  
    \tkzDrawSegment(A,B)\tkzDrawSegment(A,D)\tkzDrawSegment(B,A)\tkzDrawSegment(A,C)
    \tkzDrawPoints[fill,black,size=5](A,B,C,D)
     \tkzLabelPoint[below left](A){$-1$} 
     \tkzLabelPoint[above right](B){$-3$}\tkzLabelPoint[right](C){$-3$}
     \tkzLabelPoint[below right](D){$-3$}
       \end{tikzpicture}\hspace{1cm} 
       \begin{tikzpicture}[scale=0.7]
    \tkzDefPoints{0/0/A, 1.5/1/B, 1.5/-1/D, 1.5/0/C}  
    \tkzDrawSegment(A,B)\tkzDrawSegment(A,D)\tkzDrawSegment(B,A)\tkzDrawSegment(A,C)
    \tkzDrawPoints[fill,black,size=5](A,B,C,D)
     \tkzLabelPoint[below left](A){$-1$} 
     \tkzLabelPoint[above right](B){$-3$}\tkzLabelPoint[right](C){$-3$}\tkzLabelPoint[below right](D){$-4$}
       \end{tikzpicture}
     \caption{\smaller[1]{The standard graph $G$ of $Z=M(-1;\frac{1}{3},\frac{1}{3},\frac{1}{3})$ (left), and the one of $Y_{\widetilde G}=M(-1;\frac{1}{3},\frac{1}{3},\frac{1}{4})$ (right). The framing on the lower vertex can be at most $-3$; otherwise, the vertex would be inside $G'$.}}
     \label{Graph}
\end{figure} 

The above argument shows that when $e(Z)=0$ (and $n\geq3$) the canonical vector $V_\text{can}$ ends correctly. Therefore, the maximal $S^3$-subgraph of $G$ is unique, because as said in Corollary \ref{cor:negative} this would not be possible otherwise.

We mentioned that a manifold $Z=M(e_0;r_1,...,r_n)$ as above is a genus $g\geq1$ surface bundle when $n\geq3$. Let $K_i$ be a regular fibre of the Seifert fibration presented by any graph $\Gamma_i$, obtained by removing the $i$-th leg from $G$. It is a classical result \cite[Subsection 2.1]{Hatcher} that a surface bundle which is Seifert fibred over a sphere has fibre of genus $g=1$
if and only if its standard graph is one of the seven in Figure \ref{Torus}, and then the structures with Giroux torsion from \cite[Theorem 1]{DG} appear only in these cases.

\begin{figure}[b] 
       \begin{tikzpicture}[scale=0.8]
    \tkzDefPoints{0/0/A, 1.5/1/B, 1.5/-1/D, 1.5/0/C, 3/1/E, 3/0/F, 3/-1/G}  
    \tkzDrawSegment(A,B)\tkzDrawSegment(A,D)\tkzDrawSegment(B,A)\tkzDrawSegment(A,C)\tkzDrawSegment(B,E)\tkzDrawSegment(F,C)\tkzDrawSegment(D,G)
    \tkzDrawPoints[fill,black,size=5](A,B,C,D,E,F,G)
     \tkzLabelPoint[below left](A){$-2$} 
     \tkzLabelPoint[above right](B){$-2$}\tkzLabelPoint[above right](C){$-2$}\tkzLabelPoint[above right](E){$-2$}\tkzLabelPoint[above right](F){$-2$}
     \tkzLabelPoint[below right](D){$-2$}\tkzLabelPoint[below right](G){$-2$}
       \end{tikzpicture}
        \begin{tikzpicture}[scale=0.8]
    \tkzDefPoints{0/0/A, 1.5/1/B, 1.5/-1/D, 1.5/0/C, 4.5/0/E, 3/0/F, 3/-1/G, 4.5/-1/H}  
    \tkzDrawSegment(A,B)\tkzDrawSegment(A,D)\tkzDrawSegment(B,A)\tkzDrawSegment(A,C)\tkzDrawSegment(F,E)\tkzDrawSegment(F,C)\tkzDrawSegment(D,G)\tkzDrawSegment(H,G)
    \tkzDrawPoints[fill,black,size=5](A,B,C,D,E,F,G,H)
     \tkzLabelPoint[below left](A){$-2$} 
     \tkzLabelPoint[above right](B){$-2$}\tkzLabelPoint[above right](C){$-2$}\tkzLabelPoint[above right](E){$-2$}\tkzLabelPoint[above right](F){$-2$}
     \tkzLabelPoint[below right](D){$-2$}\tkzLabelPoint[below right](G){$-2$}\tkzLabelPoint[below right](H){$-2$}
       \end{tikzpicture}
       \begin{tikzpicture}[scale=0.7]
    \tkzDefPoints{0/0/A, 1.5/1/B, 1.5/-1/D, 1.5/0/C, 3/-1/E, 4.5/-1/F, 6/-1/G, 7.5/-1/H, 3/0/I}  
    \tkzDrawSegment(A,B)\tkzDrawSegment(A,D)\tkzDrawSegment(B,A)\tkzDrawSegment(A,C)\tkzDrawSegment(D,E)\tkzDrawSegment(E,F)\tkzDrawSegment(F,G)\tkzDrawSegment(G,H)\tkzDrawSegment(C,I)
    \tkzDrawPoints[fill,black,size=5](A,B,C,D,E,F,G,H,I)
     \tkzLabelPoint[below left](A){$-2$} 
     \tkzLabelPoint[above right](B){$-2$}\tkzLabelPoint[above right](C){$-2$}\tkzLabelPoint[above right](I){$-2$}
     \tkzLabelPoint[below right](D){$-2$}\tkzLabelPoint[below right](E){$-2$}\tkzLabelPoint[below right](F){$-2$}\tkzLabelPoint[below right](G){$-2$}\tkzLabelPoint[below right](H){$-2$}
       \end{tikzpicture}       

     \begin{tikzpicture}[scale=0.9]
    \tkzDefPoints{0/0/A, 1.5/1/B, 1.5/-1/D, 1.5/0/C}  
    \tkzDrawSegment(A,B)\tkzDrawSegment(A,D)\tkzDrawSegment(B,A)\tkzDrawSegment(A,C)
    \tkzDrawPoints[fill,black,size=5](A,B,C,D)
     \tkzLabelPoint[below left](A){$-1$} 
     \tkzLabelPoint[above right](B){$-3$}\tkzLabelPoint[right](C){$-3$}
     \tkzLabelPoint[below right](D){$-3$}
       \end{tikzpicture}\hspace{1cm}
       \begin{tikzpicture}[scale=0.9]
    \tkzDefPoints{0/0/A, 1.5/1/B, 1.5/-1/D, 1.5/0/C}  
    \tkzDrawSegment(A,B)\tkzDrawSegment(A,D)\tkzDrawSegment(B,A)\tkzDrawSegment(A,C)
    \tkzDrawPoints[fill,black,size=5](A,B,C,D)
     \tkzLabelPoint[below left](A){$-1$} 
     \tkzLabelPoint[above right](B){$-2$}\tkzLabelPoint[right](C){$-4$}
     \tkzLabelPoint[below right](D){$-4$}
       \end{tikzpicture}\hspace{2cm}
       \begin{tikzpicture}[scale=0.9]
    \tkzDefPoints{0/0/A, 1.5/1/B, 1.5/-1/D, 1.5/0/C}  
    \tkzDrawSegment(A,B)\tkzDrawSegment(A,D)\tkzDrawSegment(B,A)\tkzDrawSegment(A,C)
    \tkzDrawPoints[fill,black,size=5](A,B,C,D)
     \tkzLabelPoint[below left](A){$-1$} 
     \tkzLabelPoint[above right](B){$-2$}\tkzLabelPoint[right](C){$-3$}
     \tkzLabelPoint[below right](D){$-6$}
       \end{tikzpicture}\hspace{1cm}
      \begin{tikzpicture}[scale=0.8] 
      \tkzDefPoints{0/0/A, 1.5/1.5/B, 1.5/-1.5/D, 1.5/0.5/C, 1.5/-0.5/G}      
       \tkzDrawSegment(A,B)\tkzDrawSegment(A,D)\tkzDrawSegment(A,C)\tkzDrawSegment(A,G)
    \tkzDrawPoints[fill,black,size=5](A,B,C,D,G)
     \tkzLabelPoint[above left](A){$-2$} \tkzLabelPoint[below right](B){$-2$}\tkzLabelPoint[below right](C){$-2$}
     \tkzLabelPoint[below right](D){$-2$}\tkzLabelPoint[below right](G){$-2$}
\end{tikzpicture}
       \caption{\smaller[1]{All the possible standard graphs corresponding to torus bundles over $S^1$.}}
     \label{Torus}\end{figure}
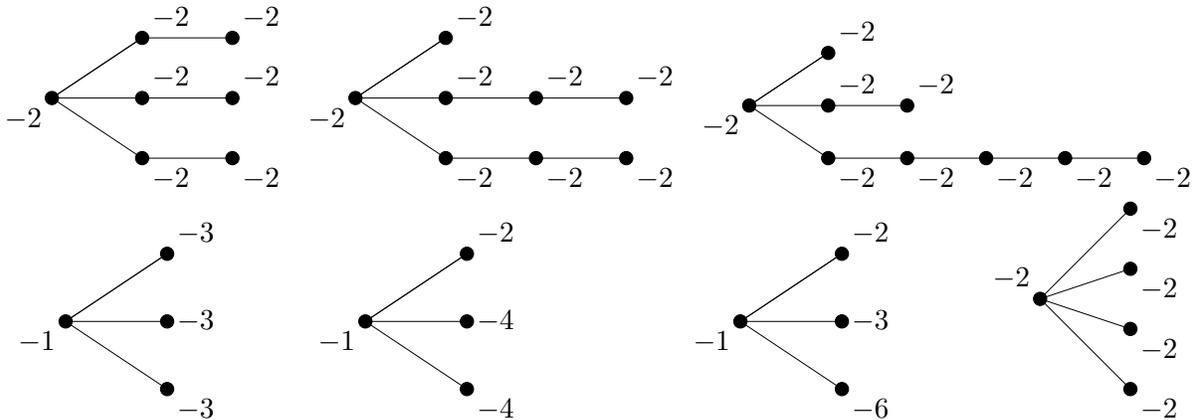

We can show that $g$ is also determined combinatorially. Note that since $e(Z)=e_0+r_1+\cdots+r_n=0$, we have that $r_i=-e(\Gamma_i)$ and $\Gamma_i$ is negative-definite; therefore, the negative rational number $\frac{1}{e(\Gamma_i)}$ coincides with $[D_i]\cdot[D_i]$, the self-intersection of the disk $D_i\subset P_{\Gamma_i}$ whose boundary is $K_i$. 
\begin{prop}
 \label{prop:regular}
 Let $\Gamma$ be the negative-definite standard graph of a Seifert fibred space, and $K$ be the regular fibre. Then we have that $g_3(Y_\Gamma,K)=\tau_{\xi_{\emph{can}}}(K)$ where $\xi_{\emph{can}}$ is the contact structure on $Y_\Gamma$ induced by a Milnor fibre. In particular, one has \[g_3(Y_\Gamma,K)=\dfrac{1}{2}\left(\dfrac{1}{-e(\Gamma)}+e_1^TQ_\Gamma^{-1}W_{\emph{can}}\right)\] where $Q_\Gamma$ is the intersection matrix of $\Gamma$.
\end{prop}
\begin{proof}
 A result of Pichon and Seade \cite[Theorem 4.1]{PS} says, in particular, that the link $K_N$ given as $N$ leaves attached to the centre of $\Gamma$ is fibred when $N$ is big enough; therefore, as $K_N$ is a positive cable of the regular fibre $K=K_1$, we have that the latter knot is rationally fibred in $Y_\Gamma$. We know from \cite[Theorem 5.1]{BEvHM} that $K$ supports the same structure $\xi$ of $K_N$; moreover, since $K_N$ consists of parallel copies of $K$, then $\xi$ is transverse to the Seifert fibration of $Y_G$. From a result of Lisca and Mati\'c \cite[Corollary 2.2]{LM-L} we can argue that $\xi$ is fillable, and according to our classification \cite[Theorem 1.1 and Proposition 5.2]{CM-negative}, that is actually Stein fillable and its contact invariant is $T_{[W]}\in\widehat{HF}(-Y_\Gamma)$ for some initial $W\in\text{Char}(Y_\Gamma)$ whose full path ends correctly. 
 
 Consider $o$ the order of $[K]\in H_1(Y_\Gamma;\Z)$, and $S$ a page of the corresponding open book. The value $\frac{1-\chi(S)/o}{2}$ is equal to the (rational) Seifert genus $g_3(Y_\Gamma,K)$, and this can be computed combining the definition of self-linking number (of a rationally null-homologous knot) with \cite[Theorem 2.1]{Cavallo-B}, and using the Hedden-Plamenevskaya bound in \cite[Theorem 3.2]{CM}. In fact, we have that the rational self-linking formula gives
 \[-\chi(S)/o=\self_\xi(K)\leq2\tau_\xi(K)-1\leq2g_3(Y_\Gamma,K)-1=-\chi(S)/o\:,\]
 forcing the inequalities to be sharp. More specifically, we can write $\tau_{\xi}(K)=\frac{1-\chi(S)/o}{2}=g_3(Y_\Gamma,K)$ which can be determined using Alfieri's formula for the $\tau$-invariant, that is $\tau_{\xi}(K)=\frac{1}{2}(\frac{1}{-e(\Gamma)}+e_1^TQ_\Gamma^{-1}W)$, see Equation \eqref{eq:tau}.
 
 Since the rational Seifert genus is the maximal non-zero value for the Alexander filtration induced by $K$, we have that $\mathcal F(W)$ should be maximal among initial vectors that end correctly. Using Lemma \ref{lemma:can}, we conclude that $W=W_\text{can}$ and then $\xi=\xi_\text{can}$ by \cite[Corollary 5.5]{BP}. 
\end{proof}

We can now give a condition that allows us to determine the genus $g$ of a fibre of any surface bundle $Z$. 
We have that $g_3(Y_{\Gamma_i},K_i)$ is equal to $\frac{1}{2}-\frac{\chi(\Sigma)}{2o}$, where $o$ is the order of $[K_i]\in H_1(Y_{\Gamma_i};\Z)$ and $\Sigma$ is the rational Seifert surface which realises its Thurston norm, see \cite{Cavallo-B} for more details. 
If $\Sigma$ has genus $g$ then $\chi(\Sigma)=2-\frac{o}{b}-2g$, where we set $-e(\Gamma_i)=r_i=\frac{b}{a}\in\Q_{>0}$ and then $\frac{o}{b}$ is precisely $|\Sigma|$, and this happens if and only if any graph $\Gamma_i$ satisfies \begin{equation}e_1^TQ_{\Gamma_i}^{-1}W_{\text{can}}=\dfrac{1+b-a}{b}+\dfrac{2}{o}(g-1)\end{equation} by Proposition \ref{prop:regular}. A posteriori the value of $g$ in this equation does not depend on the leg $i$. 

\subsection{The Maslov grading: a special case}
In the previous section we computed the Maslov grading of a non-zero homology class $[V]\in HF^-(M,\s)$, represented by a characteristic vector, when $M$ is a Seifert fibred space with $b_1(M)=0$; namely, we have that $M(V)=\frac{V^TQ^{-1}V+|G|-\epsilon}{4}$ where $\epsilon$ is zero if $G$ is negative-definite, while it is $6$ if $G$ is indefinite.

We know from Heegaard Floer theory that the Maslov grading is a well-defined absolute $\Z$-grading also when $b_1(M)=1$ and $c_1(\s)$ is torsion. We show that there exists an analogous formula for $M$ also in this case.

\begin{prop}
 Let $(Z,\mathfrak t)$ be a Seifert fibred space with $b_1(Z)=1$ and $\mathfrak t$ a torsion $\Spin^c$-structure; moreover, take a $V\in\emph{Char}(G,\mathfrak t)$ such that $[V]\in HF^-(Z,\mathfrak t)$ is non-zero. Then \[M(V)=\frac{B^TQB+|G|-3}{4}\] where $B\in\Q^{|G|}$ is any vector such that $QB=V$.   
\end{prop}
\begin{proof}
 The fact that at least one vector $B$ exists is equivalent to the fact that $\mathfrak t$ is torsion; in fact, as a class in $H_1(Z;\Q)$ one has that $\mathfrak t$ is equivalent to a spin structure, and we conclude as in the proof of Proposition \ref{prop:spinc}. It is also easy to check that if $B_1$ and $B_2$ are as in the statement then the expression for $M$ is well-defined: \[0=(B_1+B_2)^TQ(B_1-B_2)=B_1^TQB_1-B_2^TQB_2+(B_2^TQB_1-B_1^TQB_2)=B_1^TQB_1-B_2^TQB_2\] because $B_1-B_2\in\Ker Q$ and $Q$ is symmetric.

 Using Equation \eqref{eq:Maslov} and Corollary \ref{cor:even}, we only have to check that $c_1^2(\mathfrak u)[P_G]=B^TQB$ where $V=c_1(\mathfrak u)$. This follows from the standard fact that $c_1^2(\mathfrak u)[P_G]=A^TQA$ where $A$ is a vector that represents $c_1(\mathfrak u)$ in the basis $\{S_1,...,S_{|G|}\}$ of $H_2(P_G;\Z)$ given by the vertices; hence, we have that $A$ satisfies $QA=V$ because $V$ represents the Chern class in the dual basis $\{D_1,...,D_{|G|}\}$ of $H_2(P_G,Z;\Z)$ given by the dual of the vertices. 
\end{proof}

\subsection{The maximal twisting number}
We prove the equivalent of \cite[Proposition 4.2]{CM-negative}, showing that only one negative twisting number of a tight structure exists on $Z$ unless it is a torus bundle. 

\begin{prop}
 \label{prop:tw}
 If $Z=M(e_0;r_1,\dots,r_n)$ satisfies $b_1(Z)=1$, then all the negative-twisting tight contact structures on $Z$ without Giroux torsion have the same twisting number.
\end{prop}
\begin{proof}
 When $g=1$ the claim is obvious from Honda's classification \cite{Honda2}, because there is only one (Stein fillable) structure without Giroux torsion. Below we show that for $g>1$ there is a unique negative twisting number. 

 We proceed as in the proof of \cite[Proposition 4.2]{CM-negative}. We know from Ghiggini-Massot's algorithm that $-q<-1$ is a twisting number if and only if there exist positive integers $p_1,...,p_n$ such that $\frac{p_i}{q}$ is the best upper approximation for $r_i$ and $p_1+\cdots+p_n=-e_0q+n-2$.

 Suppose first that $e_0\leq-2$, then $-1$ is a twisting number and we are proving that no other $-q$ can be. Since $\frac{p_i-1}{q-1}< \frac{p_i}{q}<1$, we have that $\frac{p_i-1}{q-1}\leq r_i < \frac{p_i}{q}$ and hence \[r_1+\cdots+r_n\geq\frac{p_1-1}{q-1}+\frac{p_2-1}{q-1}+\cdots+\frac{p_n-1}{q-1}=\frac{-e_0q-2}{q-1}\geq -e_0\:.\] Since we have $e(Z)=0$, the above inequalities can be fulfilled only for $e_0=-2$ and $r_i=\frac{p_i-1}{q-1}$; as usual we write $r_i=\frac{b_i}{a_i}$ for the corresponding reduced fraction. This means in particular that $a_i<q$ and hence $r_i$ and its best upper approximations $\frac{p_i}{q}$ are Farey neighbours which gives the relation: \[1=a_ip_i-b_iq=b_i(q-1)+a_i-b_iq=a_i-b_i\:.\] This now leads to $r_i=1-\frac{1}{a_i}$ for $i=1,...,n$ and after summing up to
 \[\dfrac{1}{a_1}+\cdots+\dfrac{1}{a_n}=n-2\:.\] Since $a_i\geq2$, we have 
 $\frac{1}{a_i}+\cdots\frac{1}{a_n}\leq\frac{n}{2}$; so $n-2\leq\frac{n}{2}$ implies $n\leq4$. 
 Clearly, when $n=4$ we need to have $r_i=\frac{1}{2}$ for each $i$, while when $n=3$ the only possibilities for $(a_1,a_2,a_3)$ are $(3,3,3),$ $(2,4,4)$ and $(2,3,6)$. A quick computation of the $b_i$'s allows us to conclude that a $q>1$ can only exist in the case of torus bundles, where the other twisting numbers are obtained adding Giroux torsion.

 Now suppose that $e_0=-1$ and let $-q$ be the highest negative twisting number on $Z$. To prove that no other $-Q<-q$ works, we follow the same steps as above. To the contrary we assume that $Q>q$ gives best upper approximations $\frac{P_i}{Q}$ of $r_i$ and satisfy $P_1+\cdots+P_n=Q+n-2$. Then since $\frac{P_i}{Q}<\frac{p_i}{q}$ we have $\frac{P_i-p_i}{Q-q}\leq r_i < \frac{P_i}{Q}$ which implies \[r_1+\cdots+r_n\geq\frac{P_1-p_1}{Q-q}+\frac{P_2-p_2}{Q-q}+\cdots+\frac{P_n-p_n}{Q-q}=\frac{(Q+n-2)-(q+n-2)}{Q-q}=1.\] Again, as $e(Z)=0$, this forces $r_i=\frac{b_i}{a_i}=\frac{P_i-p_i}{Q-q}$ and because $a_i<Q$, the pairs $r_i$ and $\frac{P_i}{Q}$ are Farey neighbours. From \[1=a_iP_i-b_iQ=b_i(Q-q)+a_ip_i-b_iQ=a_ip_i-b_iq\:\] we see that then also $r_i$ and $\frac{p_i}{q}$ are Farey neighbours and we can write $r_i=\frac{p_i}{q}-\frac{1}{a_iq}$. Inserting into $e(Z)=0$ and knowing that $p_1+\cdots+p_n=q+n-2$, we get exactly the equality $\frac{1}{a_1}+\cdots+\frac{1}{a_n}=n-2$ as above. When $n=4$ we cannot at the same time satisfy $\frac{1}{a_1}+\frac{1}{a_2}+\frac{1}{a_3}+\frac{1}{a_4}=2$ and $\frac{b_1}{a_1}+\frac{b_2}{a_2}+\frac{b_3}{a_3}+\frac{b_4}{a_4}=1$, while when $n=3$ the only possibilities for $(a_1,a_2,a_3)$ are $(3,3,3),$ $(2,4,4)$ and $(2,3,6)$ as before with all $b_i$'s equal to $1$. The latter ones again correspond to torus bundles where the other twisting numbers are obtained adding Giroux torsion.
\end{proof}

The twisting number is determined by the height function, as in the negative-definite case.
Note that since blowing down $G$ yields the Stein domain $X_G$, for every $Z$ we have that $T_{[V_\text{can}]}$ is the contact invariant of a fillable structure on $Z$.

For this reason $G$ necessarily has full paths which have an initial and a terminal vector, as if all of them were having loops in them then some coordinates would satisfy $v_i=\pm m(i)$ at every step, while $V_\text{can}$ ends correctly and is not of this form. Note that this is not the case with $S^1\times S^2$, whose unique full path ending correctly is always a loop; in fact, the manifold $X_G\simeq S^2\times D^2$, obtained by blowing down any of its standard graphs (all of them have two legs), is not a Stein filling of $S^1\times S^2$.

\begin{teo}
 \label{teo:tw}
 We have that $\emph{tw}(Z,\xi)=-1-\height[V_\emph{can}]$ for every negative-twisting structure $\xi$ on $Z$ without Giroux torsion. This number is either equal to $-d_1-d_2$ when the $S^3$-subgraph $G'$ of $G$ is the Seifert fibration of $T_{d_2,d_1}$ with $1\leq d_2\leq d_1$ coprime, or to $-1$ when $G'$ is empty. 
\end{teo}
\begin{proof}
 From the proof of Proposition \ref{prop:tw} and \cite[Proposition 4.2]{CM-negative} we have that $\text{tw}(Z,\xi)$ is equal to the maximal twisting number of the fibration of $T_{d_2,d_1}$ in $S^3$, and the latter one is $-d_1-d_2$ as we know from \cite[Theorem 4.3]{CM-negative}. In addition, in the proof of the latter result is explained that $1-d_1-d_2=\height[W_\text{can}]$ where $W_\text{can}$ is the canonical vector of $G'$; the fact that $\height[V_\text{can}]=\height[W_\text{can}]$ can be proved in the same way as \cite[Remark 4.4]{CM-negative}.

 Similarly, we know that $\text{tw}(Z,\xi)=-1$ when $G'$ is empty; since in this case $e_0\leq-2$, we also have that $V_\text{can}$ is both initial and terminal and this means $\height[V_\text{can}]=0$.
\end{proof}

Completing the classification of negative-twisting structures on a Seifert fibred space $Z$ with $e(Z)=0$ now only requires a simple observation.
\begin{proof}[Proof of Theorem \ref{teo:classification0}]
 After determining the unique twisting number in Theorem \ref{teo:tw}, the proof is exactly the same as the ones of \cite[Theorem 1.1 and Proposition 5.1]{CM-negative}.
\end{proof}

\section{The highest negative twisting number}
\label{section:five}
\subsection{The contact invariant}\label{subsection:invariant}
Let $(M,\xi)$ be a contact 3-manifold, we are going to use the contact invariant $c^\circ(\xi)\in HF^\circ(-M,\s_\xi)$ defined in \cite{OSz-contact}, where $\circ$ stands either for the hat or the plus flavour. 

From this section on we take $Y$ as a negative-definite Seifert fibred space with $b_1(Y)=0$, and we set as $-Y$ the manifold obtained by swapping orientation. In order to be consistent with the previous sections, we call $\Gamma^*$ the standard graph of $-Y=M(e_0;r_1,...,r_n)$ with $n\geq3$ and $P_{\Gamma^*}$ its plumbing; recall that $b_2^+(P_{\Gamma^*})=1$. In addition, we only consider the case where $Y$ (and then also $-Y$) is not an $L$-space. 

Let us take the basis $\{[W_1],...,[W_{t+1}],T_{[V_1]},...,T_{[V_t]}\}$ of $\widehat{HF}(Y,\s)$ given by characteristic vectors, whose full path ends correctly, for a fixed $\Spin^c$-structure on $Y$. We recall that, using Equation \eqref{eq:even}, in the proof of Theorem \ref{teo:fullpath} we showed that $\Ker U\cap\:HF_\text{red}(Y,\s)$ can be canonically identified with $\widehat{HF}^\text{od}(Y,\s)$ through the inclusion $\rho_*$; hence, we refer to $\mathcal B=\{T_{[V_1]},...,T_{[V_t]}\}$ as a basis of both $\Ker U\cap\;HF_\text{red}(Y,\s)$ (from Theorem \ref{teo:fullpath}) and the odd subgroup of $\widehat{HF}(Y,\s)$.

Let us consider a contact structure $\xi$ on $-Y$. Its contact invariant $c^+(\xi)$ lives in $HF^+(Y,\s_\xi)$; moreover, if $c^+(\xi)\in HF_\text{red}$ is non-zero then from what we said above it can be written as a linear combination of elements in $\mathcal B$, and identified with $\widehat c(\xi)\in\widehat{HF}^\text{od}(Y,\s)$.

\begin{prop}
 \label{prop:neg_tw}
 Let $-Y$ be as above and equip it with a $\xi$ such that $\widehat c(\xi)$ is non-zero. If $d_3(\xi)+d(Y,\s_\xi)$ is odd, or equivalently if $c^+(\xi)\in HF_\emph{red}(Y,\s_\xi)$, then $\xi$ is negative-twisting.
\end{prop}
\begin{proof}
 From standard Heegaard Floer theory and Theorem \ref{teo:fullpath} we have that \[d_3(\xi)+d(Y,\s_\xi)=-M(\widehat c(\xi))+d(Y,\s_\xi)=M(V_a)-d(-Y,\s_\xi)\:,\] where $V_a\in\text{Char}(\Gamma^*,\s_\xi)$ is such that: $\gamma_a:=[V_a]\in\widehat{HF}(-Y,\s_\xi)$ satisfies $\langle\gamma_a,\widehat c(\xi)\rangle=1$, with respect to the duality identification $\widehat{HF}_*(-Y,\s)\simeq\widehat{HF}_{-*}(Y,\s)^\bullet$, and has minimal Alexander filtration $\mathcal F(V_a)$. In particular, the invariant $\widehat c(\xi)$ is in the odd subgroup of $\widehat{HF}(Y,\s_\xi)$ and $T_{[V_a]}\in\mathcal B$ is one of its coordinates.

 We mimic the second part of the proof of \cite[Theorem 1.1]{CM-negative}: using the definition of twisting number in Equation \eqref{eq:tw} and the $\tau$-Bennequin inequality for the $tb$-number in \cite[Theorem 1.3]{CM-negative}, we can write \[\text{tw}(-Y,\xi)=\text{TB}_\xi(K)-\dfrac{1}{e(Y)}<-\dfrac{1}{e(Y)}+\tau_\xi(K)+\tau_{\overline\xi}(K)\] where $K$ is the regular fibre of $-Y$. Now let $V_b$ be as $V_a$ but with maximal Alexander filtration, implying that if $\J\gamma_b:=[-V_b]\in\widehat{HF}(-Y,\overline\s_\xi)$ then $\langle\J\gamma_b,\widehat c(\overline\xi)\rangle=\langle\J\gamma_b,\J\widehat c(\xi)\rangle=\langle\gamma_b,\widehat c(\xi)\rangle=1$. From Equation \eqref{eq:Alexander}, Theorem \ref{teo:tau} and \cite[Lemma 2.2]{AC} we write \[\tau_\xi(K)+\tau_{\overline\xi}(K)\leq\tau_{\gamma_a}(K)+\tau_{\J\gamma_b}(K)=\mathcal F(V_a)+\mathcal F(-V_b)=\dfrac{1}{e(Y)}+e_1^TQ_*^{-1}\left(\dfrac{V_a-V_b}{2}\right)\]
 which by the definition of height leads to \[\text{tw}(-Y,\xi)<e_1^TQ_*^{-1}\left(\dfrac{V_a-V_b}{2}\right)=\left\{\begin{aligned}&-\height(\gamma_a+\gamma_b)\hspace{2.5cm}\text{ when }\gamma_a\neq\gamma_b \\ &-\height(\gamma_a)=-\height(\gamma_b)\hspace{0.5cm}\text{ when }\gamma_a=\gamma_b\end{aligned}\right.\:\leq0\:.\]
\end{proof}

\subsection{Realised characteristic vectors}
Since $-Y$ is not an $L$-space, from Remark \ref{remark:blowdown} we know that the 4-manifold $X:=X_{\Gamma^*}$, obtained by blowing down the graph $\Gamma^*$, is a Stein domain with $b_2^+(X)=1$. A famous result of Plamenevskaya \cite{OlgaP} then implies that there exists a unique $\Spin^c$-structure $\mathfrak v$ on $X$ such that $F^+_{\overline X,\mathfrak v}(c^+(\xi))=1$, where $\xi$ is induced by any Stein structure on $X$ corresponding to a given sequence of stabilisations. 

More specifically, let $(\Gamma^*)'$ be the $S^3$-subgraph of $\Gamma^*$ and let $\ell$ be equal to $|\Gamma^*|-|(\Gamma^*)'|$, we have that $(X,J)$ is determined by $(-1)$-contact (Legendrian) surgery on a Legendrian realisation $\mathcal K_1\cup...\cup\mathcal K_\ell$ of the blown down graph; moreover, we have that $c_1(\mathfrak v)=\mathbf V\in\text{Char}(X)$ and $\mathbf V=(\rot_\xi(\mathcal K_1),...,\rot_\xi(\mathcal K_\ell))$
where $\xi=J\lvert_{-Y}$. Let $s_j$ be the number of stabilisations on $\mathcal K_j$ (which is a Legendrian positive torus knot by Corollary \ref{cor:negative}) in the surgery presentation, then the smooth framing of the knot $K_j$ is given by \[\text{TB}_{\xist}(K_j)-s_j-1\] while \[\rot_\xi(\mathcal K_j)\in\{-s_j,-s_j+2,...,s_j-2,s_j\}\hspace{0.5cm}\text{ for }\hspace{0.5cm}j=1,...,\ell\:.\] The vector $\mathbf V$ is characteristic because $\text{TB}_{\xist}(T_{d_2,d_1})=\text{SL}_{\xist}(T_{d_2,d_1})$ is always odd. 

\begin{remark}
 We recall that here we denote a positive torus knot by $T_{d_2,d_1}$ with $d_2\leq d_1$. This convention is different with respect to the one we used in \cite[Section 4]{CM-negative}; in fact, there we take $d_1$ to be the denominator of $r_1$ where $M(-1;r_1,r_2)$ is the Seifert fibration of $T_{d_2,d_1}$ and $r_1>r_2$.
\end{remark}

A natural generalisation of \cite[Proposition 5.2]{CM-negative} can be shown for manifolds as $-Y$. 
\begin{lemma}
 \label{lemma:c}
 Let $-Y$ be an indefinite Seifert fibred space which is not an $L$-space, and $(X,J)$ a Stein filling obtained by blowing down the standard graph. Then there exists a full path $[V]$ which ends correctly such that $T_{[V]}=c^+(J\lvert_{-Y})\in HF_\emph{red}(Y,J\lvert_{-Y})$ where $V\in\emph{Char}(\Gamma^*)$ is the corresponding initial vector. 
\end{lemma}
\begin{proof}
 In light of Theorems \ref{teo:main} and \ref{teo:fullpath} we just have to show that $\widehat c(\xi)$ has odd parity, in other words that $c^+(\xi)\neq\Theta^+$, and then conclude as in the proof of \cite[Proposition 5.2]{CM-negative}. We recall that the latter proof is based on the result of Bodn\'ar-Plamenevskaya in \cite[Section 5]{BP} which says that the strict transform induces (through the blow-downs) an automorphism of $HF^+$ which commutes with the cobordism maps. In \cite{BP} this is done for negative-definite graphs, but by Theorem \ref{teo:main} the proof can be repeated in the same way when the graph is indefinite.
 
 Suppose that $c^+(\xi)=\Theta^+\in\Ker U\:\cap\Imm U^n$ for $n\geq1$, then the map $F^+_{\overline X,\mathfrak v}$ would send $\Theta^+$ to 1. Since we know from Lemma \ref{lemma:blowdown} that $b_2^+(X)=1$, we obtain a contradiction with $F^\infty_{\overline X,\mathfrak v}$ being vanishing \cite{OSz-negative}; namely, by commutation we would have $F^+_{\overline X,\mathfrak v}(\Theta^+)=\pi_*(F^\infty_{\overline X,\mathfrak v}(1))=\pi_*(0)=0$.
\end{proof}
In our case we can track down $V$ starting from $\mathbf V$ and going backwards along the blow-down procedure, and we denote the set of these vectors as the \emph{realised} characteristic vectors in $\text{Char}(\Gamma^*)$.

Before we give a description of which vectors are realised, we recall that there are four different types of an $S^3$-subgraph $G'$ of a standard graph $G$. Other than the empty set and the one given by the fibration of $T_{d_2,d_1}$ when $d_2>1$, we have the following two where the first one is just the centre of $\Gamma^*$ (the fibration of $T_{1,1}$) while the second one consists also of a string of $d_1-1$ vertices with framing $-2$ (the fibration of $T_{1,d_1}$ when $d_1>1$).

\begin{figure}[ht]
\begin{tikzpicture}[scale=0.9]
    \tkzDefPoints{0/0/A}  
    \tkzDrawPoints[fill,black,size=5](A)
     \tkzLabelPoint[below right](A){$-1$} 
 \end{tikzpicture}\hspace{3cm} 
 \begin{tikzpicture}[scale=0.9]
    \tkzDefPoints{1.5/-1/D, 3/-1/E, 5/-1/F}  
    \tkzDefPoints{3.5/-1/X, 4.5/-1/Y}\tkzDefPoints{3.9/-1/P, 4/-1/Q, 4.1/-1/R}
    \tkzDrawPoints[fill,black,size=1](P,Q,R)
    \tkzDrawSegment(E,D)\tkzDrawSegment(E,X)\tkzDrawSegment(Y,F)
    \tkzDrawPoints[fill,black,size=5](D,E,F)
     \tkzLabelPoint[below left](D){$-1$}\tkzLabelPoint[below left](E){$-2$}\tkzLabelPoint[below left](F){$-2$}
\end{tikzpicture}
\end{figure}

When $e_0=-1$ we call $G''$ the subgraph of $G$ made by the vertices that are not in $G'$, but are connected to the centre; in addition, in this case we also call $G''_i$ for $i=1,2$ the (eventual) unique vertex in the $i$-th leg that is not in $G'\cup G''$, but is connected to $G'$. Let us denote by $T$ the number of consecutive vertices with framing $-2$ at the end of the first leg of $G'$.

\begin{teo}
 \label{teo:realised}
 Let $G$ be the standard graph of any Seifert fibred space. An initial vector $V\in\emph{Char}(G)$ is realised if and only if its coordinates $(v_1,...,v_{|G|})$ are as follows: \[v_j=\left\{\begin{aligned} & m(j)+2\hspace{4.5cm}\text{if }j\in G' \\
 & m(j)+2,...,-m(j)-4\hspace{1.9cm}\text{ if }j\in G''_1 \\
 & m(j)+2,...,-m(j)-6-2T\hspace{0.95cm}\text{ if }j\in G''_2 \\
 & m(j)+2,...,-m(j)-2d_1-2d_2\hspace{0.5cm}\text{ if }j\in G'' \\
 & m(j)+2,...,-m(j)-2\hspace{1.9cm}\text{ otherwise}
 \end{aligned}\right.\] where $m(j)$ is the framing of the $j$-th vertex of $G$.  
\end{teo}
\begin{proof}
 Let $V$ be a vector as in the statement, then applying the strict transform, see \cite[Section 5]{BP}, to $V$ yields the vector in $\text{Char}(X_G)$ whose coordinates are untouched outside of $G'\cup \:G''\cup \:G''_1\cup \:G''_2$, and increased by $1,$ $2+T$ and $d_1+d_2-1$ in $G''_1,$ $G''_2$ and $G''$ respectively; moreover, the framings on the components are increased by $1,$ $2+T$ and $d_1d_2$ respectively (the ones in $G''$ become torus knots isotopic to $T_{d_2,d_1}$). It is then easy to check that each choice of $V$ corresponds to a choice of the rotation numbers in the surgery presentation of $X_G$.

 Suppose that $V$ is realised. Since $[V]$ ends correctly then $V$ must agree with $V_\text{can}$ on $G'$; moreover, again by using the strict transform as above we obtain that each coordinate $x$ of $j\in G''$ is increased to $x+d_1+d_2-1$ when we blow down. The value $x=m(j_i)+2$ (the one of $V_\text{can}$) is the lowest possible; therefore, we obtain \[m(j_i)+d_1+d_2+1=m(j_i)+2+(d_1+d_2-1)=-s_i\hspace{0.5cm}\text{ for }\hspace{0.5cm}i=3,...,\ell\:\] where the $j_i$'s are the vertices of $G$ that survive the blow-downs. We can now conclude by observing that the values of the rotation number are symmetric, which means the maximum is reached by \[-m(j_i)-d_1-d_2-1-(d_1+d_2-1)=-m(j_i)-2d_1-2d_2\:.\] The argument for the other vertices is similar.
\end{proof}

Let us consider a contact structure $\xi$ on $-Y$ such that $\widehat c(\xi)\in\widehat{HF}^\text{od}(Y,\s_\xi)$ is non-vanishing; as we discussed above, the condition on the parity is equivalent to the fact that $d_3(\xi)+d(Y,\s_\xi)$ is odd. We proved in \cite[Theorem 1.2]{CM-negative} that every negative-twisting structure on $Y$ has the same twisting number; the same is no longer true when we reverse the orientation as shown by Ghiggini and Van Horn-Morris in \cite[Proposition 2.7]{GvHM}. We introduce the following terminology: we say that $\overline{\text{tw}}(-Y)$ is the highest value among all the negative twisting numbers of tight structures on $-Y$; in the same way, we say that $\underline{\text{tw}}(-Y)$ is the lowest one. Note that since these are negative numbers, then the highest one actually has the smallest absolute value. We are going to ignore $\underline{\text{tw}}(-Y)$ in the remaining of the section; for this reason, we postpone the proof that this invariant is well-defined to Lemma \ref{lemma:A}.

\begin{prop}
 \label{prop:bigtw}
 Suppose that $\xi$ is any structure on $-Y$ induced by $(X,J)$ as described above. Then we have that \[\emph{tw}(-Y,\xi)=\overline{\emph{tw}}(-Y)=-1-\height[V_\emph{can}]\] and this number is either equal to $-d_1-d_2$ when the $S^3$-subgraph $(\Gamma^*)'$ of $\Gamma^*$ is the Seifert fibration of $T_{d_2,d_1}$ with $1\leq d_2\leq d_1$ coprime, or to $-1$ when $(\Gamma^*)'$ is empty. 
\end{prop}

\begin{proof}
 We refer to \cite[Section 4]{CM-negative} for more details about this proof, as it similar to the one we did in the negative-definite case. The proof that $\overline{\text{tw}}(-Y)$ is determined by $d_1$ and $d_2$ is the same as in \cite[Proposition 4.2 and Theorem 4.3]{CM-negative}; moreover, that this number equals $-1-\height[V_\text{can}]$ follows as in \cite[Remark 4.4]{CM-negative}.

 By construction the regular fibre $K$ of $-Y$ bounds a Lagrangian surface in $(X,J)$ whose boundary is its Legendrian realisation $\mathcal K$; hence, there are transverse push offs of $\mathcal K$ and $-\mathcal K$ bounding a symplectic surface in $(X,J)$. We use \cite[Theorem 1.3]{AC} to write \[\text{SL}_\xi(\pm K)=\tb_\xi(\mathcal K)\mp\rot_\xi(\mathcal K)=2\tau_\xi(\pm K)-1\] which implies \[\text{TB}_\xi(K)=\tau_\xi(K)+\tau_{\overline\xi}(K)-1\] by \cite[Theorem 1.3]{CM-negative} and the symmetries of knot Floer homology. We can now apply Theorem \ref{teo:tau} and as in the proof of Proposition \ref{prop:neg_tw} we obtain \[\text{tw}(-Y,\xi)=-1-e_1^TQ_*^{-1}\left(\dfrac{V+V'}{2}\right)=-1-\height[V]\:,\] where $V'$ is the initial vector of $[-V]$ and $V$ the one such that $T_{[V]}=\widehat c(\xi)$ from Lemma \ref{lemma:c}.

 The fact that $\height[V]=\height[V_\text{can}]$ follows easily from Theorem \ref{teo:realised} and the definition of height in Section \ref{subsection:height}. This concludes the proof as it shows that $\text{tw}(-Y,\xi)=\overline{\text{tw}}(-Y)$.
\end{proof}

In Section \ref{section:six} we are going to show that the structures obtained by blowing down $\Gamma^*$ are the only ones whose twisting number is $\overline{\text{tw}}(-Y)$.
We now prove the following lemma, which is going to be useful in Section \ref{section:seven}.
\begin{lemma}
 \label{lemma:realised}
 Let $\xi$ be a contact structure on $-Y$ obtained by Legendrian surgery on a fibre of $(M,\xi_0)$, where $M$ has indefinite standard graph $G$ and $b_1(M)=0$. Suppose that their contact invariants are non-zero and satisfy \[c^+(\xi)=T_{[V_1]}+...+T_{[V_k]}\in HF_\emph{red}(Y,\s_\xi)\hspace{0.5cm}\text{ and }\hspace{0.5cm}c^+(\xi_0)=T_{[Z_1]}+...+T_{[Z_h]}\in HF_\emph{red}(-M,\s_{\xi_0})\:,\] where each $V_i\in\emph{Char}(\Gamma^*,\s_\xi)$ and $Z_j\in\emph{Char}(G,\s_{\xi_0})$ is initial and has full path ending correctly. If each $Z_j$ is a realised characteristic vector then the same is true for $V_i$ with $i=1,...,k$.
\end{lemma}
\begin{proof}
 The graph $\Gamma^*$ is obtained by adding a new vertex either at the end of a leg of $G$, if the fibre  is singular, or at the centre of $G$, if the fibre is regular. Suppose that $T_{[V_i]}$ is a coordinate of $c^+(\xi)$, then $\mathfrak u_i\in\Spin^c(P_{\Gamma^*})$ given by $c_1(\mathfrak u_i)=V_i$ can be decomposed uniquely along $M$, because $b_1(M)=0$, as $\mathfrak u_i=\mathfrak v\cup_MJ$, where $c_1(\mathfrak v)=Z$ is the initial vector obtained by restricting $V_i$ to $G$, and $J$ is the Stein structure given by the Legendrian surgery. Say $W:=P_{\Gamma^*}\setminus\mathring P_G$, then note that $\mathfrak u_i\lvert_W=J$ because by \cite[Proposition 3.3]{G-fillability} the only $\Spin^c$-structure $\mathfrak t$ on the cobordism $W$ such that $F^+_{\overline W,\mathfrak t}(c^+(\xi))\neq0$ is the one induced by $J$; in fact, we know that $F^+_{\overline W,\mathfrak u_i\lvert_W}(c^+(\xi))\neq0$ because otherwise $F^+_{\overline P_{\Gamma^*},\mathfrak u_i}(T_{[V_i]})=(F^+_{\overline P_G,\mathfrak v}\circ F^+_{\overline W,\mathfrak u_i\lvert_W})(T_{[V_i]})=0$ contradicting the results in Section \ref{section:two}. 

 If we knew that $T_{[Z]}$ is a coordinate of $c^+(\xi_0)$ then we could use the fact that the the first Chern classes of the Stein structures on the blow-down $X_{P_G}$ of $P_G$ are precisely the strict transform of realised characteristic vectors of $P_G$. Therefore, we would have that the strict transform of $V_i$ corresponds to a Stein structure on $X_{\Gamma^*}$ given by $J_0\cup_MJ$, where $J_0$ is determined by $Z$, thus implying that $V_i$ is also realised.

 We now prove that $T_{[Z]}$ is indeed a coordinate of $c^+(\xi_0)$. By functoriality we have that $F^+_{\overline W,J}(c^+(\xi))=c^+(\xi_0)$, thus the claim follows from the fact that if $T_{[V_a]}$ and $T_{[V_b]}$ are distinct coordinates of $c^+(\xi)$ then $F^+_{\overline W,J}(T_{^[V_a]})\neq F^+_{\overline W,J}(T_{[V_b]})$. In fact, if this were not the case then we would have two distinct characteristic vectors as first Chern class of a $\Spin^c$-structure on $P_{\Gamma^*}$, which is not possible.  
\end{proof}

\section{The classification for indefinite Seifert fibred spaces}
\label{section:six}
\subsection{Background} 
\label{subsection:background}
The starting point of our classification is the observation due to Ghiggini \cite{G-} and Massot \cite{Massot} that every negative-twisting structure on a Seifert fibred space over a sphere admits a convex decomposition into background and solid tori. The \emph{background} is the piece diffeomorphic to $\Sigma_{0,n}\times S^1$ where $\Sigma_{0,n}$ is an $n$-punctured sphere and $n$ is the number of singular fibres; the possible boundary slopes are determined from the Seifert constants through the Ghiggini-Massot's algorithm in Theorem \ref{teo:Paolo}, and their common denominator fixes the unique tight contact structure with twisting number equal to the negative of this denominator on $\Sigma_{0,n}\times S^1$, see \cite[Section 3]{G-} and \cite[Proposition 5.1]{CM-negative}.

\begin{lemma}
 \label{lemma:A}
 Every Seifert fibred space $M$ with $b_1(M)=0$ has only finitely many negative twisting numbers.
\end{lemma}

\begin{proof}
 We do the proof in the case that $e(M)>0$, because for negative-definite manifolds we already know from \cite[Section 4]{CM-negative} that the twisting number is unique. The claim follows by finding an appropriate upper bound for $|\underline{\text{tw}}(M)|$ through the Ghiggini-Massot's algorithm, in terms of the negative continued fractions of the Seifert coefficients of $M$; more specifically, here we prove that if $-q$ is a negative twisting number then $q<\frac{n-2}{e(M)}$ where $n$ is the number of legs.  

 By definition for each leg there is a positive integer $p_i$ such that $\frac{p_i}{q}$ is the best upper approximation for $r_i$; equivalently, the number $p_i$ is the smallest integer such that $\frac{p_i}{q}>r_i$ and no fraction with denominator lower than $q$ is between them. Let $r_i=\frac{b_i}{a_i}$ as reduced fraction, then there is a unique integer $c_i$, with $1 \leq c_i \leq a_i$, such that $p_i = \frac{qb_i + c_i}{a_i}$; indeed, we have that $c_i$ is the smallest positive residue of $-qb_i\:\text{ mod }a_i$, and then $\frac{c_i}{a_i}\in(0,1]$.

 The twisting-number condition says $p_1 + \cdots + p_n = -e_0q + n - 2$. Using the expression above for each $p_i$ we obtain \[\sum_{i=1}^{n} \left( \frac{qb_i}{a_i} + \frac{c_i}{a_i} \right)= -e_0q + n - 2\hspace{0.5cm}\text{ which by rearranging becomes }\hspace{0.5cm}qe(M) = n - 2 - \sum_{i=1}^{n} \frac{c_i}{a_i}\:,\] after we substitute $e(M)=r_1+\cdots +r_n+e_0$ in it. Since indefiniteness is exactly $e(M)>0$ and each $\frac{c_i}{a_i}$ is positive by construction, we conclude that $q<\frac{n-2}{e(M)}$.
\end{proof}

The botany of the negative-twisting contact structures on an indefinite Seifert fibred space $-Y$ depends on its standard graph. For this reason we introduce the following definition: we say that $-Y$ is a  manifold of \emph{type B} when its standard graph $\Gamma^*$ contains one of the seven graphs in Figure \ref{Torus}, which means the standard graph of a Seifert fibred torus bundle over a circle appears as a subgraph of $\Gamma^*$; otherwise, we say that $-Y$ is a manifold of \emph{type A}. We recall that $-Y$ is not an $L$-space since the latter ones do not carry any negative-twisting structures.

The strategy of the classification is to construct some structures on $-Y$ whose contact invariant $c^+$ is non-vanishing and lives in $HF_\text{red}$; we saw in Proposition \ref{prop:neg_tw} that these are necessarily negative-twisting. The classification then follows using the usual scheme: based on the work of Honda \cite{Honda1,Honda2}, or equivalently Giroux \cite{Giroux1,Giroux2}, we find an upper bound via convex decompositions, then we show that it is realised by the ones we explicitly constructed.

The explicit value of the upper bound, for a given twisting number $-q$ is obtained as follows. For each leg take the positive reduced fraction $\frac{p_i}{q}$ to be the best upper approximation for  $r_i$. In the continued fraction expansion: if $-\frac{1}{r_i}=[m^i_1,...,m^i_{k_i}]$ then $-\frac{q}{p_i}=[m^i_1,...,m^i_{h_i-1},n^i_{h_i}]$ where either $h_i=k_i+1$ or $h_i\leq k_i$ and $n^i_{h_i}>m^i_{h_i}$. 
We have that the number of contact structures, up to isotopy, with twisting number $-q<-1$ is at most $T(1,q)\cdots T(n,q)$ where each $T(i,q)$ is equal to 1 when $h_i>k_i$ or to \begin{equation}(n^i_{h_i}-m^i_{h_i})\:\cdot\prod_{j=h_i+1}^{k_i}|m^i_j+1|\label{eq:upper1}\end{equation} when $h_i\leq k_i$. The number of contact structures with twisting number $-1$ is instead equal to \begin{equation}|e_0+1|\displaystyle\prod_{i=1}^n\displaystyle\prod_{j=1}^{k_i} |m_j^i+1|\label{eq:upper2}\end{equation} as already shown in the proofs of \cite[Theorem 1.1 and Proposition 5.1]{CM-negative} in the negative-definite case, and Theorem \ref{teo:classification0} in the singular case; same arguments hold also for $-Y$ by Proposition \ref{prop:bigtw}. 
More generally, this argument can be used to prove Proposition \ref{prop:main}, which says that all the negative-twisting structures, with the highest possible twisting number, are induced by a Stein structure on the 4-manifold obtained by blowing down the standard graph.

\begin{proof}[Proof of Proposition \ref{prop:main}]
 As mentioned above, it follows by Lemma \ref{lemma:blowdown}, Theorem \ref{teo:classification0} and Proposition \ref{prop:bigtw}, using the same arguments in \cite[Theorem 1.1 and Proposition 5.1]{CM-negative}. 
\end{proof}

Generally the described upper bound takes into account that the background admits a unique $(-q)$-twisting structure up to isotopy and counts all possible tight extensions to the glued-in solid tori, as has been described to more details in the proof of \cite[Proposition 5.1]{CM-negative}.

Below we separately analyse which twisting numbers (other than the highest) manifolds of type A and B admit, and classify the corresponding structures.

\subsection{Manifolds of type A}
\label{subsection:A}
Following the work in \cite{Massot,Bowden}, we start this subsection by recalling the properties of a contact structure which is tangent to the fibres of a Seifert fibred space $M$. The result below follows from \cite[Subsection 4.2, Lemma 7.2 and Proposition 8.9]{Massot}.

\begin{prop}
 \label{prop:transverse}
 Suppose that $\xi$ is a contact structure on a Seifert fibred space $M$; if $\xi$ is tangent to the fibres then $(M,\xi)$ has a (weak) symplectic filling $(W,\omega)$ and $\xi$ is negative-twisting. Furthermore, there is a Legendrian realisation of a regular fibre of $M$ which bounds a Lagrangian surface in $(W,\omega)$.   
\end{prop}
\begin{proof}
 Starting from the Seifert fibration $M\rightarrow B$, where the base $B$ is a symplectic 2-dimensional orbifold (of genus zero) with isolated cyclic singularities, one can consider the associated singular disk bundle $D\rightarrow B$ where each fibre is equipped with (a quotient of) the symplectic form $\frac{1}{2}\dd(r^2\dd\theta)$. As explained in \cite{Massot}, this is a symplectic orbifold $(D,\omega_D)$ which can be resolved to get the desired filling $W$; more specifically, the resolution gives a branched cover $c:W\rightarrow D$.
 
 The regular fibres of $D$ are Lagrangian disks whose boundary is Legendrian, and isotopic to a regular fibre $K$ of $M$. Since the regular fibres are preserved by the branched cover, we have that $K$ is fixed by $c$ and then it still bounds a Lagrangian surface in $(W,\omega:=c^*\omega_D)$ with the same boundary condition.
\end{proof}

Using Ghiggini-Massot's algorithm, the possible values of a negative twisting number (of a tight structure) on $-Y$ can be determined exactly through elementary number theory. This set of integers only depends on the Seifert coefficients of $-Y=M(e_0;r_1,...,r_n)$; moreover, from Lemma \ref{lemma:A} we also know that it is finite and thus we can identify the invariant $\underline{\text{tw}}(-Y)$, the lowest twisting number of $-Y$. We can pinpoint a special family of indefinite Seifert fibred spaces which can be considered minimal for a given background of $\underline{\text{tw}}(-Y)$; these manifolds allow us to determine the values of the twisting numbers of any other $-Y$ of type A. 

We introduce the following terminology: we say that a  manifold of type A is \emph{without tails} when it has at least two negative twisting numbers (a posteriori, we know there will be exactly two of them) and $|\underline{\text{tw}}(-Y)|>a_i$ for all $i=1,...,n$, where $r_i=\frac{b_i}{a_i}$ as a positive reduced fraction. A simple example of a manifold without tails is the oppositely oriented Brieskorn sphere $-\Sigma(2,5,9)$ whose standard graph is in Figure \ref{TypeA}.

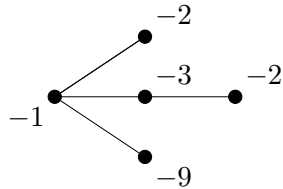
\begin{figure}[ht] 
    \begin{tikzpicture}[scale=0.8]
    \tkzDefPoints{0/0/A, 1.5/1/B, 1.5/0/D, 3/0/E, 1.5/-1/C}  
    \tkzDrawSegment(A,B)\tkzDrawSegment(A,D)\tkzDrawSegment(E,D)\tkzDrawSegment(B,A)\tkzDrawSegment(A,C)
    \tkzDrawPoints[fill,black,size=5](A,B,C,D,E)
     \tkzLabelPoint[below left](A){$-1$} 
     \tkzLabelPoint[above right](B){$-2$}\tkzLabelPoint[below right](C){$-9$}
     \tkzLabelPoint[above right](D){$-3$}\tkzLabelPoint[above right](E){$-2$}
       \end{tikzpicture}
     \caption{\smaller[1]{The manifold $-Y=-\Sigma(2,5,9)$, corresponding to $M(-1;\frac{1}{2},\frac{2}{5},\frac{1}{9})$. Computing the negative twisting numbers yields $|\underline{\text{tw}}(-Y)|=17>2,5,9$.}}
     \label{TypeA}
\end{figure}

\begin{lemma}
 \label{lemma:tails}
  Let $-Y$ be a  manifold of type A. Then either $-Y$ has a unique negative twisting number, or is without tails, or it is related to a manifold $M$ without tails such that:
  \begin{itemize}
      \item for any leg $i$ such that $|\underline{\emph{tw}}(-Y)|\leq a_i$ one has $-\frac{1}{r_i}=[m_1^i,...,m^i_{k_i}]$, while the corresponding leg of $M$ either has coefficient $-\frac{1}{r_i'}=[m^i_1,...,m^i_{l_i}]$ with $l_i<k_i$ or it is missing;
      \item if $-q$ is a twisting number of $-Y$ then the same is true for $M$.
  \end{itemize}
\end{lemma}
\begin{proof}
 It is clear that if $-Y$ is in neither of the first two categories then $\underline q=|\underline{\text{tw}}(-Y)|\leq a_i$ for at least one $i$. We first construct an $M$ such that its $i$-th leg has coefficient $r'_i$ as in the statement for every such $i$; then we prove that the set of the negative twisting numbers of $-Y$ is contained in the one of $M$; finally, we show that $M$ is also  of type $A$.

 We have already observed that $\frac{p_i}{\underline q}$ as a best upper approximation for $r_i=\frac{b_i}{a_i}$ with $\underline q<a_i$ has continued fraction expansion of the form $[m^i_1,...,m^i_{h_i-1},n^i_{h_i}]$ with $h_i\leq k_i$ and $n^i_{h_i}>m^i_{h_i}$. Taking $l_i=h_i-1$, where we interpret $0$ as no leg, gives us the desired truncation; note that in this way $r'_i$ is the biggest left Farey neighbour of $\frac{p_i}{\underline q}$ with denominator at most $\underline q$. Indeed, the rational number $\frac{p_i}{\underline q}$ is a best upper approximation also for $r'_i$  and so are all the numbers $[m^i_1,...,m^i_{j_i-1},s^i_{j_i}]$ with $j_i\leq h_i$ and $s^i_{j_i}>m^i_{j_i}$ (or $s^i_{j_i}>n^i_{j_i}$ if $j_i=h_i$). Since the absolute value $q$ of any twisting number of $-Y$ is a numerator of a continued fraction of the described form, the corresponding $\frac{P_i}{q}$ remains a best upper approximation for the truncated fraction $r'_i$. Consequently, since the $P_i$'s are the same with respect to $r_i$ and $r_i'$, so is their sum and then $-q$ is also a twisting number of $M$. 

 Since $-Y$ has at least two twisting numbers, by the argument above the same is true for $M$. Since $M$ cannot be of type B, as its standard graph is a subgraph of the one of $-Y$, it follows that it is of  type A and then without tails.
\end{proof}

We now analyse negative-twisting structures on a manifold without tails. For any $q>1$ we can analogously say that a manifold is without tails with respect to the twisting number $-q$ when $q>a_i$ for each $i$.

\begin{teo}
 \label{teo:A}
 If $-Y$ is a manifold of type A without tails then it has exactly two negative twisting numbers.   
\end{teo}
\begin{proof}
 Let $-Y=M(e_0;r_1,...,r_n)$ with $r_i=\frac{b_i}{a_i}$ as a reduced fraction, and denote by $\Gamma^*$ its standard graph.
 Being of type A without tails, the manifold $-Y$ admits tight contact structures with at least two twisting numbers: the highest $\overline{\text{tw}}(-Y)$ from Section \ref{section:five} and the lowest $\underline{\text{tw}}(-Y)$ with $a_i<|\underline{\text{tw}}(-Y)|$ for $i=1,...,n$. 
 
 We assume to the contrary with the statement that there exists another twisting number $-q$. 
  If we write $\underline q=|\underline{\text{tw}}(-Y)|$, for the absolute value of the lowest twisting number, and $\overline q=|\overline{\text{tw}}(-Y)|$, for the absolute value of the highest one, we have strict inequalities $\overline q<q<\underline q$. 
  We will first describe how the best upper approximations for $r_i$ corresponding to the three twisting numbers are related to each other, and then prove that they cannot all three satisfy the twisting number relation appearing in Ghiggini-Massot's algorithm.
 
 We write $\frac{p_i^0}{\overline q},$ $\frac{p_i}{q},$ and $\frac{p_i^1}{\underline q}$ for the best upper approximations for $r_i$ when $i=1,...,n$. Note that when $\overline q=1$ we have that $p_1^0=-e_0-1$ and $p_i^0=1$ for $i=2,...,n$; in this way, the relation $p^0_1+...+p^0_n=-e_0\overline q+n-2$ is still satisfied, see \cite[Section 4]{CM-negative}.

 Since $-Y$ is without tails, the Seifert coefficient $r_i$ and the best upper approximations $\frac{p_i^1}{\underline q}$ are Farey neighbours. We describe a subgraph $G$ of $\Gamma^*$ as the truncation we obtain by looking at a manifold without tails with respect to $q>1$: the graph $G$ is determined by $M(e_0;\frac{d_1}{c_1},...,\frac{d_{n'}}{c_{n'}})$ where $\frac{d_i}{c_i}$ is the biggest left Farey neighbour with $c_i\leq q$ (when it is non-zero) of $\frac{p_i}{q}$; similarly, we define $\frac{u_i}{v_i}$ as a left Farey neighbour of $\frac{p_i^0}{\overline q}$. 
  
For each leg $i=1,...,n$ let
\[
X_i^{(0)}=
\begin{pmatrix}
p_i^{0} & u_i\\
\overline q & v_i
\end{pmatrix}
\qquad
X_i=
\begin{pmatrix}
p_i & d_i\\
q& c_i
\end{pmatrix}
\qquad
X_i^{(1)}=
\begin{pmatrix}
p_i^{1} & b_i\\
\underline q & a_i
\end{pmatrix}
\]
be the corresponding $\text{SL}_2(\mathbb{Z})$-matrices which satisfy
$X_i=X_i^{(0)}M_i$ and $X_i^{(1)}=X_iN_i$, where
\[
M_i=
\begin{pmatrix}
\sigma_i & \rho_i\\
\tau_i & \pi_i
\end{pmatrix}\hspace{0.5cm}\text{ and }\hspace{0.5cm}
N_i=
\begin{pmatrix}
\sigma_i' & \rho_i'\\
\tau_i' & \pi_i'
\end{pmatrix}
\]
have all entries non-negative and $\tau_i,\tau_i',\pi_i,\sigma_i'>0$ when either $e_0=-1,-2$ or $i>1$; in other words, the intervals of Farey neighbours $(r_i,\frac{p_i^1}{\underline q})\subset(\frac{d_i}{c_i},\frac{p_i}{q})\subset(\frac{u_i}{v_i},\frac{p_i^0}{\overline q})$ are nested. When $e_0\leq-3$ and $i=1$ we can prove directly that $\tau_1,\pi_1>1$; in fact, $\frac{p_1^0}{\overline q}=\frac{-e_0-1}{1}$ and $\frac{u_1}{v_1}=\frac{-e_0-2}{1}$ imply $\tau_1=-p_1-(e_0+1)q$ and $\pi_1=-d_1-(e_0+1)c_1$. 

At each level, the twisting number relation is
\[
\sum_{i=1}^n p_i^{0}=-e_0 \overline q+n-2\:,\hspace{0.5cm}\sum_{i=1}^n p_i=-e_0q+n-2\hspace{0.5cm}\text{ and }\hspace{0.5cm}\sum_{i=1}^n p_i^1=-e_0\underline q+n-2\:.\]

From the first column of $X_i=X_i^{(0)}M_i$ we obtain $p_i=\sigma_ip_i^{0}+\tau_iu_i$ and $q=\sigma_i\overline q+\tau_iv_i$, which can be rephrased as \[p_i=\dfrac{p_i^0}{\overline q}(q-\tau_iv_i)+\tau_iu_i=\dfrac{qp_i^0-\tau_i(v_ip_i^0-u_i\overline q)}{\overline q}=\dfrac{qp^0_i-\tau_i}{\overline q}\:.\] Summing over $i$, substituting in the twisting number relation, and repeating the same computation for $X_i^{(1)}=X_iN_i$, gives
\begin{equation}
\sum_{i=1}^n \tau_i=(n-2)(q-\overline q)\hspace{0.5cm}\text{ and similarly }\hspace{0.5cm}\sum_{i=1}^n \tau_i'=(n-2)(\underline q-q)\:.\label{eq:2}
\end{equation}

If we consider the composition $M_iN_i$ then the corresponding lower-left entry is $\sigma_i'\tau_i+\tau_i'\pi_i$, so the same computation gives
\begin{equation}
\sum_{i=1}^n \bigl(\sigma_i'\tau_i+\tau_i'\pi_i\bigr)
=(n-2)(\underline q-\overline q)\:, \label{eq:3}
\end{equation}
and subtracting Equations \eqref{eq:2} from Equation \eqref{eq:3} yields \[\sum_{i=1}^n \bigl((\sigma_i'-1)\tau_i+(\pi_i-1)\tau_i'\bigr)=0\:.\]

We have that $\tau_i,\tau_i',\pi_i,\sigma_i'>0$, thus $\sigma_i'=\pi_i=1$ for $i=1,...,n$. Now from $X_i=X_i^{(0)}M_i$, using $\pi_i=1$, the lower-right entry gives $c_i=\rho_i\overline q+v_i$ while the lower-left one gives \[q=\overline q\sigma_i+\tau_iv_i=\overline q(1+\rho_i\tau_i)+\tau_iv_i=\overline q+\tau_i(\rho_i\overline q+v_i)=\overline q+\tau_ic_i\:;\] hence, we showed that $\tau_i=\frac{q-\overline q}{c_i}$ for $i=1,...,n$. Plugging this into Equation \eqref{eq:2} and dividing by $q-\overline q>0$ now gives
\begin{equation}
  \dfrac{1}{c_1}+\cdots+\dfrac{1}{c_n}=n-2\:.
 \label{eq:6}
\end{equation}
 We can assume that each $c_i\geq2$ if we observe that the subgraph $G$ can have $n'\leq n$ legs; in fact, it is possible that some legs of $\Gamma^*$ disappear in the process, and in this case the coefficient $\frac{d_i}{c_i}$ would be $\frac{0}{1}$. Nonetheless, we note that $\frac{1}{c_i}=1$ and this reduces by 1 the right-hand side of Equation \eqref{eq:6} corresponding to the missing $i$-th leg. 

 To complete the proof, we now just have to determine all the possible subgraphs $G$ such that $\frac{1}{c_1}+...+\frac{1}{c_{n'}}=n'-2$ with $c_i\geq2$. We already solved this problem in the proof of Proposition \ref{prop:tw}, and we know that the only cases are $(2,3,6),$ $(2,4,4),$ $(3,3,3)$ and $(2,2,2,2)$. Since we want $Y$ (and then $-Y$) to not be an $L$-space, using Theorem \ref{teo:criterion} we conclude that: if $e_0=-1$ then $n'=3$ and the $d_i$'s are $(1,1,1)$; meanwhile, if $e_0\leq-2$ then \[2=n'-\sum_{i=1}^{n'}\dfrac{1}{c_i}=\sum_{i=1}^{n'}\left(1-\dfrac{1}{c_i}\right)\geq\sum_{i=1}^{n'}\dfrac{d_i}{c_i}\geq-e_0\geq2\:,\] because by construction the manifold presented by $G$ has two distinct negative twisting numbers and thus non-negative $e$, which forces $e_0=-2$ and $d_i=c_i-1$. This means that $\Gamma^*$ contains a subgraph isomorphic to the standard graph of a torus bundle, see Figure \ref{Torus}, and this is impossible because the manifold is of type A.
\end{proof}

\begin{prop}
 \label{prop:A}
  If $-Y$ is a  manifold of type A without tails then all the negative-twisting tight structures on $-Y$ are symplectically fillable, and they are distinguished by $c^+$ in $HF_\emph{red}(Y)$.
\end{prop}
\begin{proof}
 The manifold $-Y$ has exactly two negative twisting numbers because of Theorem \ref{teo:A}. Since $\underline q=|\underline{\text{tw}}(-Y)|>a_i$ for $i=1,...,n$, Massot in \cite[Theorem 8.7]{Massot} tells us that $-Y$ actually has a unique contact structure $\xi_0$ with twisting number $\underline q$, which is tangent to the fibres; hence, this structure is symplectically fillable because of Proposition \ref{prop:transverse}. 
 
 The fact that all the tight contact structures with twisting number $\overline q=|\overline{\text{tw}}(-Y)|$ are the ones in Proposition \ref{prop:bigtw}, which means they are induced by Stein structures on the blown down graph $X_{\Gamma^*}$, has been discussed in Proposition \ref{prop:main}. More specifically, we observe that the upper bound (for these structures) from convex surface theory is precisely equal to the number of realised characteristic vectors; we know that these vectors are in one to one correspondence with the contact structures we are considering by Lemma \ref{lemma:c}. Hence, the last claim holds when $\xi$ satisfies $\text{tw}(-Y,\xi)=\overline{\text{tw}}(-Y)$, and each contact invariant is of the form $T_{[V]}\in HF_\text{red}(Y)$ for a certain realised characteristic vector $V$. 
 
 To distinguish the invariant of the structure $\xi_0$, we use the fact that there exists a Legendrian knot $\mathcal K$, smoothly isotopic to a regular fibre of $-Y$, bounding a Lagrangian surface $\Sigma$ in a symplectic filling of $(-Y,\xi)$ from Proposition \ref{prop:transverse}. By the same argument in the proof of Propositions \ref{prop:neg_tw} and \ref{prop:bigtw} we have \[\text{TB}_{\xi_0}(K)=\tau_{\xi_0}(K)+\tau_{\overline\xi_0}(K)-1\:.\] Note that in this case there is an additional detail to discuss: the fact that $\text{SL}_{\xi_0}(\pm K)=2\tau_{\xi_0}(\pm K)-1$ does not follow just by \cite[Theorem 1.3]{AC}; the knot $\pm K$ is indeed the boundary of a symplectic surface $\Sigma^\pm$, but the filling is not necessarily Stein (Massot \cite{Massot} only guarantees is symplectic). The result still holds with little modifications: from \cite[Theorem 4.1]{AC} we still have the slice-Bennequin inequality \[\text{SL}_{\xi_0}(\pm K)\leq2g(\Sigma)-1-\Sigma\cdot\Sigma+c_1(\omega)(\Sigma^\pm)\] where $\omega$ is the symplectic form. This inequality is sharp because $\Sigma^\pm$ is symplectic; a proof of this fact can be found in \cite[Lemma 2.13]{EG}.

 Assume first that $c^+(\xi_0)=\Theta^+_{\s_{\xi_0}}$; since in this case $\widehat c(\xi_0)\in\widehat{HF}^\text{ev}(Y,\s_{\xi_0})$, we can compute $\tau_{\xi_0}(K):=-\tau_{\widehat c(\xi_0)}(K)$ from Theorem \ref{teo:tau}; namely, write $\widehat c(\xi_0)=[W_1]+\cdots+[W_k]$ where $\mathcal F(W_1)<\cdots<\mathcal F(W_k)$ are initial, and $\{[W_1],...,[W_t]\}$ is the canonical basis of $\widehat{HF}^\text{ev}(Y,\s_{\xi_0})$ given by the full paths. By definition of twisting number and conjugation we can then write
 \[\begin{aligned}-\underline q=\text{tw}(-Y,\xi_0)&=-\dfrac{1}{e(Y)}+\tau_{\xi_0}(K)+\tau_{\overline\xi_0}(K)-1=\\ &=-1-e_1^TQ^{-1}\left(\dfrac{W_k+W_1'}{2}\right)\geq-1+\height[W_\text{can}]\geq-1\:.\end{aligned}\] where $W_\text{can}\in\text{Char}(P_\Gamma,\s_{\xi_0})$ and $W_1'$ is the initial vector of $[-W_1]$.
 This is not possible by the assumptions we made. 

 Now we know that $c^+(\xi_0)\in HF_\text{red}(Y,\s_{\xi_0})$. Suppose that $c^+(\xi_0)=T_{[V]}$ for some realised $V$; then doing the same computation as above, but using Equation \eqref{eq:taumirror} instead,  yields $-\underline q=-1-\height[V_\text{can}]=\overline{\text{tw}}(-Y)$ by Proposition \ref{prop:bigtw}. This is again a contradiction; therefore, the invariant $c^+(\xi_0)$ is distinct from the one of all the other structures on $-Y$.
\end{proof}

We summarise the shrinking process that lead us to the family of manifolds without tails. We started from the set of all Seifert fibred spaces whose base orbifold is a sphere; among these we first restrict to the ones with indefinite standard graph, as we already classified the negative twisting structures in the other cases, see \cite[Theorem 1.1]{CM-negative} and Theorem \ref{teo:classification0}. As we mentioned in Remark \ref{remark:L}, we do not consider indefinite $L$-spaces because they do not carry such kind of structures. By examining the standard graph, we exclude manifolds of type B: the ones whose graph contains a subgraph isomorphic to one in Figure \ref{Torus}; we then compute the twisting numbers and observe that we do not need to consider manifolds with a unique twisting number: for these Proposition \ref{prop:main} shows that the classification is the one already given in Proposition \ref{prop:bigtw}. Finally, the condition on $\underline{\text{tw}}(-Y)$ identifies our minimal manifolds, the ones whose standard graph is simple enough to prove Theorem \ref{teo:A}; for a generic $-Y$ the same result requires Lemma \ref{lemma:tails}, this is included in Proposition \ref{prop:classificationA}.  

\begin{remark}
 \label{remark:A}
 The fact that $-Y$ has at most two negative twisting numbers can be extended to any  manifold of type A by applying Lemma \ref{lemma:tails}.   
\end{remark}

From Lemma \ref{lemma:tails} we know that $-Y$ can be obtained performing surgery on an $M$ without tails. This surgery can be done in the contact setting in all tight structures on $M$. As we have described the structures with the highest twisting number $\overline{\text{tw}}(-Y)$ for any $-Y$ already in Section \ref{section:five}, we restrict here to $M$ equipped with the unique structure $\xi_0$ with twisting number $\underline{\text{tw}}(M)$, which is tangent to the fibres by \cite[Theorem 8.7]{Massot}. Going from $M$ to $-Y$ we perform surgery on push-offs of its singular fibres. Using notation from the proof of Lemma \ref{lemma:tails}, we have that $M$ has Seifert coefficients $[m^i_1,...,m^i_{h_i-1}]$, while the corresponding best upper approximations equals $[m^i_1,...,m^i_{h_i-1},n^i_{h_i}]$ and the Seifert coefficient of $-Y$ to $[m^i_1,...,m^i_{h_i-1},m^i_{h_i},...,m^i_{k_i}]$ with $m_{h_i}^i>n_{h_i}^i$. This means that the contact framing of the $i$-th singular fibre in $(M,\xi_0)$ is exactly $n_{h_i}^i$ and as $m_{h_i}^i<n_{h_i}^i$ we can realise the desired surgery as Legendrian surgery. The obtained contact structures differ by the $\text{Spin}^c$-structures on the symplectic cobordism and therefore by contact invariant.

Comparing to the upper bound, we complete the classification realising that there are no other negative twisting structures on $-Y$.

\begin{prop}
 \label{prop:classificationA}
 Suppose that $-Y$ is a  manifold of type A. A negative-twisting tight contact structure on $-Y$ is either isotopic to one of the $\overline{\emph{tw}}(-Y)$-structures, obtained by complete blow-down of $\Gamma^*$, or to a $\underline{\emph{tw}}(-Y)$-structure, obtained by a negative contact surgery on (some) fibres of the corresponding manifold $M$ without tails equipped with its unique structure tangent to the fibres. In particular, all the negative-twisting structures on $-Y$ are symplectically fillable.    
\end{prop}

\begin{proof}
 We notice that the number of constructed structures in both families exactly agrees with predicted upper bound as computed in Subsection \ref{subsection:background}. Indeed, for $\overline{\text{tw}}(-Y)$-structures the proof is the same as in the negative definite case \cite[Proposition 5.1]{CM-negative}; for $\underline{\text{tw}}(-Y)$-structures, obtained by Legendrian surgery on tails, we directly see that there are exactly $(n^i_{h_i}-m^i_{h_i})\:\cdot\prod_{j=h_i+1}^{k_i}|m^i_j+1|$ choices of stabilisations.

 All structures are symplectically fillable: the $\overline{\text{tw}}(-Y)$-structures are Stein fillable as proved in Section \ref{section:five} and the $\underline{\text{tw}}(-Y)$-structures are obtained by Legendrian contact surgery from the tangent contact manifold $(M,\xi_0)$ which is symplectically fillable by Proposition \ref{prop:A}.
\end{proof}

Regarding the contact invariant of $\xi$, we can already prove the following result.

\begin{cor}
 \label{cor:red}
 If $\xi$ is a negative-twisting structure on the  manifold $-Y$ of type A then $c^+(\xi)\in HF_\emph{red}(Y,\s_\xi)$.  
\end{cor}
\begin{proof}
 Again from Lemma \ref{lemma:c} we know that the claim holds when $\xi$ satisfies $\text{tw}(-Y,\xi)=\overline{\text{tw}}(-Y)$, while from Proposition \ref{prop:A} when $-Y$ is without tails. When $-Y$ has tails there is a cobordism map $F^+:HF^+(Y,\s_\xi)\rightarrow HF^+(-M,\s_{\xi_0})$ by Proposition \ref{prop:classificationA}, induced by Legendrian surgery, where $\xi_0$ is tangent to the fibres of a manifold $M$ without tails. It is a property of Heegaard Floer maps that if $c^+(\xi)=\Theta^+_{\s_\xi}$ then $F^+(c^+(\xi))=c^+(\xi_0)\neq0$, because $(M,\xi_0)$ is symplectically fillable, and $c^+(\xi_0)=\Theta^+_{\s_{\xi_0}}$, which is not possible by Proposition \ref{prop:A}.
\end{proof}

Suppose that $(-Y,\xi)$ is as in Proposition \ref{prop:classificationA}, and satisfies $\text{tw}(-Y,\xi)=\underline{\text{tw}}(-Y)<\overline{\text{tw}}(-Y)$. Since the structure is fillable, we have that $c^+(\xi)$ is non-zero; moreover, we have that $c^+(\xi)$ cannot be expressed as $T_{[V]}$ for any realised characteristic vector. 
We also know from the results in \cite{Massot} that if $-Y$ is without tails then $\xi$ is homotopic to Stein fillable transverse structures on $-Y$, sharing the same weak filling; then \cite[Remark 2.13]{G-fillability} implies that in this case $c^+(\xi)\neq T_{[V]}$ for any vector whose full path ends correctly. 
Lemma \ref{lemma:realised} and Corollary \ref{cor:red} then tell us that $c^+(\xi)=T_{[V_1]}+T_{[V_2]}+\cdots$ is in $HF_\text{red}(Y,\s_\xi)$, where $V_1$ and $V_2$ are distinct realised characteristic vectors, both inducing $\s_\xi$ on $-Y$ and such that $M(V_1)=M(V_2)$.

\subsection{Manifolds of type B}
\label{subsection:B}
For this type of Seifert fibred spaces there can be more than two possible negative twisting numbers; however, all the structures can be described using an explicit procedure which is based on the case of $-\Sigma(2,3,6k-1)$ done in \cite{GvHM} by Ghiggini and Van Horn-Morris. Their argument has been also applied in \cite{Tosun} for $-\Sigma(2,3,6k+1)$.

Let $\Gamma^*$ be the standard graph of $-Y$ of type B. Fix $G\subset\Gamma^*$ a subgraph isomorphic to one of a torus bundle in Figure \ref{Torus}. We say that $-Y$ is \emph{short} when for each leg of $\Gamma^*$ there is at most one vertex which is not in $G$. We now briefly sketch the procedure for a short manifold $-Y$ of type B; at the end of section we show that the general case follows easily. 

Given $M$ a Seifert fibred torus bundle over a circle, from \cite{Honda2} we know that the negative-twisting structures on $M$ are all obtained by adding Giroux torsion along a fibre, starting from the unique Stein fillable structure $\zeta_0$ obtained by blowing down the standard graph. For every $\s\in\Spin^c(-Y)$ we denote by $\xi_l$ for $l=1,...,k$ the structures on $(-Y,\s)$ with twisting number equal to $\overline{\text{tw}}(-Y)$; hence, the integer $k$ is the number of realised characteristic vectors inducing $\s$. We construct $\xi_{ab}$ for $1\leq a<b\leq k$ on $(-Y,\s)$, forming a pyramid of size $k$ as in Figure \ref{Pyramid}, as follows: 
add $b-a$ times Giroux torsion to $\zeta_0$, then $\xi_{ab}$ is gotten from the latter structure, which is weakly fillable \cite{DG}, by performing $(-1)$-contact surgery on an $\ell$-component link made by the vertices in $\Gamma^*\setminus G$.

Any graph $G$ in Figure \ref{Torus} represents a torus bundle $M$, which is $0$-surgery on a regular fibre of any Seifert fibred space $M_0$ obtained by removing one leg from $G$. We can then compute the size of the exceptional orbit in the resulting monodromy, as the regular fibre of $M_0$ is a knot with Seifert fibred complement and thus produces a periodic monodromy on the torus. The regular fibre of $M$ corresponds to a regular orbit; therefore, its size is determined by $\lcm(\mathfrak a_1,\mathfrak a_2,...)$ where the $\mathfrak a_i$'s are the denominators of the Seifert coefficients of $M$. These numbers are crucial for the procedure that we are going to describe; hence, we write them in Table \ref{Numbers}.

\begin{table}[ht]
\centering
\begin{NiceTabular}{|c|c|c|c|c|c|c|}[
  cell-space-top-limit=4pt,
  cell-space-bottom-limit=4pt
]
\hline
$M$ & 1st leg & 2nd leg & 3rd leg & 4th leg & Regular fibre & Monodromy \\
\hline
$M(-1;\frac{1}{2},\frac{1}{3},\frac{1}{6})$ & 3 & 2 & 1 & -- & 6 & $AB$  \\
$M(-2;\frac{1}{2},\frac{2}{3},\frac{5}{6})$ & 3 & 2 & 1 & -- & 6 & $(AB)^5$ \\
$M(-1;\frac{1}{2},\frac{1}{4},\frac{1}{4})$ & 2 & 1 & 1 & -- & 4 & $ABA$\\
$M(-2;\frac{1}{2},\frac{3}{4},\frac{3}{4})$ & 2 & 1 & 1 & -- & 4 & $(ABA)^3$\\
$M(-1;\frac{1}{3},\frac{1}{3},\frac{1}{3})$ & 1 & 1 & 1 & -- & 3 & $(AB)^2$\\
$M(-2;\frac{2}{3},\frac{2}{3},\frac{2}{3})$ & 1 & 1 & 1 & -- & 3 & $(AB)^4$\\
$M(-2;\frac{1}{2},\frac{1}{2},\frac{1}{2},\frac{1}{2})$ & 1 & 1 & 1 & 1 & 2 & $(AB)^3$ \\
\hline
\end{NiceTabular}
\caption{\smaller[1]{The size of the exceptional and regular orbits of the monodromy of $M$; we order the legs as in Figure \ref{Torus}. Note that $\text{MCG}(T^2)\simeq\text{SL}(2;\Z)$ is presented by $\langle A,B\:|\:(AB)^6=\Id,\:ABA=BAB\rangle$, where $A,B$ are positive Dehn twists along any pair of generators of $H_1(T^2;\Z)$.}}
\label{Numbers}
\end{table}  

Let $B=(b_1,...,b_{|\Gamma^*|})$ be the vector such that $b_l$ is either zero when $S_l\in G$, or the number in the corresponding column of Table \ref{Numbers} otherwise. Note that $\Gamma^*$ may have more legs than $G$; hence, this means that we may still have additional vertices which are attached to the centre of $\Gamma^*$ but that are not in $G$.  

\begin{lemma}
 \label{lemma:j}
 Let $-Y$ be of type B and short, and $V_1,...,V_k\in\emph{Char}(\Gamma^*,\s)$ the realised characteristic vectors. Assume that $\mathcal F(V_1)<\cdots<\mathcal F(V_k)$, then $V_{\frac{j+k+1}{2}}=(j+k-1)B+V_1$ for $j\in\{1-k,3-k,...,k-1\}$; moreover, one has that $M(V_1)=\cdots=M(V_k)$. 
\end{lemma}
\begin{proof}
 There are finitely many cases, as we have seven torus bundles and all the possible values of $B$ are given by Table \ref{Numbers}. We can check by hand that there exists a unique primitive integral vector $R\in\Ker Q_G$ such that $r_l>0$ for any $S_l\in G$; moreover, we have $Q_*\overline R=B$ where $\overline R\lvert_G=R$ and $\overline R$ is zero on $\Gamma^*\setminus G$. If $U=(V_a-V_b)/2=mB$ for an $m\in\Z$ then $Q_*^{-1}U=m\overline R\in\Z^{|\Gamma^*|}$; conversely, if $W=Q_*^{-1}U\in\Z^{|\Gamma^*|}$ then one has $Q_GW\lvert_G+EW\lvert_{\Gamma^*\setminus G}=U\lvert_G=0$, where $E$ is the incidence matrix of $G$ and $\Gamma^*\setminus G$, thus assuming $W\lvert_{\Gamma^*\setminus G}=0$ gives that $W$ is a multiple of $\overline R$ and then $U$ is a multiple of $B$.
 
 Let us prove that $W$ is zero on $\Gamma^*\setminus G$. By assumption we have \[\sum_{l\in\Gamma^*\setminus G}b_lw_l=R^TEW\lvert_{\Gamma^*\setminus G}=0\:,\] thus by contradiction there should be an $l^\pm$ such that $w_{l^+}>0$ and $w_{l^-}<0$. Since $V_a,V_b$ are realised, we certainly have $|u_l|\leq -m(l)-2$ where $S_l\in\Gamma^*\setminus G$. Let $\frac{\mathfrak b}{\mathfrak a}$ be the Seifert coefficient of $G$ corresponding to the leg to which $S_l$ is attached, or $\frac{0}{1}$ when $S_l$ is attached to the centre; then the central coordinate of $W$ is equal to \[w_1=-(\mathfrak a\cdot m(l)+\mathfrak b)w_l+\mathfrak a\cdot u_l\hspace{0.5cm}\text{ for any }\hspace{0.5cm}S_l\in\Gamma^*\setminus G\:.\] Now, we have that $w_1\geq2\mathfrak a-\mathfrak b>0$ in the case of $w_{l^+}$ while $w_1\leq\mathfrak b-2\mathfrak a<0$ in the case of $w_{l^-}$, which is impossible. We can then write each $V_1,...,V_k$ as $V_{\sigma(j)}$ where $\sigma(j)=\frac{j+k+1}{2}$, while $|j|\leq k-1$ and $j+k\equiv1\text{ mod }2$. 
 
 To show that all the vectors have the same Maslov grading, we need that $V_{\sigma(j+1)}^TQ_*^{-1}V_{\sigma(j+1)}=V_{\sigma(j)}^TQ_*^{-1}V_{\sigma(j)}$, and then that $(V_{\sigma(j)}+B)^TQ_*^{-1}B=(V_{\sigma(j)}+B)^T\overline R=0$ for each $j$. Since each $V_{\sigma(j)}$ coincides with $Z_\text{can}$, the canonical vector of $G$, on $G$, we are only left to show that $Z_\text{can}^TR=0$ for each one of the seven possible $G$'s. The vector $Z_\text{can}$ induces an $\s_0$ on $M$ with torsion first Chern class, thus there exists a $Z$ such that $Q_GZ=Z_\text{can}$ and then $Z_\text{can}^TR=Z^TQ_GR=0$.
 Note that in \cite[Theorem 3.12]{G-} Ghiggini proves this result for $-\Sigma(2,3,6k-1)$ geometrically.
\end{proof}

We now need to specify a different notation in order to simplify the construction; we do this mainly because we often reference \cite{GvHM} and we prefer to keep the notation introduced by Ghiggini and Van Horn-Morris to avoid confusion for the reader. 
We write $M^{\mathbf n}$ where $\mathbf n=(n_1,...,n_\ell)$ for the manifold $-Y$ such that $n_l$ is the absolute value of the framing of the $l$-th vertex in $\Gamma^*\setminus G$; for each $\s\in\Spin^c(M^\mathbf n)$ we then denote $\xi_1,...,\xi_{k(\s,\mathbf n)}$ by $\eta^\mathbf n_{0,\:j}$ and we take $j$ from Lemma \ref{lemma:j} applied to $M^\mathbf n$. Note that we know from the previous subsection that $\eta^\mathbf n_{0,\:j}$ is obtained by Legendrian surgery on the blow-down of $\Gamma^*$ and $c^+(\eta^\mathbf n_{0,\:j})=T_{[V_{\sigma(j)}]}$ where $\sigma(j):=\frac{j+k(\s,\mathbf n)+1}{2}$. In addition, we denote the $\xi_{ab}$'s as $\eta^\mathbf n_{i,\:j}$ for integers $1\leq i=b-a<k(\s,\mathbf n)$ and $|j|\leq k(\s',\mathbf n')-1$ such that $j+k(\s',\mathbf n')\equiv1\text{ mod }2$, where $\mathbf n'=\mathbf n-iB$ and $\s'\in\Spin^c(M^{\mathbf n'})$ is induced by $V_1+iB$. 

\begin{remark}
 \label{remark:layer}
 It is easy to check that $k(\s',\mathbf n')$ coincides with $k(\s,\mathbf n)-i$. In addition, if $\iota$ is the maximal $i$ such that $\eta^\mathbf n_{i,j}$ exists then $k(\s',\mathbf n_*):=k(\s',\mathbf n-\iota B)=1$; in fact, from Lemma \ref{lemma:j} we would have $V_1,V_2$ realised by $\eta^{\mathbf n_*}_{0,j_1}$ and $\eta^{\mathbf n_*}_{0,j_2}$ on $M^{\mathbf n_*}$ such that $V_2-V_1=2B$, and then enough choice for a structure $\eta^\mathbf n_{\iota+1,j}$ on $M^\mathbf n$. This is consistent with the fact that the contact structures $\{\eta^\mathbf n_{i,j}\}$ form a pyramid of size $k(\s,\mathbf n)$ for each $\s\in\Spin^c(M^\mathbf n)$.  
\end{remark}

We observe that in \cite{GvHM} we have $\ell=|\Gamma^*\setminus G|=1$ and $-\mathbf n$ is equal to one less than the $\tb$-number of the Legendrian realisation of the vertex, while $j$ coincides with the rotation number. This correspondence would not hold directly here, as when we blow down, the coordinates of $V_j$, as well as the framing of the surgery, change according to the strict transform, see Theorem \ref{teo:realised}. This is the reason why we symmetrise $j$.

Denote by $(M,\zeta_i)$ the structure obtained by adding Giroux torsion $i$-times to $\zeta_0$. It is known that the contact invariant $\widehat c(\zeta_i)$ vanishes, while $\widehat c(\zeta_0)$ does not as $\zeta_0$ is Stein fillable. Since we want to use $\zeta_i$ to define the structure $\eta^\mathbf n_{i,\:j}$, we use the same strategy from \cite{GvHM} and work with $\widehat{HF}$ with twisted coefficients \cite{OSz-genus}. We refer to \cite{GvHM} for more details about the construction. 

Ozsv\'ath and Szab\'o define Heegaard Floer homology with twisted coefficients as a module over the group algebra $\F[H^1(M;\Z)]$; in particular, whenever $M$ is a rational homology sphere the action of the group algebra vanishes, meaning that this construction only gives more information when $b_1(M)>0$. In our case, we take $M$ as the torus bundle over the circle represented by $G$. The only module that we consider, as in \cite{GvHM}, is $\Lambda:=\F[H^1(M;\frac{\Z}{2})]$, which we identify with the ring of the Laurent polynomials $\F[t^{\frac{1}{2}},t^{-\frac{1}{2}}]$. We denote this homology group by $\underline{\widehat{HF}}(M)$.

It is proved \cite{OSz-contact} that there is a twisted contact invariant $\underline{\widehat c}(\zeta_i)$ which is defined as an equivalence class in $\underline{\widehat{HF}}(M,\s_{\zeta_i})$ up to multiplication for an invertible element $t^{\pm\frac{\ell}{2}}$ with $\ell\in\Z$; moreover, such an invariant is non-vanishing when a structure is weakly fillable, as $\zeta_i$ is. We denote by $\s_0$ the spin structure on $M$ which supports $\zeta_i$ for every $i\geq0$. We also recall that from \cite{OSz-genus} and \cite[Sections 3 and 4]{GvHM} we have Stein cobordisms $(V_\mathbf n,J_{\mathbf n})$ induced by the Legendrian surgery on a Legendrian realisation of the link $L\subset M$, represented by the vertices in $\Gamma^*\setminus G$, and $(W,J)$ from $(M,\zeta_{i+1})$ to $(M,\zeta_i)$ which reduces Giroux torsion by one, which give the following maps \[\underline{\widehat F}_{\overline V_\mathbf n,J_{\mathbf n}}:\widehat{HF}(-M^\mathbf n,\s)\longrightarrow\underline{\widehat{HF}}(-M,\s_0)\hspace{0.5cm}\text{ and }\hspace{0.5cm}\underline{\widehat F}_{\overline W,J}:\underline{\widehat{HF}}(-M,\s_0)\longrightarrow\underline{\widehat{HF}}(-M,\s_0)\] where 
\begin{equation}
 \underline{\widehat F}_{\overline W,J}(\alpha)=(t^{\frac{1}{2}}+t^{-\frac{1}{2}})\cdot\alpha\hspace{0.5cm}\text{ for any }\hspace{0.5cm}\alpha\in\underline{\widehat{HF}}(-M,\s_0)\:.\label{eq:value}
\end{equation} 
In addition, we have that $\underline{\widehat F}_{\overline V_\mathbf n,J_{\mathbf n}}(\widehat c(\eta^\mathbf n_{i,\:j}))=\underline{\widehat c}(\zeta_i)$ for every $j$ such that $|j|\leq k(\s,\mathbf n)-i-1$ and $j+k(\s,\mathbf n)\equiv i+1\text{ mod }2$, the invariant $\underline{\widehat c}(\zeta_0)$ is a generator of $\underline{\widehat{HF}}(-M,\s_0)$, and $\underline{\widehat F}_{\overline W,J}(\underline{\widehat c}(\zeta_i))=\underline{\widehat c}(\zeta_{i+1})$; moreover, according to the argument in \cite[Lemma 4.13]{GvHM}, the involution $\J$, which acts as $t\mapsto t^{-1}$, together with the fact that each $\zeta_i$ for $i\geq0$ is self-conjugate can be used to remove the ambiguity in the definition of the twisted cobordism maps in \cite{OSz-genus}. 

We can compute $\underline{\widehat{HF}}(M,\s_0)$ for each of the seven torus bundle over the circle. This is done in \cite[Lemma 4.9]{GvHM} for $M(-1,\frac{1}{2},\frac{1}{3},\frac{1}{6})\simeq S^3_0(T_{2,3})$ using the exact triangle with twisted coefficients \cite{OSz-genus}. The same argument applies here, we give a brief description in the following lemma.

\begin{lemma}
 \label{lemma:twisted}
 For each $M$ as above we have that $\underline{\widehat{HF}}(M,\s_0)\simeq\Lambda_*\oplus\Lambda_{*-1}$; moreover, the contact invariant $\underline{\widehat c}(\zeta_0)$ can be identified with $[1]\subset\{0\}\oplus\Lambda_{-*}\subset\underline{\widehat{HF}}(-M,\s_0)$. The values of the Maslov grading are listed in Table \ref{Homology}.    
\end{lemma}

\begin{table}[t]
\centering
\begin{NiceTabular}{|c|c|c|c|c|}[
  cell-space-top-limit=4pt,
  cell-space-bottom-limit=4pt
]
\hline
$M$ & Full paths & $M(V)$ & $\widehat{HF}_*(M)$ & $\underline{\widehat{HF}}_*(M,\s_0)$ \\
\hline\hline
$M(-1;\frac{1}{2},\frac{1}{3},\frac{1}{6})$ & $(1,0,-1,-4)$ & $-\frac{1}{2}$ & $\F_{-\frac{1}{2}}\oplus\F_{-\frac{3}{2}}$ & $\Lambda_{-\frac{1}{2}}\oplus\Lambda_{-\frac{3}{2}}$  \\
\hline
$M(-2;\frac{1}{2},\frac{2}{3},\frac{5}{6})$ & $(0,0,0,0,0,0,0,0,0)$ & $\frac{3}{2}$ & $\F_{\frac{3}{2}}\oplus\F_{\frac{1}{2}}$ & $\Lambda_{\frac{3}{2}}\oplus\Lambda_{\frac{1}{2}}$ \\
\hline
\Block{2-1}{$M(-1;\frac{1}{2},\frac{1}{4},\frac{1}{4})$} & $(1,0,-2,-2)$ & $-\frac{1}{4}$ & \Block{2-1}{$\F_{\frac{1}{4}}\oplus\F_{-\frac{1}{4}}\oplus\F_{-\frac{3}{4}}\oplus\F_{-\frac{5}{4}}$} & \Block{2-1}{$\Lambda_{-\frac{1}{4}}\oplus\Lambda_{-\frac{5}{4}}$} \\
 & $(-1,0,0,4)$ & $-\frac{3}{4}$ & & \\
 \hline
\Block{2-1}{$M(-2;\frac{1}{2},\frac{3}{4},\frac{3}{4})$} & $(0,0,0,0,0,0,0,0)$ & $\frac{5}{4}$ & \Block{2-1}{$\F_{\frac{5}{4}}\oplus\F_{\frac{3}{4}}\oplus\F_{\frac{1}{4}}\oplus\F_{-\frac{1}{4}}$} & \Block{2-1}{$\Lambda_{\frac{5}{4}}\oplus\Lambda_{\frac{1}{4}}$} \\
 & $(-2,0,0,0,2,2,0,0)$ & $-\frac{1}{4}$ & & \\
\hline
\Block{3-1}{$M(-1;\frac{1}{3},\frac{1}{3},\frac{1}{3})$} & $(1,-1,-1,-1)$ & 0 & \Block{3-1}{$\F^2_{\frac{1}{3}}\oplus\F_0\oplus\F^2_{-\frac{2}{3}}\oplus\F_{-1}$} & \Block{3-1}{$\Lambda_0\oplus\Lambda_{-1}$} \\
 & $(-1,3,1,-1)$ & $-\frac{2}{3}$ & & \\
 & $(-1,3,-1,1)$ & $-\frac{2}{3}$ & & \\
\hline
\Block{3-1}{$M(-2;\frac{2}{3},\frac{2}{3},\frac{2}{3})$} & $(0,0,0,0,0,0,0)$ & 1 & \Block{3-1}{$\F_{1}\oplus\F^2_{\frac{2}{3}}\oplus\F_0\oplus\F^2_{-\frac{1}{3}}$} & \Block{3-1}{$\Lambda_1\oplus\Lambda_{0}$} \\
 & $(-2,0,0,0,2,2,0)$ & $-\frac{1}{3}$ & & \\
 & $(-2,0,2,0,0,2,0)$ & $-\frac{1}{3}$ & & \\
\hline
\Block{4-1}{$M(-2;\frac{1}{2},\frac{1}{2},\frac{1}{2},\frac{1}{2})$} & $(0,0,0,0,0)$ & $\frac{1}{2}$ & \Block{4-1}{$\F^4_{\frac{1}{2}}\oplus\F^4_{-\frac{1}{2}}$} & \Block{4-1}{$\Lambda_{\frac{1}{2}}\oplus\Lambda_{-\frac{1}{2}}$}  \\
 & $(-2,2,2,0,0)$ & $-\frac{1}{2}$ & &  \\
 & $(-2,2,0,2,0)$ & $-\frac{1}{2}$ & &  \\
 & $(-2,2,0,0,2)$ & $-\frac{1}{2}$ & &  \\
\hline
\end{NiceTabular}
\caption{\smaller[1]{The Heegaard Floer homology of the seven torus bundles in Figure \ref{Torus}. Note that $\widehat c(\zeta_0)=T_{[W_\text{can}]}$ where $W_\text{can}$ is the first vector for each $M$ in the table, all the other full paths contain loops. In addition, we have that $M(\widehat c(\zeta_0))=-d_3(\zeta_0)=-M(W_\text{can})$.}}
\label{Homology}
\end{table} 

\begin{proof}
 We use the full path algorithm, which holds for $M$ as we know from Section \ref{section:two}, to compute $\widehat{HF}(M)$ first, see Table \ref{Homology}. At this point we apply the exact triangle 
 \begin{equation*}
    \begin{tikzcd}
 \widehat{HF}(M')[t^{\frac{1}{2}},t^{-\frac{1}{2}}] \arrow[dr] & & \widehat{HF}(M'')[t^{\frac{1}{2}},t^{-\frac{1}{2}}] \arrow[ll,"F\otimes\Lambda" swap] \\
 & \underline{\widehat{HF}}(M) \arrow[ur] &  
\end{tikzcd}
\end{equation*}
 with both untwisted and twisted coefficients. The manifold $M''$ is obtained by increasing by one the framing of the last leg of $G$, while $M'$ is the one presented by the graph $G$ without the last leg; note that $M''$ and $M'$ are $L$-spaces. In the first case we obtain that the map $F$ is zero by counting the dimension of the groups; in addition, in the second case, where the corresponding map is $F\otimes\Lambda$ by \cite[Theorem 3.2]{GvHM}, this immediately implies that $\underline{\widehat{HF}}(M)\simeq\widehat{HF}(M)[t^{\frac{1}{2}},t^{-\frac{1}{2}}]$. We obtain the first claim by observing that the gradings of the other two maps are the same in both cases.

 Regarding the second claim, we know \cite{OSz-contact} that $M(\zeta_0)=-d_3(\zeta_0)$ because $\s_0$ is spin, and then $c_1(\s_0)$ is torsion; in particular, the contact invariant has a well-defined Maslov grading. It follows from the discussion above and the computation in Table \ref{Homology} that $\underline{\widehat c}(\zeta_0)$ lives in the unique summand of $\underline{\widehat{HF}}(-M,\s_0)$ with the correct grading.
\end{proof}

\begin{figure}[ht]  
 \def\svgwidth{0.57\textwidth}
        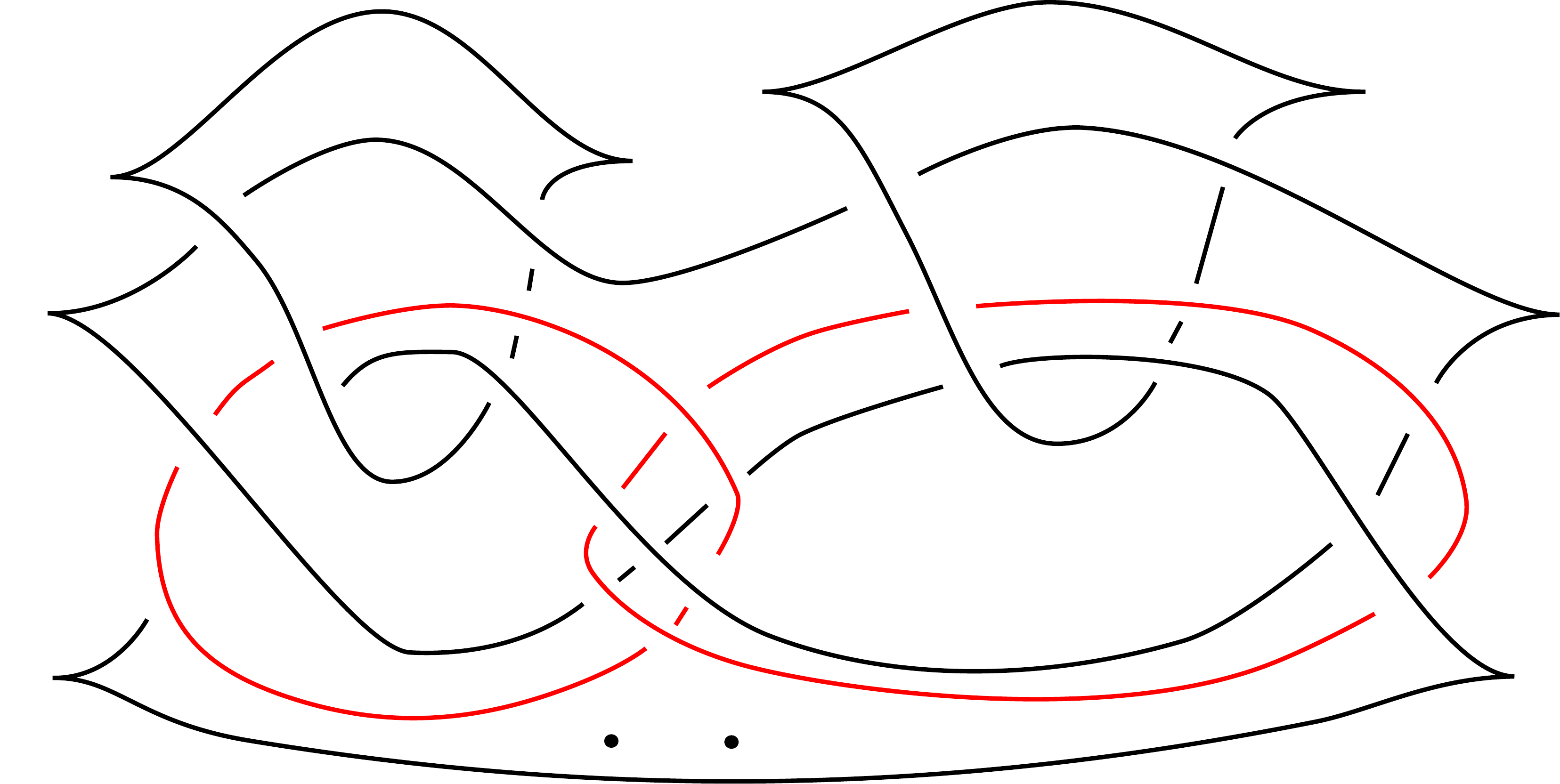
     \caption{\smaller[1]{The link $L$ in $(M,\zeta_i)$. Since the trefoil component is standard, Legendrian surgery on it gives the unique Stein fillable structure on $T^3$; moreover, we have that $\Sigma$ is a fibred surface in $\#^2S^1\times S^2$, and after capping it off it becomes a fibre of $T^3=T^2\times S^1\rightarrow S^1$. The manifold $(M,\zeta_i)$ is then obtained by performing $(\pm1)$-contact surgery on some push-off's of the curves $A,B\subset\Sigma$ along this fibration, according to the monodromies in Table \ref{Numbers}. The link $L$ is transverse to the toric fibres, and it is given by the corresponding orbits of the monodromy of $M$; one has that $|L\cap\Sigma|$ is the sum of the numbers in the columns of Table \ref{Numbers} where $L$ has a component.}}
     \label{TypeB}
 \end{figure}     

The manifold $(M,\zeta_0)$ is obtained from the Stein fillable contact structure $\xist$ on $T^3$, which can be thought as the cap off of a fibred surface $\Sigma$ in $\#^2S^1\times S^2$ as in Figure \ref{TypeB} and whose monodromy is the identity, by adding positive Dehn twists according to Table \ref{Numbers}. More specifically, we take the curves $A,B\subset \Sigma =T^2\setminus\mathring{D^2}$ whose Seifert framing is $-1$, and we push as many of them off as many Dehn twists we need, each one in a parallel fibre. In this way, each curve becomes an unknot with $\tb$-number equal to $-1$ in $(T^3,\xist)$, and Legendrian surgery on all of them gives $(M,\zeta_0)$. In order to obtain the structure $\zeta_i$ we need to add $i$ layers of Giroux torsion to $\zeta_0$; this is done as in \cite{GvHM} by composing the monodromy with $(AB^{-1}A)^4$ for each layer. Since here $B$ is a negative Dehn twist, in this case we perform $(+1)$-contact surgery. 

The Stein cobordism $(W,J)$ from $(M,\zeta_{i+1})$ to $(M,\zeta_i)$ which removes one layer of Giroux torsion consists of eight Legendrian surgeries, changing each block $AB^{-1}A$ of the  monodromy of $(M,\zeta_{i+1})$ to $ABA$. Since $(ABA)^4=(AB)^6=\Id$ in $\text{SL}(2;\Z)$, the resulting monodromy represents $(M,\zeta_i)$. Denote by $\mathcal C$ the $8$-component link mentioned above, we can do Legendrian surgery on the link $\mathcal L\cup \mathcal C$ in $(M,\zeta_{i+1})$ where $\mathcal L$ is a Legendrian realisation of $L$. In \cite[Figure 3]{GvHM} the link $\mathcal C$ is shown for $M\simeq S^3_0(T_{2,3})$, but Van Horn-Morris presents equivalent constructions for every torus bundle in his thesis, see \cite{vHM}; these are consistent with the description we give here. We then have the following commutative diagram
\begin{equation}
    \begin{tikzcd}
 \Lambda_{-d_3(\zeta_0)} & & \widehat{HF}(-M^{\mathbf m},\s') \arrow[ll,"\underline{\widehat F}_{\overline V_{\mathbf m},J_{\mathbf m}}" swap] \\
 \Lambda_{-d_3(\zeta_0)} \arrow[u,"\widehat F_{\overline W,J}"] & & \widehat{HF}(-M^\mathbf n,\s) \arrow[u,"\widehat F_{\overline W',J'}" swap] \arrow[ll,"\underline{\widehat F}_{\overline V_{\mathbf n},J_{\mathbf n}}"]
\end{tikzcd}\label{eq:diagram}
\end{equation} 
where $\s'$ is the $\Spin^c$-structure induced by the realised characteristic vector $V_1-B$ and $\mathbf m=\mathbf n+B$. 

We can see that the manifold appearing in the upper right corner is $M^{\mathbf m}$ in the following way. In \cite[Section 4]{GvHM} the link $\mathcal L\cup \mathcal C$ appears on a page of a genus one abstract open book decomposition $(S,\phi)$ for $M$, compatible with $\zeta_{i+1}$; more specifically, Ghiggini and Van Horn-Morris describe how to read the monodromy of the torus bundle, see Table \ref{Numbers}, from the given open book in \cite[Figure 3]{GvHM}. In their paper $\mathcal L$ is a knot; hence, it is drawn as a longitude of the (punctured) torus $S$. 
In our Figure \ref{TypeB} we have a similar presentation: we show a page of the fibration over $S^1$ of $M$ and $L$ is transverse to the fibres; therefore, in our case $\mathcal L$ is a Legendrian realisation of a link, and each component interacts with the monodromy curves in Figure \ref{Numbers}, and then to the corresponding ones on $S$ pictured in \cite[Figure 3]{GvHM}, as many times as prescribed by Table \ref{Numbers}. 
Therefore, as shown in \cite[Proposition 4.1]{GvHM}, after Legendrian surgery on $\mathcal C$ simplifying the monodromy increases the $\tb$-number of each component of $\mathcal L\subset(M,\zeta_{i+1})$ precisely by $b_l$ for any $S_l\in\Gamma^*\setminus G$, where $b_l$ is the corresponding coordinate of the vector $B$ in Lemma \ref{lemma:j}. 

\begin{remark}
 \label{remark:remark}
 From \cite[Lemma 4.12]{GvHM} we have that $\underline{\widehat F}_{\overline V_{\mathbf n},J_{\mathbf n}}(\eta^\mathbf n_{0,j})=t^{\frac{j}{2}}$ for every $|j|\leq k(\s,\mathbf n)-1$ and $j+k(\s,\mathbf n)\equiv1\emph{ mod }2$, and it maps each $\widehat c(\eta^\mathbf n_{i,j})$ to the contact invariant $\widehat c(\zeta_i)$, whose Maslov grading does not depend on $i$ because of Lemma \ref{lemma:twisted}. Since $\Spin^c$-cobordism maps have well-defined degree shift, we have that every structure $\eta^\mathbf n_{i,j}$ has the same $d_3$-invariant and induced $\Spin^c$-structure, and thus the same homotopy type. Furthermore, Lemma \ref{lemma:c} implies that $c^+(\eta^\mathbf n_{i,j})\in HF_\emph{red}(-M^\mathbf n,\s)$. 
\end{remark} 

Based on the previous discussion, we can now use an inductive argument to find how to express $\widehat c(\eta^\mathbf n_{i,j})$ in terms of the homology classes $\{T_{[V_1]},...,T_{[V_{k(\s,\mathbf n)}]}\}=\{\widehat c(\eta^\mathbf n_{0,j})\}$ in $\widehat{HF}^\text{od}(-M^\mathbf n,\s)$. Note that these elements are linearly independent from Theorem \ref{teo:fullpath}; moreover, as we mentioned in Subsection \ref{subsection:invariant}, there is a (canonical) isomorphism between $\widehat{HF}^\text{od}$ and $HF_\text{red}\cap\Ker U$.

\begin{cor}
 \label{cor:invariant}
 The contact invariant $\widehat c(\eta^\mathbf n_{i,j})$ is equal to $\widehat c(\eta^\mathbf n_{i-1,j-1})+\widehat c(\eta^\mathbf n_{i-1,j+1})$ in $\widehat{HF}(-M^\mathbf n,\s)$ for any $0<i<k(\s,\mathbf n)$ and $j$ such that $|j|\leq k(\s,\mathbf n)-i-1$ and $j+k(\s,\mathbf n)\equiv i+1\emph{ mod }2$.   
\end{cor}
\begin{proof}
 The proof is based on the one of \cite[Theorem 1.3]{GvHM}. Let $\mathbf o=\mathbf n-B$, and note that because we are assuming $i>0$ then $k(\s,\mathbf n)\geq2$ thus the manifold $M^\mathbf o$ exists. We start with $i=1$: from Equation \eqref{eq:value}, Diagram \eqref{eq:diagram} and Remark \ref{remark:remark} we have that \[\underline{\widehat F}_{\overline V_{\mathbf n},J_\mathbf n}(\widehat c(\eta^\mathbf n_{1,j}))=\underline{\widehat F}_{\overline W,J}(\underline{\widehat F}_{\overline V_{\mathbf o},J_{\mathbf o}}(\widehat c(\eta^{\mathbf o}_{0,j})))=\underline{\widehat F}_{\overline W,J}(t^{\frac{j}{2}})=(t^{\frac{1}{2}}+t^{-\frac{1}{2}})t^{\frac{j}{2}}=t^{\frac{j+1}{2}}+t^{\frac{j-1}{2}}\:.\] In the same way, we have that $\underline{\widehat F}_{\overline V_{\mathbf n},J_\mathbf n}(\widehat c(\eta^\mathbf n_{0,j\pm1}))=t^{\frac{j\pm1}{2}}$ and then \[\underline{\widehat F}_{\overline V_{\mathbf n},J_\mathbf n}(\widehat c(\eta^\mathbf n_{1,j})+\widehat c(\eta^\mathbf n_{0,j-1})+\widehat c(\eta^\mathbf n_{0,j+1}))=0\:.\] The fact that $\underline{\widehat c}(\zeta_1)$ is non-zero implies that $\widehat c(\eta^\mathbf n_{1,j})$ is non-zero. By Remark \ref{remark:remark}, and the fact that for a short manifold of type B any initial vector whose full path ends correctly is realised, we obtain that $\widehat c(\eta^\mathbf n_{1,j})$ is an element of the subgroup of $\widehat{HF}^\text{od}(M^\mathbf n,\s)$ generated by $T_{[V_1]},...,T_{[V_{k(\s,\mathbf n)}]}$. We conclude by observing that the map $\underline{\widehat F}_{\overline V_{\mathbf n},J_\mathbf n}$ is injective by construction on this subgroup.

 We finish with the inductive step when $i>1$: \[\begin{aligned}\underline{\widehat F}_{\overline V_{\mathbf n},J_\mathbf n}(\widehat c(\eta^\mathbf n_{i,j}))&=\underline{\widehat F}_{\overline W_,J}(\underline{\widehat F}_{\overline V_{\mathbf o},J_{\mathbf o}}(\widehat c(\eta^{\mathbf o}_{i-1,j})))=\\ &=\underline{\widehat F}_{\overline W_,J}(\underline{\widehat F}_{\overline V_{\mathbf o},J_{\mathbf o}}(\widehat c(\eta^\mathbf o_{i-2,j-1})))+\underline{\widehat F}_{\overline W_,J}(\underline{\widehat F}_{\overline V_{\mathbf o},J_{\mathbf o}}(\widehat c(\eta^\mathbf o_{i-2,j+1})))=\\ &=\underline{\widehat F}_{\overline V_{\mathbf n},J_\mathbf n}(\widehat c(\eta^\mathbf n_{i-1,j-1}))+\underline{\widehat F}_{\overline V_{\mathbf n},J_\mathbf n}(\widehat c(\eta^\mathbf n_{i-1,j+1}))\end{aligned}\] again by Diagram \ref{eq:diagram}. This leads to \[\underline{\widehat F}_{\overline V_{\mathbf n},J_\mathbf n}(\widehat c(\eta^\mathbf n_{i,j})+\widehat c(\eta^\mathbf n_{i-1,j-1})+\widehat c(\eta^\mathbf n_{i-1,j+1}))=0\] from which we conclude exactly as before.
\end{proof}

We finish the section by showing that the number of negative-twisting structures on any $-Y$ of type B equals the upper bound, see Subsection \ref{subsection:background}. We first need a lemma about the twisting numbers of manifolds of type B. Note that from Ghiggini-Massot's algorithm for each torus bundle $M=M(e_0;\frac{\mathfrak b_1}{\mathfrak a_1},...,\frac{\mathfrak b_{n'}}{\mathfrak a_{n'}})$ the twisting numbers are in an arithmetic progression of ratio $\mathfrak k:=\lcm(\mathfrak a_1,\mathfrak a_2,...)$; namely, one has $\text{tw}(M,\zeta_i)=1-(i+1)\mathfrak k$, and in particular $q$ is a twisting number if and only if $q\cdot\mathfrak b_j\equiv-1\text{ mod }\mathfrak a_j$ for each $j$. 

\begin{lemma}
 \label{lemma:twistingB}   
 For any $-Y=M(e_0;\frac{b_1}{a_1},...,\frac{b_n}{a_n})$ of type B we have that the negative twisting numbers of its tight structures are of the form $-\overline q-k\mathfrak k$, where $\overline q=|\overline{\emph{tw}}(-Y)|=\mathfrak t-1$ and $k$ is a non-negative integer. Furthermore, we have that the best upper approximations with denominator $\overline q+k\mathfrak k>1$ for $\frac{\mathfrak b_j}{\mathfrak a_j}$ and $\frac{b_j}{a_j}$ are the same.   
\end{lemma}
\begin{proof}
 Set $\mathfrak r_j=\frac{\mathfrak b_j}{\mathfrak a_j}$ and $r_j=\frac{b_j}{a_j}$ for $j=1,...,n$ where $\mathfrak r_j=\frac{0}{1}$ for $j=n'+1,...,n$. Suppose that $-q$ is a twisting number for $-Y$; for any $j$ define $c_j(q)$ as the smallest positive integer such that $\mathfrak r_j<\frac{c_j(q)}{q}$, thus $p_j\geq c_j(q)$ where $p_j$ is the best upper approximation for $r_j$ with denominator $q$. If $\mathfrak r_j<\frac{a}{q}$ for some $a>0$ integer then $a\geq\lfloor q\cdot\mathfrak r_j\rfloor+1$; therefore, we have \[c_j(q)=\lfloor q\cdot\mathfrak r_j\rfloor+1=q\cdot\mathfrak r_j-\{q\cdot\mathfrak r_j\}+1\] and then \[\sum_{j=1}^np_j\geq\sum_{j=1}^nc_j(q)=q\sum_{j=1}^n\mathfrak r_j-\sum_{j=1}^n\{q\cdot\mathfrak r_j\}+n=-e_0q+n-\sum_{j=1}^n\{q\cdot\mathfrak r_j\}\geq-e_0q+n-2\:,\] where we are using the fact that  \[\sum_{j=1}^{n'}\{q\cdot\mathfrak r_j\}\leq\sum_{j=1}^{n'}\frac{\mathfrak a_j-1}{\mathfrak a_j}=2\hspace{0.5cm}\text{ for every }\hspace{0.5cm}q>0\] as each torus bundle $M$ satisfies Equation \eqref{eq:6}. One has that equality holds if and only if $-q$ is a twisting number for $M$; this can be checked with a simple computation because this is equivalent to $q\cdot\mathfrak b_j\equiv-1\text{ mod }\mathfrak a_j$ for each $j$. From Ghiggini-Massot's algorithm all the previous inequalities are sharp, and then $q$ is of the required form.

 Now let $\frac{\mathfrak p_j}{\overline q+k\mathfrak t}$ be the best upper approximation for $\mathfrak r_j$ and suppose that $r_j<\frac{\mathfrak p_j}{\overline q+k\mathfrak t}$, then it follows immediately that $\frac{\mathfrak p_j}{\overline q+k\mathfrak t}$ is also the best upper approximation for $r_j$; moreover, we have that $\overline q+k\mathfrak t$ is a twisting number of $-Y$ for the same reason as the case of manifolds of type A with tails in Lemma \ref{lemma:tails}. Conversely, if $r_j\geq\frac{\mathfrak p_j}{\overline q+k\mathfrak t}$ then there is no best upper approximation with given denominator.
\end{proof}

Note that it follows in the same way as \cite[Proposition 2.7]{GvHM} that $\text{tw}(-Y,\eta^\mathbf n_{i,j})=\text{tw}(M,\zeta_i)=1-(i+1)\mathfrak k$ whenever $-Y$ is short.

\begin{prop}
 \label{prop:classificationB}
 Suppose that $-Y$ is a manifold of type B. Then a negative-twisting tight contact structure on $-Y$ is isotopic to $\eta^\mathbf n_{i,j}$ for some $\s\in\Spin^c(-Y)$. Furthermore, such structures are all symplectically fillable and distinguished by their contact invariant $c^+\in HF_\emph{red}$.
\end{prop}
\begin{proof}
 We recall hat each structure $\eta^\mathbf n_{0,j}$ is obtained by blowing down $\Gamma^*$. We start by assuming that $-Y=M^\mathbf n$ is short; hence we have that $\{\eta^\mathbf n_{i,j}\}$ is a pyramid for each $\s\in\Spin^c(M^\mathbf n)$ by Remark \ref{remark:layer}. We saw that any such structure is obtained by $(-1)$-contact surgery on a Legendrian knot in a weakly fillable manifold, namely $(M,\zeta_i)$; therefore, they are weakly fillable themselves, and then strongly fillable as $M^\mathbf n$ is a rational homology sphere. In addition, the fact that these structure have pairwise distinct $c^+$ can be seen from Corollary \ref{cor:invariant}, as each one can be expressed by a different linear combination of $\{\widehat c(\eta^\mathbf n_{0,j})\}$.

 We have to show that the number of the structures of the form $\eta^\mathbf n_{i,j}$, among every $\s$ on $M^\mathbf n$, coincides with the upper bound in Subsection \ref{subsection:background}. We first observe that the structures as $\eta^\mathbf n_{0,j}$ are precisely the ones with twisting number equal to $\overline{\text{tw}}(M^\mathbf n)$ and there are as many as the upper bound, see Propositions \ref{prop:main} and \ref{prop:A}, and \cite[Proposition 5.1]{CM-negative}. We proceed by induction on the maximal size of a pyramid of $M^\mathbf n$.

 If there is only one layer then the proof follows as said above. Suppose that there is more than one layer; then we consider each twisting number $-q$ separately. If $-q=\overline{\text{tw}}(M^\mathbf n)$ then we are done for the same reason. Since the number of the structures $\eta^\mathbf n_{i,j}$ coincides with the ones in $\eta^\mathbf o_{i-1,j}$ where $\mathbf o=\mathbf n-B$, and this holds for every $\Spin^c$-structure, we just have to prove that the upper bound for $-q$ on $M^\mathbf n$ equals the upper bound for $-q+\mathfrak k$ on $M^\mathbf o$; we then conclude by induction.

 From Subsection \ref{subsection:background} we know that we can compute the upper bound leg-by-leg (we only care about the ones where there is a vertex in $\Gamma^*\setminus G$) by comparing the resulting continued fractions for $r_i$ and $\frac{p_i}{q}$, for $M^\mathbf n$, and $r_i'$ and $\frac{p_i'}{q-\mathfrak k}$, for $M^\mathbf o$. From Lemma \ref{lemma:twistingB} one has that $\frac{p_i}{q}$ and $\frac{p_i'}{q-\mathfrak k}$ are also the best upper approximations for $\mathfrak r_i$, the $i$-th Seifert coefficient of the torus bundle $M$. 
 
 We assume $q-\mathfrak k>1$, the other case is similar. We have $-\frac{1}{r_i}=[m^i_1,...,m^i_{k_i}]$ and $-\frac{q}{p_i}=[m^i_1,...,m^i_{k_i-1},n^i_{k_i}]$ where $n^i_{k_i}>m^i_{k_i}$. Then one gets precisely $-\frac{1}{r_i'}=[m^i_1,...,m^i_{k_i}+b_{i}]$, where $b_{i}=\frac{\mathfrak k}{\mathfrak a_i}$ is the corresponding number in Table \ref{Numbers}, and \[-\frac{q-\mathfrak k}{p_i'}=\left\{\begin{aligned}&[m^i_1,...,m^i_{k_i-1},n^i_{k_{i}}+b_{i}]\hspace{0.5cm}\text{ if }-n^i_{k_i}>b_{i}+1\\ &[m^i_1,...,m^i_{k_i-1}+1]\hspace{1.28cm}\text{ if }-n^i_{k_i}=b_{i}+1\:.\end{aligned}\right.\] The contributions for the upper bound are $(n^i_{k_{i}}+b_{i})-(m^i_{k_i}+b_{i})=n^i_{k_{i}}-m^i_{k_i}$ and $[(m^i_{k_i-1}+1)-m^i_{k_i-1}]\cdot(-m^i_{k_{i}}-b_{i}-1)=n^i_{k_{i}}-m^i_{k_i}$ respectively, exactly as we wanted.

 If $-Y$ is not short then it is obtained by Legendrian contact surgery on (some) fibres of a short manifold of type B, exactly as in the case of manifolds of type A with tails. Hence, the claims in the statement are obtained with the same arguments of Proposition \ref{prop:classificationA} and Corollary \ref{cor:red}.
\end{proof}

It follows that the contact invariant of $\eta^{\mathbf n}_{i,j}$ is always a linear combination of $T_{[V_1]},...,T_{[V_{k(\s,\mathbf n)}]}$, the realised characteristic vectors in $\text{Char}(\Gamma^*,\s)$. In Section \ref{section:seven} we determine the coordinates of $c^+$ exactly, depending on the parity of $k(\s,\mathbf n)$.

\section{The correspondence with Heegaard Floer homology}
\label{section:seven}
From our exposition so far we deduce the following corollary for a $-Y$ whose standard graph $\Gamma^*$ has $b_2^+$ equal to 1. 

\begin{cor}
 \label{cor:fillable}
 Every negative-twisting structure $\xi$ on $-Y$ is symplectically fillable; moreover, negative-twisting structures are precisely the ones with $0\neq c^+(\xi)\in HF_\emph{red}(Y)$.  
\end{cor}
\begin{proof}
 This follows from Corollary \ref{cor:red} and Propositions \ref{prop:neg_tw}, \ref{prop:classificationA} and \ref{prop:classificationB}. 
\end{proof}

\subsection{The full path determines the twisting numbers}
Let us take the set $\{V_0,...V_{S-1}\}\subset\text{Char}(\Gamma^*,\overline\s_\text{can})$ of the realised characteristic vectors such that $M(V_l)=M(V_\text{can})$ for $l=0,...,S-1$, where $\s_\text{can}\in\Spin^c(-Y)$ is induced by $V_\text{can}$. It follows from the discussion in Section \ref{section:five} that the $T_{[-V_l]}$'s are the contact invariants of the $\xi_l$'s in Proposition \ref{prop:bigtw} homotopic to the structure $\xi_\text{can}$; therefore, one has $\text{tw}(-Y,\xi_l)=\overline{\text{tw}}(-Y)$. We fix the convention that $V_0$ is the initial vector of $[-V_\text{can}]$, and we compute the values of the possible negative twisting numbers.

\begin{teo}
  \label{teo:twisting}
  There exists a tight structure $\xi$ on $-Y$ such that $\emph{tw}(-Y,\xi)=-q<\overline{\emph{tw}}(-Y)$ if and only if $q=1+\height([V_\emph{can}]+[-V_l])$ for $l=1,...,S-1$. In particular, there are $S$ integers which appear as the negative-twisting number of a contact structure on $-Y$.
\end{teo}

In order to prove this result we need to reformulate both sides of the equivalence, which is the content of the following preliminary lemmate.

Let us first consider a realised characteristic vector $V$. We write $2U:=V-V_\text{can}$, and we define \[X:=\dfrac{1}{2}Q_*^{-1}(V+V_\text{can})=Q_*^{-1}(V_\text{can}+U)\] where $Q_*$ is the intersection matrix of $\Gamma^*$. Then \[V^TQ_*^{-1}V-V_\text{can}^TQ_*^{-1}V_\text{can}=(V-V_\text{can})^TQ_*^{-1}(V+V_\text{can})=4U^TX\:;\] hence, by Proposition \ref{prop:spinc} and Equation \eqref{eq:Maslov2}, the vector $V$ is one of the $V_l$'s introduced at the beginning of the subsection if and only if $U$ is in the right range (for $V$ to be a realised vector), $X\in\Z^{|\Gamma^*|}$ (which ensures that $V$ and $-V_\text{can}$ are the same in $\Spin^c(-Y)$), and $U^TX=0$ (so that $M(V)=M(V_\text{can})$). Moreover, when $l\in\{1,...,S-1\}$ we have \[x_1=e_1^TX=-\height([V_\text{can}]+[-V_l])<-\height[V_\text{can}]\leq0\:.\] 
For vertices on the $i$-th leg we define \[Z^{(i)}:=\mathbf{1}-X^{(i)} \in \mathbb{Z}^{k_i}\:,\] then from $Q_*X=V_{\text{can}}+U$ one gets for a single leg
\begin{equation}\label{eq:leg}
 -Q_*^{(i)} Z^{(i)} = U^{(i)} + qe_1 + e_{k_i}
\end{equation}
for $q:=1-x_1$.

\begin{lemma}
 \label{lemma:orthogonal}
 When $q>1$ we have that $U^TX=0$ if and only if $u_1=0$ and $u_j^{i}(z_j^{i}-1)=0$ for each leg and every $j=1,...,k_i$.  
\end{lemma}
\begin{proof}
 Because the matrix $(-Q_*^{(i)})^{-1}$ has strictly positive entries (it is the inverse of an irreducible positive-definite $M$-matrix) and the right-hand side of Equation \eqref{eq:leg} has non-negative coordinates, we get $z^{i}_j>0$. Since $Z^{(i)}$ is integral, this implies $z_j^{i}\geq 1$ and then $x^{i}_j = 1 - z^{i}_j \leq 0$ for every vertex on every leg. 
 In addition, when $q>1$ one has $x_1=1-q<0$.

 Since the coordinates of $U$ are non-negative, while the ones of $X$ are non-positive, we have that $U^{T}X=0$ is equivalent to $u_1=0$ and $u_j^{i}x_j^{i}= 0$ for every $i$ and $j$.
\end{proof}

For each leg $i$ we define
\begin{equation*}\label{eq:pj}
p_i := z^{i}_1 = 1 - x_1^{i} 
\end{equation*}
where here the index 1 refers to the attachment vertex of the $i$-th leg. Then from the first row of $Q_*X=V_{\mathrm{can}}+U$, using $u_1=0$, we get
\[e_0x_1 + \sum_{i=1}^{n} x_1^{i} = e_0+2
\quad\text{ which is equivalent to }\quad
\sum_{i=1}^{n}(1-p_i)= e_0(1-x_1)+2\:,
\]
and then to
\begin{equation}\label{eq:sum}
\sum_{i=1}^{n} p_i = -e_0q + n - 2\:. 
\end{equation}

In addition, for each leg we write $r_i=\frac{b_i}{a_i}$ for its positive reduced fraction. A standard adjugate
computation on Equation \eqref{eq:leg} yields
\begin{equation}\label{eq:det}
 a_ip_i-b_iq \;=\; 1 + \bigl(D^{(i)}\bigr)^{T}U^{(i)} \:,
\end{equation}
where $D^{(i)} := \operatorname{adj}(-Q_*^{(i)})e_1$ has positive integer entries. This means $a_i p_i - b_iq>0$ and then
\begin{equation*}\label{eq:ineq}
 r_i=\frac{b_i}{a_i}<\frac{p_i}{q}\:.
\end{equation*}

Combining these observations with Lemma \ref{lemma:orthogonal}, in order for a $q$ obtained from the $V_l$'s to be a twisting number we need the following reinterpretation of the best upper approximation condition.

\begin{lemma}[Leg lemma]
 Assume $-Y=M(e_0;r_1,...,r_n)$ and fix $i\in\{1,...,n\}$.  Then for integers $1<p<q$ the fraction $\frac{p}{q}$ is the best upper approximation for $r_i$ if and only if there exist vectors $Z,U\in\Z^{k_i}$ such that 
\begin{enumerate}
    \item $z_1=p$;
    \item $0\leq u_j\leq \overline u_j$ where $2\overline u_j+m_j^i+2$ is the maximal value of the $j$-th coordinate in the $i$-th leg for a realised characteristic vector as in Theorem \ref{teo:realised};
    \item $AZ=U+qe_1+e_{k_i}$ where $-A:=Q_*^{(i)}$ is the intersection matrix of the $i$-th leg of $\Gamma^*$;
    \item $u_j(z_j-1)=0$
\end{enumerate}
for every $j=1,...,k_i$ where $k_i$ is the length of the $i$-th leg.
\end{lemma}

Since the leg lemma concerns a single leg, we opt to skip the leg index $i$ in the notation for simplicity. We still describe the leg as a part of the standard graph in order to keep all the usual assumptions. 

We will prove the Leg lemma by induction on the length of the leg. Let $r_i'$ be the Seifert coefficient of the linear graph obtained by removing the first vertex from the $i$-th leg, then we have $-\frac{1}{r_i}=m_1^i+r_i'$. We set $a_j:=-m_j^i>0$ for the framings of the vertices in the $i$-th leg. The original and cut coefficient are then related by the Möbius transform $\Phi(x)=a_1-\frac{1}{x}$ which has inverse $\Phi^{-1}(y)=\frac{1}{a_1-y}$, as $r_i=\Phi^{-1}(r'_i)$. 

In particular, given $p>1$ we have that $\frac{p}{q}$ is the best upper approximation for $r_i$ if and only if $\frac{a_1p-q}{p}$ is for $r'_i$. Indeed, the value of $\Phi$ on a rational number is given by \[\Phi\left(\frac{k}{h}\right)=a_1-\frac{h}{k}
=\frac{a_1k-h}{k}\:,\] so the new denominator is $k$ and if we have the restriction $k\leq p$ then the new denominator is at most $p$; moreover, the restriction $h<q$ becomes a restriction on the new numerator $a_1k-h$ to be at least $a_1k-(q-1)$. The inequality $r_i<\frac{k}{h}<\frac{p}{q}$ is then equivalent to
\begin{equation*}
r'_i < \Phi\left(\dfrac{k}{h}\right) < \Phi\left(\dfrac{p}{q}\right)
= \dfrac{a_1p-q}{p}\:.
\end{equation*}

\begin{proof}[Proof of the Leg lemma]
 We proceed by induction on the length $k_i$ of the leg. Let us assume that $k_i=1$, then Property (3) becomes $a_1p=u_1+q+1$ while Property (4) is $u_1(p-1)=0$; this is equivalent to $u_1=0$ and then $a_1p-q=1$, which holds if and only if $r_i=\frac{1}{a_1}$ is a Farey neighbour of $\frac{p}{q}$. For such $r_i$ the latter condition is equivalent to be the best upper approximation.

 Now assume $k_i>1$: we start with the if implication. By the first row of Property (3)
 \begin{equation}\label{eq:firstrow}
   z_2 = a_1p - q - u_1\:;
 \end{equation}
 by Property (2) $0\le u_1\le \overline u_1$, which restricts $z_2$ to the interval
 \begin{equation*}\label{eq:z2interval}
  a_1p - q - \overline u_1 \;\leq\; z_2 \;\leq\; a_1p - q \:;
 \end{equation*}
 and from Property (4) we have that $z_1=p>1$ implies $u_1=0$ and $z_2=a_1p-q$.

Let $A'$ be the chain matrix for $a_2,\dots,a_{k_i}$, and write $Z'=(z_2,\dots,z_{k_i})$ and  $U'=(u_2,\dots,u_{k_i})$. Then rows $2,\dots,k_i$ of Property (3) are exactly
\begin{equation}\label{eq:Lprime}
A' Z' = U'+ p\,e_1 + e_{k_i-1} 
\end{equation}
so if Properties (1) -- (4) hold for $(q,p)$ on $A$ then they also hold for $(p,z_2)$ on $A'$ with the bounds $\overline u_2,\dots,\overline u_{k_i}$. We have $z_2=a_1p-q$, thus $\frac{z_2}{p}=\frac{a_1p-q}{p}=\Phi(\frac{p}{q})$. 
If there were a forbidden fraction $\frac{k}{h}$ in $(r_i,\frac{p}{q})$ with $h<q$ and $k\leq p$, then since $p>1$ its transform $\Phi(\frac{k}{h})$ would be a forbidden fraction in $(r'_i,\frac{z_2}{p})$ with denominator $k<p$, and numerator bounded appropriately; that contradicts the inductive hypothesis applied to the cut leg and fraction $\frac{z_2}{p}$. 

We continue by showing the only if implication. We prove the inductive step: assuming the claim is true for chains of length $k_i-1$ with coefficients $a_2,\dots,a_{k_i}$, we prove it for $a_1,\dots,a_{k_i}$. 
Since $p>1$, we have that $\frac{a_1p-q}{p}$ is the best upper approximation of $r_i'$. By inductive hypothesis we have $Z',U'\in\mathbb Z^{k_i-1}$ satisfying the four properties for the cut leg. Extending to the first coordinate, we set $z_1=p$ and $u_1=a_1p-q-z_2=0$ by Properties (1) and (3) from inductive hypothesis, which clearly fulfil all four conditions.
\end{proof} 

We state an alternative version of Ghiggini-Massot's algorithm, so that we can determine possible negative twisting numbers in a more useful way for our applications.

\begin{prop}
 \label{prop:twisting}
 If there exist positive integer numbers $p_1,...,p_n$ such that 
\begin{itemize}
 \item $p_1+...+p_n = -e_0q+n-2$;
 \item there are $Z^{(i)}$ and $U^{(i)}$ as in the Leg lemma when setting $p=p_i$ for $i=1,...,n$,
 \end{itemize} 
 then the Seifert fibred space $-Y=M(e_0;r_1,\dots,r_n)$ admits a tight contact structure with twisting number equal to $-q$. The converse is also true if each best upper approximation $\frac{p_i}{q}$ has $p_i>1$.
\end{prop}
\begin{proof}
 It follows by Theorem \ref{teo:Paolo}. If $p_i>1$ for $i=1,...,n$ then we just need a comparison with the Leg lemma; if $p_i=1$ for some $i$ then the second item still gives us $r_i<\frac{1}{q}$, and we clearly see that there are no forbidden rationals in $(r_i,\frac{1}{q})$, as $h\leq q$ is equivalent to $\frac{1}{q}\leq\frac{1}{h}$.  
\end{proof}

We can now prove Theorem \ref{teo:twisting}, but first we need the following lemma to handle some specific cases which are not covered by Proposition \ref{prop:twisting}.

\begin{lemma}
 \label{lemma:twisting}    
 Consider a $-Y=M(e_0;r_1,...,r_n)$ such that $M=M(e_0;\frac{\mathfrak b_1}{\mathfrak a_1},...,\frac{\mathfrak b_{n-1}}{\mathfrak a_{n-1}})$ is a torus bundle as in Figure \ref{Torus} and $r_i=\frac{\mathfrak b_i}{\mathfrak a_i}$ for $i=1,...,n-1$. If $-q$ is a twisting number then we can find a realised characteristic vector $V$ such that, given $2X=Q_*^{-1}(V+V_\emph{can})$ and $2U=V-V_\emph{can}$, one has $X\in\Z^{|\Gamma^*|},$ $U^TX=0$ and $q=1-x_1=1-e_1^TX$. 
\end{lemma}
\begin{proof}
 From Lemma \ref{lemma:twistingB} we have $q=\overline q+k\mathfrak k$ for $0\leq k\leq K$ where $\mathfrak k=\lcm(\mathfrak a_1,...,\mathfrak a_{n-1})$, $\overline q=|\overline{\text{tw}}(-Y)|$, and $K$ is a positive integer; moreover, the inequality $r_n<\frac{1}{q}$ holds. We set $V=V_\text{can}$ except for the coordinates of the $n$-th leg which are \[(-m_1^n-2\overline q-2k\mathfrak k,-m^n_2-2,...,-m^n_{k_n}-2)\] for $0\leq k\leq K$ where $-\frac{1}{r_n}=[m_1^n,...,m^n_{k_n}]$. Note that when $k=0$ the first coordinate on the $n$-th leg coincides with $-m_1^n-2\overline q$, while when $k=K$ it is equal to \[-m_1^n-2\overline q-2(q-\overline q)\geq-m_1^n+2(m_1^n+1)=m_1^n+2\] because $-m_1^n\geq\overline q+K\mathfrak k+1$, thus $V$ is realised by Theorem \ref{teo:realised}; moreover, one has $\frac{V+V_\text{can}}{2}=V_\text{can}$ everywhere except on the $n$-th leg where is equal to $(1-\overline q-k\mathfrak k,0,...,0)$, while $U$ is zero except on the $n$-th leg where is equal to \[(-m_1^n-1-\overline q-k\mathfrak k,-m_2^n-2,...,m^n_{k_n}-2)\:.\] Using the Schur complement and the fact from the proof of Lemma \ref{lemma:j} that for each $M$ there exists an integral vector $R\in\Ker Q_G$ such that $Z_\text{can}^TR=0$, where $G$ is the standard graph of $M$ and $Z_\text{can}$ its canonical vector, we can solve for $X$ and obtain \[X=Q_*^{-1}\dfrac{V+V_\text{can}}{2}=(Z,0,...,0)\] where $Q_GZ=Z_\text{can}$ such that $z_1=1-\overline q-k\mathfrak k$; hence, we have that $U^TX=0$. The first coordinate of $Z$ on the $i$-th leg is $1+\frac{(z_1-1)\mathfrak b_i-1}{\mathfrak a_i}$, while the remaining ones are then obtained recursively with integer coefficients; for this reason, the vector $Z$, and then $X$, is integral if and only if $z_1\equiv1+\mathfrak b_i^{-1}\text{ mod }\mathfrak a_i$ for $i=1,...,n-1$. Plugging in the value of $z_1$, this reduces to $\overline q\cdot\mathfrak b_i\equiv-1\text{ mod }\mathfrak a_i$ which is indeed satisfied as we mentioned before Lemma \ref{lemma:twistingB}.
 
 We now just need to observe that $e_1^TX=z_1=1-\overline q-k\mathfrak k=1-q$ and we are done.
\end{proof}

\begin{proof}[Proof of Theorem \ref{teo:twisting}]
 We write $\overline q=|\overline{\text{tw}}(-Y)|$.
 Suppose that $-q$ is defined by the height function. 
 Construct vectors $X$ and $U$ as described immediately after the statement of the theorem, take $q=1-x_1$, and set $p_i=1-x_1^{i}\geq1$. By Lemma \ref{lemma:orthogonal} we have $u_1=0$, thus Equation \eqref{eq:sum} holds; hence, by Proposition \ref{prop:twisting} we know that $q$ is obtained from Ghiggini-Massot's algorithm, and then is indeed the twisting number of a tight structure on $-Y$. In addition, when $U$ corresponds to $V_l$ for $l=1,...,S-1$ one has that $q\neq1+\height[V_\text{can}]$, and then $q\neq\overline q$ by Proposition \ref{prop:bigtw}.

 Suppose now that $-q<-1$ is a twisting number, to show that $q=1-x_1$ for some $V_l$ when $p_i>1$ for $i=1,...,n$ we use Proposition \ref{prop:twisting} to find vectors $Z^{(i)},U^{(i)}$ satisfying Equation \eqref{eq:leg} for each leg. In other words  \[-Q_*^{(i)} Z^{(i)} = U^{(i)} + qe_1 + e_{k_i}\hspace{0.5cm}\text{ and }\hspace{0.5cm}u^{i}_j\bigl(z^{i}_j-1\bigr)=0\] with $0\leq u_j^{i} \leq \overline u_j^{i}$ and  $z_j^{i}\geq 1$ for $j=1,...,k_i$. We have that $X^{(i)}= \mathbf{1}-Z^{(i)}$ has non-positive entries on the legs, and we set $u_1=0$ for the central coordinate. From Equation \eqref{eq:sum} this is exactly what is needed so that the first row of $Q_*X=V_{\mathrm{can}}+U$ holds, so we get an integer solution $U$ whose entries are in the right range; finally, Lemma \ref{lemma:orthogonal} gives $U^{T}X=0$. Now set $V=V_{\mathrm{can}}+2U$: it is realised by construction, and \[\dfrac{1}{2}Q_*^{-1}(V_{\mathrm{can}}+V)=Q_*^{-1}(V_{\mathrm{can}}+U)=X\in \mathbb{Z}^{|\Gamma^*|}\] with $x_1=1-q$. Hence, the vector $V$ has the same Maslov grading and induces the same $\Spin^c$-structure as $-V_\text{can}$; this means $V=V_l$ for some $l\in\{0,...,S-1\}$. Since the value of $x_1$ cannot be $-\height[V_\text{can}]$, otherwise $q=\overline q$ in light of Proposition \ref{prop:bigtw}, it has to be $-\height([V_\text{can}]+[-V_l])$ for some $l=1,...,S-1$.

 We need to prove the claim when $p_i=1$ for some $i$, note that we are assuming that $q\neq\overline q$. We cannot have $p_3=1$ if $n=3$: this is true because removing $p_i$ from Equation \eqref{eq:sum} gives that $q$ would be a twisting number for the lens space whose standard graph is $\Gamma^*$ without the 3rd leg, but by \cite[Proposition 4.2]{CM-negative} when $n=2$ there is only one twisting number (of a regular fibre of the fibration determined by the graph) and $(p_1,p_2,1)$ gives rise to $-\overline q$. We can then proceed by induction on the number of $p_i$'s equal to 1.

 From \cite[Proposition 4.2 and Theorem 4.3]{CM-negative} we know that if $p_i=1$ then the $i$-th leg is not contained in the $S^3$-subgraph of $\Gamma^*$, and as we remarked in the proof of Proposition \ref{prop:bigtw} their proofs also hold in the indefinite case. Removing the $i$-th leg yields an indefinite graph, except when $-Y$ is as in Lemma \ref{lemma:twisting}. To see this we consider the $n$-tuple $(p_1,p_2,1,...,1)$ which gives rise to $-\overline q$, and the same is true when we remove one 1. We know that if the new graph is negative-definite then there is only one twisting number \cite[Proposition 4.2]{CM-negative}, thus we conclude that $q=\overline q$ and this is impossible. If the new graph is singular and not as in Figure \ref{Torus} then there is still only one twisting number by Proposition \ref{prop:tw}.

 Say $G$ is obtained by removing the $n$-th leg where $p_n=1$, by induction there is a realised characteristic vector $V'$ such that $2X':=Q_G^{-1}(V'+V'_\text{can})\in2\Z^{|G|}$, one has $(V'-V'_\text{can})^TX=0$ and $q=1-x'_1=1-e_1^TX'$. We show that if we extend $X'$ to $X=(X',0,...,0)$ on the last leg then $V=2Q_*X-V_\text{can}$ satisfies all the properties we want. The integrality condition is automatic as $X'$ is integral, we have $x_1=e_1^TX=e_1^TX'=x_1'$, and $(V-V_\text{can})^TX=(V'-V'_\text{can})^TX'=0$. We need to show that $V$ is realised: its coordinates on the $n$-th leg are \[(2x'_1-m^n_1-2,-m^n_2-2,...,-m^n_{k_n}-2)\] so we just have to check the first one. Since the best upper approximation for $r_n$ is $\frac{1}{q}$, we have that $-m_1^n\geq q+1$; therefore, we conclude that \[-m_1^n-2\overline q\geq2x'_1-m^n_1-2=2(1-q)-m_1^n-2\geq m^n_1+2\] which is precisely the range of the possible values for $V$ to be realised, see Theorem \ref{teo:realised}. 
\end{proof}

This completes the proof that the full path algorithm determines the twisting numbers.

\begin{proof}[Proof of Theorem \ref{teo:twisting_numbers}]
 It follows from Theorem \ref{teo:twisting} and Proposition \ref{prop:bigtw}. When $\s_\xi$ is spin then \[-1-\height([V_\text{can}]+[-V_\text{can}])=-1+e_1^TQ_*^{-1}\left(\frac{V_\text{can}-(-V_\text{can})}{2}\right)=-1+e_1^TQ_*^{-1}V_\text{can}\] is a twisting number; the fact that it is the lowest one is a consequence of the total ordering of $\mathcal F$ on the full paths of $\text{Char}(\Gamma^*,\s_\xi)$ ending correctly, and Lemma \ref{lemma:can} which tells us that $V_\text{can}$ and $-V_\text{can}$ reach the extremal values of $\mathcal F$ among these vectors.
\end{proof}

\subsection{The full path determines the contact structures}
We now show that the total number of negative-twisting contact structures, up to isotopy, which we determined in Section \ref{section:six}, is the same as the number of unordered pairs of realised characteristic vectors sharing the same $\Spin^c$-structure and Maslov grading. 

The assumption of $-Y$ being indefinite is not needed for the result we are proving here, but we keep the notation to avoid confusion, as we know the result holds in the negative-definite case.

\begin{teo}
 \label{teo:contact}
 Suppose that $-Y=M(e_0;r_1,...,r_n)$ is a Seifert fibred space. Let $V_1$ and $V_2$ be two realised characteristic vectors and set $2X=Q_*^{-1}(V_1+V_2)$ and $2U=V_1-V_2$. Then the number of ordered pairs $(V_1,V_2)$ for which $X$ is integral, $U^TX=0$ and $-q=-1+e_1^TX$ is a twisting number of $-Y$ is equal to the number of negative-twisting structures on $-Y$, up to isotopy.
\end{teo}
By Proposition \ref{prop:spinc} we know that $X$ being integral is equivalent to $V_1$ and $-V_2$ inducing the same $\Spin^c$-structure on $-Y$, while by Equation \eqref{eq:Maslov2} that $M(V_1)=M(V_2)$; moreover, from the definition of height we have that $q\geq1$. We recall that the number of negative-twisting structures coincides with the upper bound from convex surface theory, which we explicitly write in Subsection \ref{subsection:background}; in particular, these structures are precisely the ones constructed in Subsections \ref{subsection:A} and \ref{subsection:B}, together with the ones induced by Stein structures on the blown down standard graph.

\begin{lemma}
 \label{lemma:W}
  Let $\frac{p}{q}$ be the best upper approximation for $r\in(0,1)\cap\Q$ for some $q>1$; expand $-\frac{q}{p}$ into the negative continued fraction $[n_1,...,n_{h}]$ and call $B$ the negative of its chain matrix. 
  Then there exists a positive integral vector $W$ such that $BW=qe_1$; moreover, one has $w_1=p$ and $w_h=1$.    
\end{lemma}
\begin{proof}
 If we define $w_0:=q$ and  $w_1:=p$ and then recursively \[w_{j+1} = (-n_j)\,w_j - w_{j-1}\:,\] the continued fraction relation ensures $w_{h} = 1$, while the row equations become:
 \begin{itemize}[leftmargin=0.8cm]
     \item for the first row: $(-n_1)w_1 - w_2 = q$;
     \item for the interior rows: $-w_{j-1} + (-n_j)w_j - w_{j+1} = 0$;
     \item for the last row: $-w_{h-1} + (-n_{h})w_{h} = 0$.
 \end{itemize}
 This is exactly $BW= q e_1$.   
\end{proof}

\begin{proof}[Proof of Theorem \ref{teo:contact}]
 Fix a twisting number $-q$; for each leg $i$ let $\frac{p_i}{q}$ be the best upper approximation for $r_i$, and let $-\frac{q}{p_i}=[n^i_1,...,n_{h_i}^i]$ and $-\frac{1}{r_i}=[m^i_1,...,m_{k_i}^i]$ be the corresponding negative continued fractions. 

 For an ordered pair $(V_1,V_2)$ contributing to the twisting number $-q$ we have $V_1=Q_*X+U$ and $V_2=Q_*X-U$. If we define $Z^{(i)}:=\mathbf 1-X^{(i)}$ as in the previous subsection, then on the $i$-th leg we have that Equation \eqref{eq:leg} holds. Moreover, the following properties are already established: $z_j^{i}\geq1,$ $0\leq u_j^{i}\leq \overline u_j^{i}$, the vertex-dependent bound from the realised definition, and $u_j^{i}\bigl(z_j^{i}-1\bigr)=0$ for every $i=1,...,n$ and $j=1,...,k_i$.

 For any given $q$ the vectors $Z^{(i)}$ and $U^{(i)}$ are determined by the best upper approximation up to the dependence given by Equation \eqref{eq:leg} from the Leg lemma. The count of admissible ordered pairs $(V_1,V_2)$ is then equal to the product of the numbers of the possible difference vectors $U^{(i)}$ on each leg.

 Let us compute the latter number for an arbitrary leg. Looking at the dependence condition $-Q^{(i)}Z^{(i)}-U^{(i)}=qe_1+e_{k_i}$, we use the auxiliary Lemma \ref{lemma:W} to express $qe_1$ as $-R^{(i)}W^{(i)}$, where $R^{(i)}$ is the chain matrix corresponding to $\frac{q}{p_i}$ and $W^{(i)}$ a positive integral vector with $w_1^i=p_i$ and $w_h^i=1$.

 We recall that $-\frac{q}{p_i}=[n^i_1,...,n_{h_i}^i]$ has $n^i_j=m_j^i$ for $j=1,...,h_i-1$ and either $h_i=k_i+1$ or $h_i\leq k_i$ with $m_{h_i}^i<n_{h_i}^i$. We check for possible values of $u^i_j$ row-by-row in $-Q^{(i)}Z^{(i)}-U^{(i)}=-R^{(i)}W^{(i)}+e_{k_i}$. Since $z_1^i=w_1^i=p_i$, we have $u_1^i=0$, and inductively $z_j^i=w_j^i$ for $j\leq h_i$ and $u_j^i=0$ for $j<h_i$. At the $h_i$-th row we finally have $z_{h_i}^i=w_{h_i}^i=1$ and for $u_{h_i}^i$ there are $m_j^i-n_j^i$ choices, while in all later tail coordinates there are $1-m_j^i$ of them. 

 Therefore, when $h_i=k_i+1$ the vectors $Z^{(i)}, U^{(i)}$ are uniquely determined, while for $h_i\leq k_i$ we set \[T(i,q)=(n_{h_i}^i-m_{h_i}^i)\cdot\prod_{j=h_i+1}^{k_i}|m_j^i+1|\] and this agrees exactly with the upper bound in Subsection \ref{subsection:background}. Because the choices on different legs are independent once $q$ is fixed, the total number of ordered pairs with twisting number $-q<-1$ is $T(1,q)\cdots T(n,q)$.

 When $q=1$, the interval condition is vacuous (there is no $h<1$), so the above description degenerates. In this case every vertex contributes its full number of values: the centre contributes for $\lvert e_0+1\rvert$, while each leg vertex $m^i_j$ contributes for $\lvert m^i_j+1\rvert$. The total is \[\lvert e_0+1\rvert\prod_{i=1}^{n}\prod_{j=1}^{k_i}\lvert m^i_j+1\rvert\:.\]
 Different twisting numbers give disjoint classes of vectors $X$. Thus the total number of ordered pairs is the sum of the fixed $q$-contributions, which is exactly our statement.
\end{proof}

We can find a correspondence between the sets of ordered and unordered pairs of realised characteristic vectors; namely, we have that if $V_1$ and $V_2'$ are as above then the vector $V_2$ is also realised, for symmetry reasons, where $V_2'$ is the initial vector of the full path $[-V_2]$. Therefore, the ordered pair $(V_1,V_2')$ is identified with $(V_1,V_2)$ (as an unordered pair), and then $V_1$ and $V_2'$ induce conjugated $\Spin^c$-structures on $-Y$ if and only if $V_1$ and $V_2$ induce the same one. We also know that by definition, see Section \ref{section:three}, one has \[-e_1^TQ_*^{-1}\left(\frac{V_1+V_2}{2}\right)=\left\{\begin{aligned}&\height([V_1]+[-V_2])\hspace{1.22cm}\text{ when }V_1\neq V_2'\\ 
&\height[V_1]=\height[V_2]\hspace{0.5cm}\text{ when }V_1=V_2'\:.\end{aligned}\right.\] 
The number of unordered pairs $(V_1,V_2)$ of realised characteristic vectors with same Maslov grading and $\Spin^c$-structure, which appears in the statement of Theorem \ref{teo:classification}, then also coincides with the number of negative-twisting structures.

\begin{proof}[Proof of Theorem \ref{teo:classification}]
 The correspondence follows from the discussion above and Theorem \ref{teo:contact}, while fillability follows from Corollary \ref{cor:fillable}.
\end{proof}

As we mentioned before, when $q=-\overline{\text{tw}}(-Y)$ then the upper bound coincides with the number of realised characteristic vectors. In this setting the corresponding ordered pairs are of the form $(V,V')$, while the unordered ones are of the form $(V,V)$.

\subsection{The contact invariant}
We start this subsection by showing that, when $-Y$ is type A, there is at most one contact structure $\xi$ such that $\text{tw}(-Y,\xi)<\overline{\text{tw}}(-Y)$ in each $\s\in\Spin^c(-Y)$; meanwhile, when $-Y$ is type B, there is at most one pyramid in each $\s\in\Spin^c(-Y)$. Later on, we determine exactly what are the coordinates of the contact invariant $c^+(\xi)\in HF_\text{red}(Y,\s)$ with respect to the canonical basis $\mathcal B=\{T_{[V_1]},...,T_{[V_t]}\}$ given by the full path algorithm. We recall that $\mathcal B$ can be identified with a basis of $\widehat{HF}^\text{od}(Y,\s)$, see Section \ref{section:two}, and thus we can also identify $c^+(\xi)$ with $\widehat c(\xi)$. 

We also recall the terminology that we introduced in Section \ref{section:six}. An indefinite Seifert fibred space $-Y=M(e_0;\frac{b_1}{a_1},...,\frac{b_n}{a_n})$ which is not an $L$-space can be either of type B, when its standard graph has a subgraph $G$ isomorphic to one of a torus bundle in Figure \ref{Torus}, or of type A otherwise. We say that a manifold of type B is short when there is at most one vertex not in $G$ on each leg; otherwise, we say it is long. We say that a manifold of type A is without tails when $|\underline{\text{tw}}(-Y)|>a_i$ for $i=1,...,n$; otherwise, if it has two distinct negative twisting numbers then we say it has tails. We proved in Lemma \ref{lemma:tails} and Proposition \ref{prop:classificationA} that when $-Y$ has tails any negative-twisting structure is obtained from a manifold $M$ of type A without tails, equipped with the contact structure tangent to the fibres, by doing Legendrian surgery on some new vertices attached to the standard graph of $M$. Similarly, we proved in Proposition \ref{prop:classificationB} that negative-twisting structures on a long manifold of type B are obtained from a short one through the same construction.

By Lemma \ref{lemma:realised} if $-Y$ is not of type A without tails then all the coordinates of the contact invariant, with respect to $\mathcal B$, correspond to realised characteristic vectors. When $-Y$ is without tails we prove that this is still true in the following lemma. For this reason, in the remaining of the section we only consider the subset of $\mathcal B$ spanned by realised characteristic vectors.

\begin{lemma}
 \label{lemma:final}
 Suppose that $-Y$ is an indefinite Seifert fibred space. If $V\in\emph{Char}(\Gamma^*,\s_\emph{can})$ is initial, its full path ends correctly, and it is such that $M(V)=M(V_\emph{can})$, then $V$ is realised.
\end{lemma}
\begin{proof}
 We do the proof for a $V\in\text{Char}(\Gamma^*,\overline\s_\text{can})$, the given statement then follows by conjugation. Using the same terminology of the previous subsections, we write $X=\frac{1}{2}Q_*^{-1}(V+V_\text{can})\in\Z^{|\Gamma^*|}$ and $U=\frac{1}{2}(V-V_\text{can})\in\Z^{|\Gamma^*|}$; then $0=M(V)-M(V_\text{can})=U^TX$. Clearly, one has $u_j\geq0$ for $j=1,...,|\Gamma^*|$, while by the same argument for Equation \eqref{eq:leg} and Lemma \ref{lemma:orthogonal} we have $x_j\leq0$ for each $j$; hence, we need $u_jx_j=0$. Since $\Gamma^*$ is indefinite, we have that either $[V]=[-V_\text{can}]$, and then $V$ is realised, or $x_1=-\height([V]+[-V_\text{can}])<0$, implying that $u_1=0$.

 Whenever $u_j=0$ the coordinate $v_j$ of $V$ agrees with the corresponding one of $V_\text{can}$, which means $v_j=m(j)+2$. Assume that $u_k\neq0$, then we have that $x_k=0$ and the vertex $S_k$ is on a leg of $\Gamma^*$; moreover, we can also assume that such a vertex cannot be in, or connected, to the $S^3$-subgraph $(\Gamma^*)'$ because in this case $v_k$ is necessarily in the realised range in Theorem \ref{teo:realised}: this follows from the fact that the $S^3$-subgraph has a unique initial vector ending correctly.  

 For the remaining case, we argue that $U=-V_\text{can}+QX$ and then $u_k=-m(k)-2+e_k^TQX$; therefore, we have \[v_k=m(k)+2+2u_k=m(k)+2-2(m(k)+2)+2e_k^TQX=-m(k)-2+2e_k^TQX\:.\] Now, since $x_k=0$ and $x_j\leq0$ for $j=1,...,|\Gamma^*|$, we have that \[e_k^TQX=x_{k-1}+m(k)\cdot x_k+x_{k+1}=x_{k-1}+x_{k+1}\leq0\] where we are assuming that the ordering of the vertices is consecutive along each leg, and that $S_{k-1}$ could be taken as the central vertex. Then we can conclude that $v_k\leq-m(k)-2$ for any such vertex $S_k$, which means precisely that $V$ is realised.
\end{proof}

We specified in Subsection \ref{subsection:B} that structures may form pyramids. Consider \[-q_0=\overline{\text{tw}}(-Y)>-q_1>\cdots>\underline{\text{tw}}(-Y)\] the decreasing sequence of all the negative twisting numbers on $-Y$; then a pyramid of size $k$ is the family of all the negative-twisting structures $\xi_{ij}$ for $1\leq i\leq j\leq k$, in a fixed homotopy type, such that $\text{tw}(-Y,\xi_{ij})=-q_{j-i}$ for every $i,j$. In the terminology of the previous subsection, a pyramid of size $k$ is the family of all the realised characteristic vectors $V_1,...,V_k\in\text{Char}(-Y,\s)$ with fixed Maslov grading such that each unordered pair $(V_i,V_j)$ corresponds to a negative-twisting structure, in the sense of Theorem \ref{teo:classification}, and $\height([V_i]+[V_{i+1}])$ is the same for $i=1,...,k-1$. 

\begin{prop}
 \label{prop:map}
 Suppose that $-Y$ is an indefinite Seifert fibred space, if $-Y$ is of type A then every $\s\in\Spin^c(-Y)$ admits at most one structure with twisting number smaller than $\overline q=\overline{\emph{tw}}(-Y)$. Furthermore, if $-Y$ is of type $B$ then every $\s\in\Spin^c(-Y)$ supports exactly a single pyramid of size $k(\s)$.
\end{prop}
\begin{proof}
 If $-Y$ is short of type B then the claim has been already proved in Subsection \ref{subsection:B}. If $-Y$ is of type A without tails then, by the results of Massot \cite{Massot} and the discussion in Subsection \ref{subsection:A}, we know there is a unique $\xi_0$ such that \[\text{tw}(-Y,\xi_0)=\underline q=\underline{\text{tw}}(-Y)<\overline q\]
 and $\xi_0$ is tangent to the fibres and supports a spin structure $\s_0$. 
 From Theorem \ref{teo:twisting} one has $\s_0=\s_\text{can}$ and then $c^+(\xi)=T_{[V_\text{can}]}+T_{[-V_\text{can}]}$ by Lemma \ref{lemma:final}. 

 Suppose that $-Y$ is not in one of these two families, we prove the claim by induction on the number of vertices we need to add to the standard graph of a manifold (either short or without tails) to obtain $-Y$; note that if $-Y$ is of type A and has only one negative twisting number then there is nothing to prove. 
 
 Consider the map $F^+_{\overline W,J}:HF^+(-Y,\s)\rightarrow HF^+(M',\s')$ induced by $(-1)$-contact surgery, which as we know maps $T_{[V_i]}$ to $T_{[V_i']}$ where $\mathcal C=\{V_1,...,V_{k(\s)}\}$ are the realised characteristic vectors in $\text{Char}(\Gamma^*,\s)$ while $V'_i$ is obtained by removing the coordinate of the new vertex from $V_i$. Notice that $F^+_{\overline W,J}$ is injective on $\mathcal C$ because then $X_{\Gamma^*}$ would admit a Stein structure $I$, extending $J$, such that $F^+_{\overline X_{\Gamma^*},I}(T_{[V_i]})=1$ for $i=1,2$, contradicting the usual result of Plamenevskaya \cite{OlgaP}.

 Say that $V_1,...,V_h$ such that $\mathcal F(V_1)<\cdots<\mathcal F(V_h)$ would form a maximal pyramid, and assume that there is a $V_{h+1}$ such that either $\mathcal F(V_{h+1})>\mathcal F(V_h)$ or $\mathcal F(V_{h+1})<\mathcal F(V_1)$. We consider first the case that $\s$ and $\s'$ supports negative-twisting tight structures with the same twisting numbers; then by induction $\mathcal F(V_1')<\mathcal F(V_{h+1}')<\mathcal F(V_h)$ but we now use the fact that $F^+_{\overline W,J}$ has well-defined Maslov and Alexander grading shift, see \cite{OSz-contact,OSz-genus,G-fillability,GvHM}, and we immediately have a contradiction. In the case $\s'$ supports a structure with an additional twisting number $q$, then Theorem \ref{teo:contact} and the same argument of the degree shifts of $F^+_{\overline W,J}$ imply that $q$ is also the twisting number of a contact structure which induces $\s$, and this is again not possible.

 Note that when $-Y$ is of type $B$ the argument above and Lemma \ref{lemma:j} also imply that $V_{h+1}$ should necessarily extend the pyramid whose base is made by $V_1,...,V_h$, contradicting the maximality that we assumed at the beginning. Therefore, we need to have $h=k(\s)$ and all the realised characteristic vectors form a pyramid of size $k(\s)$.  
\end{proof}

We proved that the Stein structures $\{J_i\}$ on the cobordism obtained by attaching new vertices on a standard graph, which extend a fixed $\Spin^c$-structure $\s'$, are in one-to-one correspondence with the choices of coordinates for the resulting characteristic vector to be realised. In addition, the $\Spin^c$-structures $J_i\lvert_{-Y}$, where the manifold $-Y$ is obtained this way, which support tight structures with at least two distinct negative twisting numbers are all distinct.  

\begin{teo}
 \label{teo:final}
 Suppose that $(-Y,\s)$ is an indefinite Seifert fibred space equipped with a $\Spin^c$-structure, and let $\xi$ be a negative-twisting structure on $-Y$ such that $\emph{tw}(-Y,\xi)<\overline{\emph{tw}}(-Y)$ and $\s_\xi=\s$. Denote by $V_1,...,V_{k(\s)}$ the realised characteristic vectors in $\emph{Char}(\Gamma^*,\s)$, in increasing order according to $\mathcal F$; then $0\neq c^+(\xi)\in HF_\emph{red}(Y)$ and we have that:
 \begin{enumerate}[label=\Alph*),leftmargin=0.8cm]
     \item if $-Y$ is of type A then $\xi$ is unique, up to isotopy, and $c^+(\xi)=T_{[V_1]}+T_{[V_{k(\s)}]}$;
     \item if $-Y$ is of type B then $\xi=\xi_{ij}$ for some $1\leq i<j\leq k(\s)$, and $c^+(\xi)=\alpha_iT_{[V_i]}+\alpha_{i+1}T_{V_{i+1}}+\cdots+\alpha_{j}T_{[V_j]}$ where $\alpha_l=\binom{j-i}{l-i}\emph{ mod }2$ for $l=i,i+1,...,j$.
 \end{enumerate}
 Furthermore, the twisting number is determined by \[\emph{tw}(-Y,\xi)=\left\{\begin{aligned}&-1-\height([V_1]+[V_{k(s)}])\hspace{0.5cm}\text{ in Case A}  \\ &-1-\height([V_i]+[V_j])\hspace{0.96cm}\text{ in Case B}\:.\end{aligned}\right.\]
\end{teo}
\begin{proof}[Proof of Theorem \ref{teo:final} and Proposition \ref{prop:classification}]
 We already saw in Corollary \ref{cor:fillable} that for indefinite Seifert fibred spaces $c^+$ is non-vanishing in $HF_\text{red}$ precisely for negative-twisting structures. The claim about the height function follows immediately by Theorem \ref{teo:contact}. The coordinates of the contact invariant are determined using the same cobordism map we used in the proof of Proposition \ref{prop:map}, and the fact that such a map is injective on the subgroup spanned by $V_1,...,V_{k(\s)}$, provided that we can identify them for manifolds that are either short or without tails. 
 
 For the type A case we did this in the proof of Proposition \ref{prop:map}, while for the type B one by Corollary \ref{cor:invariant} we have that $c^+(\xi_{ij})=c^+(\xi_{ij-1})+c^+(\xi_{i+1j})$ leading to the binomial coefficient formula in the statement. Hence, for each $\xi_{ij}$ the Alexander filtration of any coordinate $T_{[V_l]}$ of $c^+(\xi_{ij})$ is in the interval $[\mathcal F(V_i),\mathcal F(V_j)]$, with extremal points included, and the unordered pair $(V_i,V_j)$ is unique, in accordance to the correspondence in Theorem \ref{teo:contact}. 
\end{proof}

\section{Applications}
\label{section:eight}
\subsection{Brieskorn spheres} 
In this subsection we collect some results about Brieskorn spheres. We start by showing that every oppositely oriented Brieskorn sphere $-Y$, different from $-\Sigma(2,3,6k\pm1)$ for $k\geq1$, is of type A and that $-Y$ has at most one structure with twisting number smaller than $\overline{\text{tw}}(-Y)$. We need the following lemma.
\begin{lemma}
 \label{lemma:Brieskorn}
 If $-Y=-\Sigma(a_1,...,a_n)$ is an oppositely oriented Brieskorn sphere, and $\Gamma^*$ is its standard graph, then any Seifert fibred space $M$ whose standard graph is obtained by removing a vertex from $\Gamma^*$ is such that $e(M)\leq0$.   
\end{lemma}
\begin{proof}
 Since $-Y$ is an integral homology sphere, the matrix $Q_*^{-1}$ has integral entries. In addition, we have that the diagonal entries corresponding to the terminal vertex of any leg are non-negative; otherwise, if $q_{ii}:=e_i^TQ_*^{-1}e_i<0$ then adding a new vertex to the $i$-th leg with framing $q_{ii}$ would produce a graph $\widehat G$ such that $e(-Y)<e(M_{\widehat G})=0$, which is not possible because $\Gamma^*$ is indefinite.

 We now know that $q_{ii}\geq0$ for each vertex $S_i$ as above. If $q_{ii}=0$ then removing $S_i$ yields $e(M)=0$. Suppose that $q_{ii}>0$; we have that the graph $\widetilde G$, obtained by lowering the framing of $S_i$ by one and adding vertices with framings given by the negative continued fraction of $-\frac{q_{ii}}{q_{ii}-1}$, satisfies $e(M)<e(M_{\widetilde G})=0$.  
\end{proof}
\begin{proof}[Proof of Corollary \ref{cor:Brieskorn}]
 Suppose that $-Y$ is of type B, and denote by $G_0\subset\Gamma^*$ a subgraph isomorphic to one of a torus bundle in Figure \ref{Torus}. Since we cannot have $G_0=\Gamma^*$, we can choose a vertex at the end of a leg which is outside of $G_0$ and remove it; hence, by Lemma \ref{lemma:Brieskorn} we obtain a new subgraph $G$, which contains $G_0$, such that $0=e(M_{G_0})\leq e(M_G)\leq0$. This immediately implies that $G=G_0$, and then $\Gamma^*$ is obtained by adding a single vertex to the standard graph of a torus bundle. Checking by hand all the possible cases, we determine that the Seifert fibred spaces of this form which are Brieskorn spheres are $-\Sigma(2,3,6k\pm1)$ for $k\geq1$.

 If $-Y$ is of type A then, by Theorem \ref{teo:A} and Remark \ref{remark:A}, at most two negative twisting numbers are realised by tight structures. Suppose that there is a $\xi$ such that $\text{tw}(-Y,\xi)=\underline{\text{tw}}(-Y)$, then it is unique by Proposition \ref{prop:map} because $Y$ is an integral homology sphere.  
\end{proof}

We now show that there are only three Brieskorn spheres with a unique tight contact structure; namely, with the canonical orientation we only have $\Sigma(2,3,5)$ (the fact that it carries a unique tight structure is well-known and follows from the techniques in \cite{EH}), while with opposite orientation we have $-\Sigma(2,3,11)$ and $-\Sigma(2,3,7)$ (they both have only one tight structure by \cite{G-,Tosun}). We show that there are no other ones. 

\begin{proof}[Proof of Theorem \ref{teo:Brieskorn}]
 We first observe that a Brieskorn sphere with a unique fillable structure has blown down standard graph $X_G$ with either $-2$-framed unknots, or $0$-framed positive trefoils by Lemma \ref{lemma:Brieskorn}. We assume that $Y=\Sigma(a_1,...,a_n)$ is a canonically oriented Brieskorn sphere. 
 
 If $G=\Gamma$ is negative-definite then the only way we can have $e(Y)<0$ is when $e_0=-2$, there are three legs, and all the vertices of $\Gamma$ have framing $-2$. 
 Hence, we have $-Y=M(-1;\frac{1}{a},\frac{1}{b},\frac{1}{c})$ where $c>b>a\geq2$ are integers; we need that $a,b$ and $c$ are pairwise coprime and that $e(-Y)=\frac{1}{a}+\frac{1}{b}+\frac{1}{c}-1>0$: this can only happen when $a=2,$ $b=3,$ and $c=5$. 

 If $G=\Gamma^*$ is indefinite then the analysis is more involved. When $-Y$ is of type B then from Corollary \ref{cor:Brieskorn} and \cite{G-,Tosun} we only have the two manifolds in the statement, so we consider a $-Y$ of type A. Assume first that $e_0\leq-2$, then $n=3$ because otherwise we would have $M(-2;\frac{1}{2},\frac{1}{2},\frac{1}{2},\frac{1}{2})$ as subgraph; hence, we can write $Y=M(-1;\frac{1}{a},\frac{1}{b},\frac{1}{c})$ and $\frac{1}{a}+\frac{1}{b}+\frac{1}{c}-1=-\frac{1}{abc}$ with $a,b,c$ as above: this is only possible when $a=2,$ $b=3$ and $c=7$ which leads to a type B manifold. 

 Assume now that $e_0=-1$, we have that $X_{\Gamma^*}$ has only unknots with framing $-2$ otherwise either $M(-1;\frac{1}{2},\frac{1}{3},\frac{1}{6})$ would be a subgraph of $\Gamma^*$, or $Y=S_{d_1d_2-d_1-d_2-1}^3(T_{d_2,d_1})$ by Lemma \ref{lemma:Brieskorn} but this is not a Brieskorn sphere. We consider the possible $S^3$-subgraph $(\Gamma^*)'$: from what we said before this cannot be neither empty nor the fibration of any $T_{d_2,d_1}$ with $d_1,d_2>1$. If $(\Gamma^*)'$ coincides with the centre then $M(-1;\frac{1}{3},\frac{1}{3},\frac{1}{3})$ would be a subgraph; in the same way, if $(\Gamma^*)'$ coincides with $M(-1;\frac{1}{2})$ then $M(-1;\frac{1}{2},\frac{1}{4},\frac{1}{4})$ would be a subgraph. This is not possible as both correspond to torus bundles.

 The last case to consider is when $(\Gamma^*)'$ is $M(-1;\frac{n}{n+1})$ when $n\geq2$, and here we have $\widetilde G:=M(-1;\frac{n}{n+1},\frac{1}{n+3},\frac{1}{n+3})$ as a subgraph of $\Gamma^*$. If $\Gamma^*\neq\widetilde G$ then applying Lemma \ref{lemma:Brieskorn} we obtain a subgraph $G$ such that $0<e(M_{\widetilde G})\leq e(M_G)\leq0$, which is not possible; therefore, we conclude that $-Y=M_{\widetilde G}\simeq S^3_{-2,-2}(T_{2,2n+2})$ and this is never a Brieskorn sphere.
\end{proof}

Note that if an oppositely oriented Brieskorn sphere $-Y$ of type A had no tight structure with twisting number $\underline{\text{tw}}(-Y)$ (the structure tangent to the fibres as in \cite{Massot}) then from the proof of Corollary \ref{cor:Brieskorn} it should have a unique realised characteristic vector (that is $V_\text{can}$); hence, this implies that $-Y$ is as in Theorem \ref{teo:Brieskorn} which, conversely, does not mention any manifold of type A. This means that any Brieskorn sphere $-Y$ different from $-\Sigma(2,3,5)$ always carries a (unique) contact structure tangent to the fibres.

\subsection{Surgeries on torus knots}
In this subsection we prove our results about $S^3_r(T_{d_2,\pm d_1})$ where $r\in\Q$ and $1<d_2<d_1$ coprime. Namely, we first show that our classification results hold for these manifolds, and then we compute all the possible values of the negative twisting numbers for a tight structure on $S^3_r(T_{d_2,\pm d_1})$.

\begin{proof}[Proof of Corollary \ref{cor:surgeries}]
 It is well-known that any surgery on a torus link is a Seifert fibred space, except when the framing on a component is the exceptional slope (that is $\pm d_2d_1$); therefore, it can be presented by a star-shaped graph. When this graph is negative-definite then the classification is given by \cite[Corollary 5.3]{CM-negative}; otherwise, by Proposition \ref{prop:main} and Theorem \ref{teo:classification}.

 If the exceptional slope appears then the graph would have a leg which consists of a single vertex with framing zero; hence, the central vertex can be cancelled, and then the manifold is the connected sum of the lens spaces (possibly $S^1\times S^2$'s) represented by the remaining legs. It is a result of Colin \cite{Colin} that a tight structure decomposes along connected sums, and tight structures on lens spaces were already classified explicitly, see \cite{Honda1}.
\end{proof}

\begin{proof}[Proof of Proposition \ref{prop:Tsurgeries}]
 Surgeries on a positive torus knot when $d_1d_2-d_1-d_2\leq r<d_1d_2$ are $L$-spaces, and have normalised Seifert coefficients with $e_0\geq -1$; hence, they admit only zero-twisting structures. Meanwhile, all the others admit negative-twisting structures by \cite{LM-L,Wu}. From Section \ref{section:five}, once we have negative-twisting structures, we have those with the highest twisting number, which is computed, as for Proposition \ref{prop:bigtw}, to be $-d_1-d_2$ (when smaller than $-1$). 

 For $(d_2,d_1)\neq(2,3)$ with these restrictions, the manifold $S^3_r(T_{d_2,d_1})$ is Seifert fibred of type A; hence, there is at most one additional twisting number. It is of the form $q:=d_1+d_2+sd_1d_2$ because it is denominator of the best upper approximations for $r_1$ and $r_2$ simultaneously, and we get that the best upper approximation of $r_3$ is of the form $\frac{p_3}{q}=\frac{s+1}{d_1+d_2+sd_1d_2}$ as a reduced fraction by requiring that the three approximations sum up to $1+\frac{1}{q}$ (by Ghiggini-Massot's algorithm). Now for the possible values of the surgery coefficient $r$, we determine $r_3$ so that it has $\frac{p_3}{q}$ as the best upper approximation. 

 Surgeries on $T_{2,3}$ (when $r\in(0,1)$) are of type B: they are negative-definite for $r<0$, give the torus bundle $M(-1;\frac{1}{2},\frac{1}{3},\frac{1}{6})$ with infinitely many twisting numbers of the form $-5-6s$ for $r=0$, while for surgeries in $(0,1)$ we need that $-6+r<-6+\frac{1}{s}$ for $[-6,-s]$ to be a twisting number.

 Surgeries on negative torus knots when $r<-d_1d_2$ are also (indefinite) $L$-spaces and have no negative-twisting structures. All other surgeries have $e_0\leq-2$ and thus $-1$ as their highest twisting number. Suppose that $-d_1d_2+1\leq r\leq d_1d_2-1$; if we write $S^3_r(T_{d_2,d_1})$ as $M(-1;r_1,r_2,\frac{1}{d_1d_2-r})$, then $S^3_r(T_{d_2,-d_1})$ is $M(-2;1-r_1,1-r_2, 1-\frac{1}{d_1d_2+r})$. Now the possible denominators of the best upper approximations of $1-r_1$ and $1-r_2$ simultaneously take the form $q':=-d_1-d_2+d_1d_2l$, and the best upper approximation of the third coefficient is forced to be $\frac{p_3'}{q'}=1-\frac{l-1}{-d_1-d_2+d_1d_2l}$ and $l>2$ whenever $(d_2,d_1)\neq(2,3)$ by the sum-equal-to $2+\frac{1}{q'}$ condition. We conclude the proof for the given surgery coefficients by noticing the equivalence \[\frac{1}{d_1d_2-r}<\frac{s+1}{d_1+d_2+sd_1d_2}\hspace{0.5cm}\text{ if and only if }\hspace{0.5cm}1-\frac{1}{d_1d_2+r}<1-\frac{s+1}{-d_1-d_2+d_1d_2(s+2)}\ ,\] and the fact that the M\"obius transform \[\begin{pmatrix} 2d_1d_2-1&-1\\2d_1d_2&-1 \end{pmatrix}\ \in\ \text{SL}_2(\mathbb Z)\ , \] 
 which takes $r_3=\frac{1}{d_1d_2-r}$ to $r_3'=1-\frac{1}{d_1d_2+r}$, maps also $\frac{s+1}{d_1+d_2+sd_1d_2}$ to $1-\frac{s+1}{-d_1-d_2+d_1d_2(s+2)}$ and preserves the continued fraction relations that characterise the best upper approximations (see Subsection \ref{subsection:background}). In the case of $(d_2,d_1)=(2,3)$, we can check by hand that also the twisting number $-5$ on $S^3_r(T_{2,3})$ gives rise to the twisting number $-7$ on $S^3_r(T_{2,-3})$.

 When $-d_1d_2<r<-d_1d_2+1$ the manifold $M$ is negative-definite and then there is only $-1$ as twisting number. Assume $r>d_1d_2-1$ and that $q>1$ is such that $-q$ is a twisting number. We know from Ghiggini-Massot's algorithm that $p_3\leq q-1$ and then $1-\frac{1}{d_1d_2+r}<\frac{p_3}{q}<1-\frac{1}{q}$, which implies $q\geq d_1d_2+r+1>2d_1d_2$; moreover, by the proof of Lemma \ref{lemma:A} we have $q<\frac{1}{e(M)}=\frac{d_1d_2(d_1d_2+r)}{r}\leq2d_1d_2+1$ leaving no possible value for $q$. 
\end{proof}

We conclude the paper by showing how the computation in Proposition \ref{prop:Tsurgeries}, together with the explicit expression of the upper bound in Subsection \ref{subsection:background}, can be used to determine the number of negative-twisting structures on every surgery on a torus knot. We do the case of $M=S_r^3(T_{d_2,d_1})$, and we explain how to recover the same result for $S_r^3(T_{d_2,-d_1})$.

We first note that (when $\overline{\text{tw}}<-1$) the standard graph of $M$ has three legs and $e_0=-1$; moreover, the centre and the first two legs form the $S^3$-subgraph which coincides with the fibration of $T_{d_2,d_1}$. The set of realised characteristic vectors is then completely determined by the third leg, because from Theorem \ref{teo:realised} we know that the coordinates of such vectors are fixed on the $S^3$-subgraph. The framing of the third leg is $-d_1d_2+r$ whenever $r<d_1d_2-d_1-d_2$; note that when $d_1d_2-d_1-d_2<r<d_1d_2$ the manifold $M$ is an $L$-space with $e_0=-1$ which has no negative-twisting structures, see Remark \ref{remark:L}. We write $-d_1d_2+r$ as the negative continued fraction $[m_1,...,m_k]$, whose coefficients correspond to the framings of the vertices of the third leg.

Since we know from Proposition \ref{prop:bigtw} and \cite[Theorem 1.1]{CM-negative} that the number of contact structures with twisting number $\overline{\text{tw}}$ is equal to the number of realised characteristic vectors, we conclude that such a quantity is \[|m_1+d_1+d_2|\cdot|m_2+1|\cdots|m_k+1|\:,\] where the first term accounts for the restricted values of the coordinate of the vertex attached to the centre, as in Theorem \ref{teo:realised}. When $\overline{\text{tw}}=-1$ the claim follows from Equation \eqref{eq:upper2}. 

For the other twisting numbers we distinguish the case of a type B (when $(d_2,d_1)=(2,3)$) or a type A manifold. In the first case by Theorem \ref{teo:final} we just have to count how many pyramids appear; since $0<r<1$ and we have that $-6+r=[-6,-k-1,m_3,...,m_k]$, where $k$ is the size of the pyramid of the short manifold $M(-1;\frac{1}{2},\frac{1}{3},\frac{k+1}{6k+5})$, while the best upper approximation for the lowest twisting number is $\frac{k}{6k-1}$, we obtain that there are \[|m_3+1|\cdots|m_k+1|\hspace{0.5cm}\text{ pyramids of size }\hspace{0.5cm}k=|m_2+1|\:.\] In the second case by Theorem \ref{teo:final} we have to count the structure with the unique non-highest twisting number $-\underline q$. We computed in Proposition \ref{prop:Tsurgeries} that $-\underline q$ exists only when $0<r<d_1d_2-d_1-d_2$ and $-d_1d_2+r$ falls in the interval given by \[\left[[n_1,...,n_{h-1}],\:[n_1,...,n_h]=-\frac{d_1+d_2+sd_1d_2}{s+1}=-\frac{\underline q}{s+1}\right)\] for a positive $s$ when the fraction is reduced. Note that the left-most endpoint of the interval corresponds to the manifold without tails, while the right-most one to the best upper approximation (for each rational in the interval); hence, we have that the number we want is \[|m_h-n_h|\cdot|m_{h+1}+1|\cdots|m_k+1|\] as $m_i=n_i$ for $i=1,...,h-1$.

Finally, for $S^3_r(T_{d_2,-d_1})$ we just need to identify the standard graph: this has $e_0=-2$ (in the relevant cases), the first two legs as in the fibration of $T_{d_2,-d_1}$, and the third leg with framing $-\frac{d_1d_2+r}{d_1d_2+r-1}$. Everything that we said above can then be replicated.

\end{document}

%% file: TypeB.pdf_tex
\begingroup%
  \makeatletter%
  \providecommand\color[2][]{%
    \errmessage{(Inkscape) Color is used for the text in Inkscape, but the package 'color.sty' is not loaded}%
    \renewcommand\color[2][]{}%
  }%
  \providecommand\transparent[1]{%
    \errmessage{(Inkscape) Transparency is used (non-zero) for the text in Inkscape, but the package 'transparent.sty' is not loaded}%
    \renewcommand\transparent[1]{}%
  }%
  \providecommand\rotatebox[2]{#2}%
  \newcommand*\fsize{\dimexpr\f@size pt\relax}%
  \newcommand*\lineheight[1]{\fontsize{\fsize}{#1\fsize}\selectfont}%
  \ifx\svgwidth\undefined%
    \setlength{\unitlength}{1589.31284987bp}%
    \ifx\svgscale\undefined%
      \relax%
    \else%
      \setlength{\unitlength}{\unitlength * \real{\svgscale}}%
    \fi%
  \else%
    \setlength{\unitlength}{\svgwidth}%
  \fi%
  \global\let\svgwidth\undefined%
  \global\let\svgscale\undefined%
  \makeatother%
  \begin{picture}(1,0.50260007)%
    \lineheight{1}%
    \setlength\tabcolsep{0pt}%
    \put(0,0){\includegraphics[width=\unitlength,page=1]{TypeB.pdf}}%
    \put(0.4851949,0.02075381){\color[rgb]{0,0,0}\makebox(0,0)[lt]{\lineheight{1.25}\smash{\begin{tabular}[t]{l}$L$\end{tabular}}}}%
    \put(0.14004047,0.09584631){\color[rgb]{1,0,0}\makebox(0,0)[lt]{\lineheight{1.25}\smash{\begin{tabular}[t]{l}$A$\end{tabular}}}}%
    \put(0.95339263,0.13429734){\color[rgb]{1,0,0}\makebox(0,0)[lt]{\lineheight{1.25}\smash{\begin{tabular}[t]{l}$B$\end{tabular}}}}%
    \put(0.3562709,0.44235786){\color[rgb]{0,0,0}\makebox(0,0)[lt]{\lineheight{1.25}\smash{\begin{tabular}[t]{l}$(+1)$\end{tabular}}}}%
    \put(0.89106484,0.43837281){\color[rgb]{0,0,0}\makebox(0,0)[lt]{\lineheight{1.25}\smash{\begin{tabular}[t]{l}$(+1)$\end{tabular}}}}%
    \put(0,0){\includegraphics[width=\unitlength,page=2]{TypeB.pdf}}%
    \put(-0.0015498,0.25100748){\color[rgb]{0,0,0}\makebox(0,0)[lt]{\lineheight{1.25}\smash{\begin{tabular}[t]{l}$\Sigma$\end{tabular}}}}%
    \put(0.40422158,0.01668238){\color[rgb]{0,0,0}\makebox(0,0)[lt]{\lineheight{1.25}\smash{\begin{tabular}[t]{l}$\cdots$\end{tabular}}}}%
  \end{picture}%
\endgroup%